\documentclass[11pt, reqno]{amsart}
   \usepackage{amsmath,amssymb,amsthm}
   \usepackage{mathtools}
  \usepackage{array}
 2

\usepackage{tikz-cd}
\usepackage{graphicx}
\usepackage{xcolor}
\usepackage[colorlinks=true,linkcolor=blue,citecolor=blue,pdfpagelabels=false]{hyperref}
\usepackage{cleveref}
\usepackage{comment}
\newtheorem{theorem}{Theorem}[section]
\newtheorem{lemma}[theorem]{Lemma}
\newtheorem{prop}[theorem]{Proposition}
\newtheorem{cor}[theorem]{Corollary}

\theoremstyle{remark}
\newtheorem{rmk}{Remark}[section]

\renewcommand{\geq}{\geqslant}
\renewcommand{\leq}{\leqslant}

\begin{document}
% Remove page number
\thispagestyle{empty}

\title{REDUCTIONS OF GALOIS REPRESENTATIONS OF SLOPE 2}

\author{EKNATH GHATE \and SUDIPTA MAJUMDER  }

\begin{abstract}
Let $p$ be a prime $\ge 5$. We determine the
semisimplifications of the mod \(p\) reductions of irreducible two-dimensional
crystalline representations of the Galois group of ${\mathbb Q}_p$ of slope \(2\) and weights $k\ge3p+4$ with $k \not \equiv 6 \bmod (p-1)$.  The proof relies on the compatibility between the \(p\)-adic and the mod \(p\)
Local Langlands Correspondences for \(\mathrm{GL}_2(\mathbb{Q}_p)\).

\end{abstract}
\maketitle
\section{Introduction}

Let $p$ be an odd prime. This paper is devoted to computing the mod $p$ reductions of certain two-dimensional crystalline representations of the local Galois group $G_{\mathbb{Q}_p}$.

Following the work of Edixhoven~\cite{Edi92}, the first computations of reductions in positive slope were carried out by Breuil~\cite{Bre03b} for weights up to $2p+1$. The reductions are also known in the case where the slope is large relative to the weight, see~\cite{BLZ04, BL22}.

 In another direction, reductions have been determined for small fractional slopes: for slopes in $(0,1)$ by Buzzard--Gee~\cite{BG09, BG13}, for slopes in $(1,2)$ in~\cite{GG15,BG15, GR25}, and for slopes in $(2,3)$ partially by Nagel-Pande~\cite{NP25}. The case of slope $1$ was treated by Bhattacharya-Ghate-Rozensztajn~\cite{BGR18}, providing important insight into integral slope phenomena. Some of these papers assume $p \geq 5$.

In this paper, we study the reduction in the important missing case of integral slope $2$. The slope $2$ case is of particular interest due to the plethora of examples of cuspidal newforms whose associated Hecke eigenvalues have $p$-adic valuation equal to $2$. Let \(v\) denote the $p$-adic valuation on
\(\overline{\mathbb{Q}}_p\), normalized by \(v(p)=1\). For instance, the Ramanujan $\Delta$-function satisfies $v(\tau(3)) = 2$. Similarly, the modular form $
\Delta_{26} := \Delta \cdot E_{14}$, where $E_{14}$ denotes the Eisenstein series of weight $14$, provides an example for which $v(a_7) = 2$, with $a_7$ denoting the seventh Fourier coefficient. These examples demonstrate that slope $2$ naturally arises in significant arithmetic situations, thereby motivating a detailed study of the reductions in this case. It also seems natural to treat
the case of slope $2$ completely as a final test case before launching on a study of all integral slopes. 

We  introduce some notation used throughout the paper. Let \(E\) be a
finite extension of \(\mathbb{Q}_p\). Let \(a_p \in E\)
satisfy \(v(a_p) > 0\), and let \(k \ge 2\). Let \(V_{k,a_p}\) be the
irreducible crystalline representation of \(G_{\mathbb{Q}_p}\) with
Hodge--Tate weights \((0,k-1)\) and slope \(v(a_p)\), defined by the
condition that
\[
D_{\mathrm{cris}}(V_{k,a_p}^*) = D_{k,a_p},
\]
where \(D_{k,a_p} = E e_1 \oplus E e_2\) is the filtered \(\varphi\)-module
defined in \cite[\S2.3]{Ber11}. The importance
of the representation $V_{k,a_p}$ arises from the fact that
if $f$ is a normalized cuspidal newform of 
weight $k \geq 2$ and level prime to $p$ (and trivial character at $p$) with
$p$-th Fourier coefficient equal to $a_p$,
then the associated $p$-adic
Galois representation
$\rho_f$ satisfied
$\rho_f |_{G_{{\mathbb Q}_p}} \cong V_{k,a_p}$.

The semisimplification \(\overline{V}^{\,{ss}}_{k,a_p}\) of the
reduction \(\overline{V}_{k,a_p}\) %taken with respect to a lattice in
is independent of the choice of lattice in \(V_{k,a_p}\).  Let
\(\omega = \omega_1\) and \(\omega_2\) denote the fundamental characters of the inertia subgroup $I_{{\mathbb{Q}_p}}$ of
levels \(1\) and \(2\), respectively. Note that
$\omega$ has a canonical extension to ${\mathbb{Q}_p}$. For \(a \in \mathbb{Z}\), let $\omega_2^a$ continue to denote a specific extension of the same character to $G_{\mathbb{Q}_{p^2}}$. Let 
\(\operatorname{ind}(\omega_2^a)\) denote the representation of
\(G_{\mathbb{Q}_p}\) obtained by inducing the extended character 
%character \(\omega_2^a\) 
from
\(G_{\mathbb{Q}_{p^2}}\) to \(G_{\mathbb{Q}_p}\); this representation is irreducible if $p+1 \nmid a$. The extension of $\omega_2^a$ from $I_{{\mathbb Q}_p}$ to 
\(G_{\mathbb{Q}_{p^2}}\) 
is chosen so that  \(\det(\operatorname{ind}(\omega_2^a))=\omega^a\). Finally, let \(\mu_x\) denote the unramified character of \(G_{\mathbb{Q}_p}\) sending geometric Frobenius at \(p\) to
\(x \in \overline{\mathbb{F}}_p^\times\).

With this notation, we have the following theorem for the reduction ${\overline{V}^{ss}_{k,a_p}}$ when the slope $v(a_p)$ is equal to $2$, for all prime $p>3$. We defer the treatment of one tricky
congruence class of weights mod $(p-1)$ since a complete treatment would require a new paper.

\begin{theorem}\label{main theorem}
    Let $p>3$, let $k\geq 3p+4 $ and let $r=k-2 \equiv a \bmod(p-1)$, with $1\leq a \leq p-1$ and $a \neq4$. Suppose that the slope $v(a_p)=2$. Then the semi-simpliflication ${\overline{V}^{ss}_{k,a_p}}$ of the reduction of $V_{k,a_p}$ is as follows:

\medskip
\medskip
\noindent\textbf{Case $a = 1$:}
\[
\overline{V}_{k,a_p}^{\,ss} \cong
\begin{cases}
\mu_\lambda \cdot \omega^{p-1}
\;\oplus\;
\mu_{\lambda^{-1}} \cdot \omega^2,
& \text{if } p \nmid r(r+1),
\,\,\,\text{with} \,\, 
\displaystyle
\lambda = 
\overline{{\frac{2}{r(r+1)}}.\frac{a_p}{p^2}} \in \overline{\mathbb{F}}_p^\times,\\[1em]
\mathrm{ind}\!\left(\omega_2^{2p}\right),
& \text{if } p \mid r+1, \\[1em]
\mathrm{ind}\!\left(\omega_2^{2p}\right), 
& \text{if } p\mid r-p \,\,  \text{and}\,\,  p^2 \nmid r-p,\\[1em]
\mu_\lambda \cdot \omega
\;\oplus\;
\mu_{\lambda^{-1}} \cdot \omega,
& \text{if } p^2 \mid r-p,\,\,\,\text{with }
\displaystyle
\lambda + \frac{1}{\lambda}
=\overline{\frac{p-r}{a_p}}
\in \overline{\mathbb{F}}_p.
\end{cases}
\]

\medskip
\noindent\textbf{Case $a = 2$:}
\[
\overline{V}_{k,a_p}^{\,ss} \cong
\begin{cases}
\mathrm{ind}\!\left(\omega_2^{3p}\right),
& \text{if } p \nmid r(r-1),\\[1em]
\mu_\lambda \cdot \omega
\;\oplus\; 
\mu_{\lambda^{-1}} \cdot \omega^2,
& \text{if } p \mid r \,\, and \,\,p^2\nmid r-2p,
\,\,\,\text{with} \,\, 
\displaystyle
\lambda = 
\overline{\frac{2p}{r-2p}\cdot\frac{a_p}{p^2}} \in \overline{\mathbb{F}}_p^\times,\\[1em]
\mathrm{ind}\!\left(\omega_2^{2p+1}\right),
& \text{if } p^2 \mid r-2p,\\[1em]
\mu_\lambda \cdot \omega
\;\oplus\;
\mu_{\lambda^{-1}} \cdot \omega^2,
& \text{if } p \mid r-1 \,\, and \,\,p^2 \nmid r-p-1,
\,\,\,\text{with} \,\, 
\displaystyle
\lambda = 
\overline{\frac{p}{p+1-r}\cdot \frac{a_p}{p^2}} \in \overline{\mathbb{F}}_p^\times,
\\[1em]
\mathrm{ind}\!\left(\omega_2^{2p+1}\right), 
& \text{if }  p^2\mid r-p-1.
\end{cases}
\]
\medskip

\noindent\textbf{Case $a = 3$:}
\[
\overline{V}_{k,a_p}^{\,ss} \cong
\begin{cases}
\mu_\lambda \cdot \omega^{p}
\;\oplus\;
\mu_{{\lambda}^{-1}} \cdot \omega^3,
& \text{if } p \nmid \binom{r}{4},
\,\,\,\text{with} \,\, 
\displaystyle
\lambda = 
\overline{\frac{r-3}{3\binom{r}{3}}\cdot\frac{a_p}{p^2}} \in \overline{\mathbb{F}}_p^\times,\\[1em]

\mathrm{ind}\!\left(\omega_2^{p+3}\right),
& \text{if } p \mid r,\\[1em]
\mu_\lambda \cdot \omega^{2}
\;\oplus\; 
\mu_{\lambda^{-1}} \cdot \omega^2,
& \text{if } p \mid r-1,
\,\,\,\text{with} \,\, 
\displaystyle
\lambda +\lambda^{-1}= 
\overline{\frac{a_p^2-p^3(\frac{r-2p-1}{2})}{p^2a_p}} \in \overline{\mathbb{F}}_p,\\[1em]
\mathrm{ind}\!\left(\omega_2^{3p+1}\right),
& \text{if } p \mid r-2,\\[1em]
\mathrm{ind}\!\left(\omega_2^{4p}\right),
& \text{if } p \mid r-3.
\end{cases}
\]

\noindent\textbf{Case $5 \le a \le p-1$:}
\[
\overline{V}_{k,a_p}^{\,ss} \cong
\begin{cases}
\mu_\lambda \cdot \omega^{a-1}
\;\oplus\;
\mu_{\lambda^{-1}} \cdot \omega^2,
& \text{if } p \nmid (r-a)(r-a+1),\,\,\text{with }
\displaystyle
\lambda =
\overline{\frac{(a-1)(a-2)}{(r-a)(r-a+1)}\cdot \frac{a_p}{p^2}}
\in \overline{\mathbb{F}}_p^{\times}, \\[1em]
\mathrm{ind}\!\left(\omega_2^{{a+p}}\right),
& \text{if } p \mid r-a+1,\\[1em]
\mu_\lambda \cdot \omega^{a}
\;\oplus\;
\mu_{\lambda^{-1}} \cdot \omega, 
& \text{if } p \mid r-a\,\,  \text{and}\,\,  p^2 \nmid (r-a),\,\, \text{with}\,\,
 \displaystyle 
 \lambda = 
\overline{ \frac{a  p}{a-r} \cdot \frac{a_p}{p^2}}\in \overline{\mathbb{F}}_p^{\times}, \\[1em]
 \mathrm{ind} \left(\omega_2^{a+1}\right),
 & \text{if } p^2 \mid r-a.
\end{cases}
\]

\end{theorem}

\medskip

For instance, returning to $\Delta_{26}$ and $p = 7$,
we have that $r= 24$, $a= 6$ and
 $7 \nmid 18 \cdot 19$, so by the first part of the last case of the theorem, we have 
$\overline{V}_{k,a_p}^{\,ss} \cong \omega^2 \oplus \omega^5$, as it should be since 
it is well-known that even the global mod $7$ Galois representation attached to $\Delta_{26}$ satisfies $\bar\rho^{\,ss}_{\Delta_{26}}  \cong \omega^5 \oplus \omega^2$
(see Example~\ref{example Delta 26}).

The theorem is proved using the usual compatibility between the $p$-adic and mod $p$ local Langlands correspondences with respect to the process of reduction. 
Indeed, the compatibility implies that it suffices to compute the reduction $\overline{\Theta}_{k,a_p}$ of the standard lattice $\Theta_{k,a_p}$
in the corresponding locally algebraic representation of $\mathrm{GL}_2(\mathbb{Q}_p)$.
Now, $\overline{\Theta}_{k,a_p}$ is uniformized by the induction of 
the symmetric $r$-th power ($r = k-2$)
of the standard mod $p$ representation of 
$\mathrm{GL}_2(\mathbb{F}_p)$, which is equipped with the so called theta filtration. As such,  $\overline{\Theta}_{k,a_p}$ is built out of six subquotients $F_0, F_1, \ldots, F_5$. 

The goal is to show that all but one or two of these subquotients survive, and then to use the mod $p$ local Langlands correspondence to read of ${\overline{V}^{ss}_{k,a_p}}$.
In the former case ${\overline{V}^{ss}_{k,a_p}}$ is irreducible, and in the latter case it is reducible, except if the underlying weight involved is of dimension $p-1$
in which case in the former case a reducible possibility also arises (this subtlety is dealt with completely in this paper).

Some of the subquotients $F_i$ can be killed immediately using results going back to \cite{GV22}. However, most of the surviving ones must be killed using delicate computations with the Hecke operator $T$.
These are carried out in Theorems~\ref{thm:3.1}, \ref{thm: con 1}, \ref{thm:a=2} and \ref{thm:a=3} 
in Sections~\ref{sec: 5 to p-1}, \ref{sec: a=1}, \ref{sec: a=2} and \ref{sec: a=3}, respectively. These results correspond to the four cases in the theorem (the last case in the theorem is the easiest to treat, so is treated first). 
Some preliminary material is covered in Section~\ref{sec: basics}.

While the reducibility of 
${\overline{V}^{ss}_{k,a_p}}$ 
%reductions of mod $p$ Galois representations 
has traditionally been obtained for small slopes by showing that two adjacent subquotients $F_i$ and $F_{i+1}$ survive,
%contributing Jordan–Hölder factors,
the slope $2$ setting displays more complicated behavior. For instance, we show that for weights $k$ for which $r = k-2
$ belongs to the  congruence class $a = 3 \pmod{p-1}$ treated in this paper, the proof unexpectedly involves contributions 
%from
%$F_2$ and $F_4$, and even more surprisingly, 
from $F_1$ and $F_5$. This phenomenon highlights a qualitative shift in the complexity of the computation of the reduction of crystalline representations as the slope increases to 2.

\section{Basics}
\label{sec: basics}
We begin by recalling some notation and some standard results that will be used throughout; see \cite{Bre03b} and \cite{BG15} for further details.
\subsection{Hecke operator $T$}
Let $G=\mathrm{GL}_2(\mathbb{Q}_p)$ and $K=\mathrm{GL}_2(\mathbb{Z}_p)$ be the maximal compact open
subgroup of $G$, and let $Z=\mathbb{Q}_p^{\times}$ be the center of $G$.
Let $R$ be a $\mathbb{Z}_p$-algebra and let
\[
V=\mathrm{Sym}^r R^2 \otimes D^s
\]
be the usual symmetric power representation of $KZ$ twisted by a power of the determinant character
$D$, modeled on homogeneous polynomials of degree $r$ in the variables $X,Y$ over $R$.

We write $\mathrm{ind}_{KZ}^G$ for compact induction.
Thus $\mathrm{ind}_{KZ}^G V$ consists of functions $f:G\to V$ such that
\[
f(hg)=h\cdot f(g)\quad \text{for all } h\in KZ,\ g\in G,
\]
and such that $f$ is compactly supported modulo $KZ$.

For $g\in G$ and $v\in V$, let $[g,v]\in \mathrm{ind}_{KZ}^G V$ denote the function supported on
$KZg^{-1}$ defined by
\[
g'\longmapsto
\begin{cases}
g'g\cdot v, & \text{if } g'\in KZg^{-1},\\
0, & \text{otherwise}.
\end{cases}
\]
Any element of $\mathrm{ind}_{KZ}^G V$ is a finite $R$-linear combination of functions of the form
$[g,v]$ with $g\in G$ and $v\in V$.

The Hecke operator $T$ is defined by its action on these functions as follows:
\begin{equation}\label{eq:hecke}
T([g,v(X,Y)]) =
\sum_{\lambda\in\mathbb{F}_p}
\left[
g
\begin{pmatrix}
p & [\lambda]\\
0 & 1
\end{pmatrix},
\, v(X,-[\lambda]X+pY)
\right]
+
\left[
g
\begin{pmatrix}
1 & 0\\
0 & p
\end{pmatrix},
\, v(pX,Y)
\right],
\end{equation}
where $[\lambda]$ denotes the Teichm\"uller representative of $\lambda\in\mathbb{F}_p$.

For $m=0$, we set $I_0=\{0\}$, and for $m>0$ let
\[
I_m=\left\{
[\lambda_0]+[\lambda_1]p+\cdots+[\lambda_{m-1}]p^{m-1}
:\lambda_i\in\mathbb{F}_p
\right\}\subset\mathbb{Z}_p.
\]
For $m\ge1$, define the truncation map
\[
[\cdot]_{m-1}:I_m\to I_{m-1}
\]
by retaining the first $m-1$ terms in the above $p$-adic expansion; for $m=1$ this is the zero map.

Let
\[
\alpha=
\begin{pmatrix}
1 & 0\\
0 & p
\end{pmatrix}.
\]
For $m\ge0$ and $\lambda\in I_m$, define
\[
g^0_{m,\lambda}=
\begin{pmatrix}
p^m & \lambda\\
0 & 1
\end{pmatrix},
\qquad
g^1_{m,\lambda}=
\begin{pmatrix}
1 & 0\\
p\lambda & p^{m+1}
\end{pmatrix}.
\]
Note that $g^0_{0,0}=\mathrm{Id}$ and $g^1_{0,0}=\alpha$.

Recall the decomposition
\[
G=\bigsqcup_{\substack{m\ge0,\ \lambda\in I_m,\\ i\in\{0,1\}}}
KZ\,(g^i_{m,\lambda})^{-1}.
\]
Thus any element of $\mathrm{ind}_{KZ}^G V$ can be written as a finite sum of functions of the form
$[g,v]$, where $g=g^0_{m,\lambda}$ or $g=g^1_{m,\lambda}$ for some $\lambda\in I_m$ and $v\in V$.

Let
\[
v=\sum_{i=0}^r c_i X^{r-i}Y^i \in V=\mathrm{Sym}^r R^2\otimes D^s.
\]
Expanding \eqref{eq:hecke}, we may write $T=T^+ + T^-$, where
\[
T^+\bigl([g^0_{n,\mu},v]\bigr)
=
\sum_{\lambda\in I_1}
\left[
g^0_{n+1,\mu+p^n\lambda},
\,
\sum_{j=0}^r
\left(
\sum_{i=j}^r
p^j c_i \binom{i}{j}(-\lambda)^{i-j}
\right)
X^{r-j}Y^j
\right],
\]
\[
T^-\bigl([g^0_{n,\mu},v]\bigr)
=
\left[
g^0_{n-1,[\mu]_{n-1}},
\,
\sum_{j=0}^r
\left(
\sum_{i=j}^r
p^{r-i} c_i \binom{i}{j}
\left(\frac{\mu-[\mu]_{n-1}}{p^{n-1}}\right)^{i-j}
\right)
X^{r-j}Y^j
\right],
\qquad (n>0),
\]
and
\[
T^-\bigl([g^0_{0,\mu},v]\bigr)
=
\left[
\alpha,
\,
\sum_{j=0}^r p^{r-j} c_j X^{r-j}Y^j
\right].
\]
We will use these explicit formulas for $T^+$ and $T^-$ to compute $(T-a_p)f$ for
$f\in \mathrm{ind}_{KZ}^G \mathrm{Sym}^r \overline{\mathbb{Q}}_p^{\,2}$.
\subsection{The mod $p$ Local Langlands Correspondence}
For $0 \le r \le p-1$, $\lambda \in \overline{\mathbb{F}}_p$, and a smooth character
\[
\eta : \mathbb{Q}_p^\times \to \overline{\mathbb{F}}_p^{\times},
\]
let
\[
\pi(r,\lambda,\eta)
:=
\frac{\mathrm{ind}_{KZ}^{G} \mathrm{Sym}^r \overline{\mathbb{F}}_p^{\,2}}{(T-\lambda)}
\;\otimes\; (\eta \circ \det)
\]
be the smooth admissible representation of $G$.

This representation is known to be irreducible unless
\[
(r,\lambda) = (0,\pm 1) \quad \text{or} \quad (p-1,\pm 1),
\]
by the classification of irreducible representations of $G$ in characteristic $p$
\cite{BL94,BL95,Bre03a}.

With this notation, Breuil’s semisimple mod $p$ Local Langlands Correspondence
\cite[Def.~1.1]{Bre03b} is given by:
\begin{itemize}
\item $\lambda = 0$:
\[
\mathrm{ind}(\omega_2^{r+1}) \otimes \eta
\;\overset{LL}{\longmapsto}\;
\pi(r,0,\eta),
\]

\item $\lambda \neq 0$:
\[
\left(\mu_\lambda \cdot \omega^{r+1} \oplus \mu_{\lambda^{-1}}\right) \otimes \eta
\;\overset{LL}{\longmapsto}\,\,
\pi(r,\lambda,\eta)^{{ss}}
\oplus
\pi([p-3-r],\lambda^{-1},\eta\omega^{r+1})^{{ss}},
\]
\end{itemize}
where
\[
[p-3-r] \in \{0,1,\dots,p-2\}
\quad\text{is defined by}\quad
[p-3-r] \equiv p-3-r \pmod{p-1}.
\]

Now consider the locally algebraic representation of $G$ given by
\[
\Pi_{k,a_p}
=
\frac{\mathrm{ind}_{KZ}^{G} \mathrm{Sym}^r \overline{\mathbb{Q}}_p^{\,2}}{(T-a_p)},
\]
where $r = k-2 \ge 0$ and $T$ denotes the Hecke operator.

Let $\Theta_{k,a_p}$ be the standard lattice in $\Pi_{k,a_p}$, defined by
\begin{equation}\label{eq:lattice}
\Theta := \Theta_{k,a_p}
:=
\mathrm{image}\!\left(
\mathrm{ind}_G^{KZ} \mathrm{Sym}^r \overline{\mathbb{Z}}_p^{\,2}
\longrightarrow
\Pi_{k,a_p}
\right).
\end{equation}
Equivalently,
\[
\Theta
\simeq
\frac{\mathrm{ind}_{KZ}^{G} \mathrm{Sym}^r \overline{\mathbb{Z}}_p^{\,2}}
{(T-a_p)\bigl(\mathrm{ind}_{KZ}^{G} \mathrm{Sym}^r \overline{\mathbb{Q}}_p^{\,2}\bigr)
\cap \mathrm{ind}_{KZ}^{G} \mathrm{Sym}^r \overline{\mathbb{Z}}_p^{\,2}}.
\]

It is known that the semisimplification of the reduction of this lattice satisfies
\[
\overline{\Theta}^{\,{ss}}_{k,a_p}
\simeq {LL}\!\left(\overline{V}^{\,{ss}}_{k,a_p}\right),
\]
where $\mathrm{LL}$ denotes the (semisimple) mod $p$ Local Langlands Correspondence
\cite{Ber10}.

Since the map ${LL}$ is injective, it suffices to determine
\[
{LL}\!\left(\overline{V}^{\,{ss}}_{k,a_p}\right)
\]
in order to determine $\overline{V}^{\,{ss}}_{k,a_p}$.
\subsection{Useful lemmas}\label{roots of 1}
We begin by proving some combinatorial results which will be useful in computing the mod $p$ reductions.

The following lemma is an analog of \cite[Lemma $2.5$]{BG15} and
\cite[Proposition $2.8$]{BGR18}.
\begin{lemma}\label{lem:2.1}
Let \(p > 2\), \(i = 0,1,\dots,p-1\) and
$p \le r \equiv a \pmod{p-1}$. If
$a \in \{i+1,\dots,p-1+i\},$
then we have
\[
\sum_{\substack{j \equiv a \!\!\!\pmod{p-1} \\ i<j<r}}
\binom{j}{i}\binom{r}{j}
\equiv
p \binom{r}{i}\frac{a-r}{a-i}
\pmod{p^2}.
\]
\end{lemma}

\begin{proof}
We know that      \begin{eqnarray*}
        (p-1) \sum_{\substack{
    j \geq 0 \\
    j \equiv a \pmod{p-1}
}}
\binom{j}{i}\binom{r}{j} = \binom{r}{i} \sum_{\xi \in \mu'} \xi^{i-a} (1+\xi)^{r-i},
    \end{eqnarray*}
    where $\mu'$ is the set of $(p-1)$-th roots of unity except $-1$.
   Now,
$
(1+\xi)^{p-1}=1+pz_\xi,
$
and therefore, by the binomial theorem,
$
(1+\xi)^{k(p-1)}
=(1+pz_\xi)^k
\equiv 1+pkz_\xi \pmod{p^2}
$. Hence, we have \begin{eqnarray*}
        (p-1) \sum_{\substack{
    j \geq 0 \\
    j \equiv a \pmod{p-1}
}}
\binom{j}{i}\binom{r}{j} &=& \binom{r}{i} \sum_{\xi \in \mu'} \xi^{i-a} (1+\xi)^{a-i}(1+\xi)^{r-a}\\
&\equiv& \binom{r}{i} \sum_{\xi \in \mu'} \xi^{i-a} (1+\xi)^{a-i}\biggl(1+p \frac{r-a}{p-1} z_{\xi}\biggl) \pmod{p^2}\\
&\equiv& \binom{r}{i} \sum_{\xi \in \mu'} \xi^{i-a} (1+\xi)^{a-i}\biggl[1+\frac{r-a}{p-1} \biggl((1+\xi)^{p-1}-1\biggl)\biggl] \pmod{p^2}\\
&\equiv& \binom{r}{i} \sum_{\xi \in \mu'} \xi^{i-a} (1+\xi)^{a-i}+ \frac{r-a}{p-1}\binom{r}{i} \sum_{\xi \in \mu'} \xi^{i-a} (1+\xi)^{a-i+p-1}\\
 & & \quad - \frac{r-a}{p-1}\binom{r}{i} \sum_{\xi \in \mu'} \xi^{i-a} (1+\xi)^{a-i} \pmod{p^2}\\
&\equiv& \binom{r}{i} \biggl[(-1)^{i-a}+\sum_{\xi \in \mu'} \xi^{i-a}\biggl]+(p-1)\binom{r}{i}\\
 && \quad + \frac{r-a}{p-1}\binom{r}{i}(p-1)\binom{a-i+p-1}{a-i}\pmod{p^2}\\
&\equiv& \binom{r}{i} \biggl[(-1)^{i-a}+\sum_{\xi \in \mu'} \xi^{i-a}\biggl]+(p-1)\binom{r}{i}\\
&& \quad + (r-a)\binom{r}{i}\frac{(p+a-i-1)(p+a-i-2) \cdots (p+1)p}{(a-i)!}\pmod{p^2}\\
&\equiv& \binom{r}{i} \biggl[(-1)^{i-a}+\sum_{\xi \in \mu'} \xi^{i-a}\biggl]+(p-1)\binom{r}{i}\\
&& \quad + p(r-a)\binom{r}{i}\frac{(a-i-1)(a-i-2) \cdots 1}{(a-i)!}\pmod{p^2}\\
&\equiv& \binom{r}{i} \biggl[(-1)^{i-a}+\sum_{\xi \in \mu'} \xi^{i-a}\biggl]+(p-1)\binom{r}{i}+p\binom{r}{i}\frac{r-a}{a-i}\pmod{p^2}.
    \end{eqnarray*}
Now, if $a-i =p-1$, then $\biggl[(-1)^{i-a}+\sum_{\xi \in \mu'} \xi^{i-a}\biggl]=p-1$ and if $a-i<p-1$, then $\biggl[(-1)^{i-a}+\sum_{\xi \in \mu'} \xi^{i-a}\biggl]=0$. Hence, we have \begin{eqnarray*}
     (p-1) \sum_{\substack{
   i< j<r  \\
    j \equiv a \pmod{p-1}
}}
\binom{j}{i}\binom{r}{j} \equiv p\binom{r}{i}\frac{r-a}{a-i} \pmod{p^2}.
\end{eqnarray*}
Dividing both side by $p-1$ proves the lemma.
    \end{proof}
    
    \begin{rmk}\label{rmk:2.1}
        Let \(p > 2\) and
$p \le r \equiv a \pmod{p(p-1)}$. Then, taking $i = 0$ in the proof above shows that
\[
\sum_{\substack{j \equiv a \!\!\!\pmod{p-1} \\ 0<j<r}}
\binom{r}{j}
\equiv
 \frac{(a-r)p}{a}
\pmod{p^3}.
\]
%As we shall see later, this formula also holds if $a=p$
%(since the congruence \eqref{} also holds
%mod $p^3$).
\end{rmk}
    \begin{lemma}\label{lem:2.2}
    Let \(p > 2\). Let \(r \equiv a \pmod{p-1}\) with
\(a \in \{3,\dots,p-1\}\). Then there exist integers
\[
\{\alpha_j : a \le j < r \text{ and } j \equiv a \pmod{p-1}\}
\]
such that the following conditions are satisfied:

\begin{enumerate}
    \item for all $j$,
    \[
    \alpha_j \equiv \binom{r}{j} \pmod p,
    \]
    \item for $n = 0,1,2$, 
\[
\sum_{j \geq n} \binom{j}{n}\alpha_j \equiv 0 \pmod{p^{4-n}},
\]
\item  for $n = 3$, 
\begin{itemize}
\item if $a = 3$, then
    \[
    \sum_{j \geq 3} \binom{j}{3}\alpha_j
    \equiv \binom{r}{3} \pmod{p},
    \]
    \item and if $a = 4,\ldots,p-1$, then
    \[
    \sum_{j \geq 3} \binom{j}{3}\alpha_j \equiv 0 \pmod{p}.
    \]
    \end{itemize}
\end{enumerate}
\end{lemma}

\begin{proof}
    We know that \begin{eqnarray*}
        (p-1) \sum_{\substack{
    j \geq 0 \\
    j \equiv a \pmod{p-1}
}}
\binom{j}{m}\binom{r}{j} = \binom{r}{m} \sum_{\xi \in \mu'} \xi^{m-a} (1+\xi)^{r-m},
    \end{eqnarray*}
    where $\mu'$ is the set of $(p-1)$-th roots of unity except $-1$.
    \begin{enumerate}
        \item [(i)] For $m=0$, we have \begin{eqnarray*}
        (p-1) \sum_{\substack{
    j \geq 0 \\
    j \equiv a \pmod{p-1}
}}
\binom{r}{j}    &\equiv& \sum_{\xi \in \mu'} \xi^{-a} (1+\xi)^a \pmod{p}\\
&\equiv&  (1+\delta_{a,p-1})(p-1) \pmod{p},  \text{ so} \\
(p-1) \sum_{\substack{
    a \le j <r \\
    j \equiv a \pmod{p-1}
}}
\binom{r}{j}  &\equiv& 0 \pmod{p}.
\end{eqnarray*}
\item[(ii)] For $m=1$, we have \begin{eqnarray*}
    (p-1) \sum_{\substack{
    j \geq 0 \\
    j \equiv a \pmod{p-1}
}}
j\binom{r}{j}    &\equiv& r\sum_{\xi \in \mu'} \xi^{1-a} (1+\xi)^{a-1} \pmod{p} \\
&\equiv& (p-1)r \pmod{p},\,\,\text{ so}\\
(p-1) \sum_{\substack{
a\le j<r \\
    j \equiv a \pmod{p-1}
}}
j\binom{r}{j}    &\equiv& 0\pmod{p}.
\end{eqnarray*}
\item[(iii)] For $m=2$, we have \begin{eqnarray*}
(p-1) \sum_{\substack{
    j \ge 0 \\
    j \equiv a \pmod{p-1}
}}
\binom{j}{2}\binom{r}{j}    &\equiv& \binom{r}{2}\sum_{\xi \in \mu'} \xi^{2-a} (1+\xi)^{a-2} \pmod{p}\\
&\equiv& (p-1)\binom{r}{2}
\pmod{p},\,\,\text{ so}\\
(p-1) \sum_{\substack{
    a\le j<r \\
    j \equiv a \pmod{p-1}
}}
\binom{j}{2}\binom{r}{j}    &\equiv& 0 \pmod{p}.
\end{eqnarray*}
    \end{enumerate}
    Now, put $\alpha_j =\binom{r}{j}+p\delta_j$ for $a \le j<r$ such that $j \equiv a \pmod{p-1}$ and define  \begin{eqnarray*}
        s_0 &=& -p^{-1} \sum_{\substack{
    a \le j <r \\
    j \equiv a \pmod{p-1}
}}
\binom{r}{j},\\
s_1 &=& -p^{-1}\sum_{\substack{
    a \le j <r\\
    j \equiv a \pmod{p-1}
}}
j\binom{r}{j},\\
s_2 &=& -p^{-1}\sum_{\substack{
    a \le j <r \\
    j \equiv a \pmod{p-1}
}}
\binom{j}{2}\binom{r}{j}.
\end{eqnarray*}
Recall the determinant of \[
\begin{pmatrix}
1 & 1 & 1 \\
a& b & c \\
\binom{a}{2} & \binom{b}{2} & \binom{c}{2}
\end{pmatrix}
\]
is $\frac{(a-b)(b-c)(c-a)}{2} = {(p-1)^3}$ if $a$, $b$, $c$ are in AP with common difference $p-1$.  Thus the following system of linear equations 
\begin{eqnarray*}
    \delta_a+\delta_{a+p-1}+\delta_{a+2p-2} &\equiv& s_0 \pmod{p^3}\\
   a \delta_a+(a+p-1)\delta_{a+p-1}+(a+2p-2)\delta_{a+2p-2} &\equiv& s_1 \pmod{p^2}\\
    \binom{a}{2}\delta_a+\binom{a+p-1}{2}\delta_{a+p-1}+\binom{a+2p-2}{2}\delta_{a+2p-2} &\equiv& s_2 \pmod{p}
\end{eqnarray*}
is solvable for $\delta_a$, $\delta_{a+p-1}$, $\delta_{a+2p-2}$. We take the other $\delta_j=0$.
It follows that 
\begin{eqnarray*}
      \sum_{\substack{
    a\le j <r \\
    j \equiv a \pmod{p-1}
}} \alpha_j &\equiv& 0 \pmod{p^4} ,  \\
\sum_{\substack{
    a \le j <r \\
    j \equiv a \pmod{p-1}
}} j\alpha_j &\equiv& 0 \pmod{p^3}, \\
\sum_{\substack{
    a \le j <r \\
    j \equiv a \pmod{p-1}
}} \binom{j}{2}\alpha_j &\equiv& 0 \pmod{p^2}.
     \end{eqnarray*}
     This proves part $(1)$ and part $(2)$.

Now, by part $(1)$ and Lemma~\ref{lem:2.1} , for $a=3$ we have \begin{eqnarray*}
    \sum_{j \geq 3} \binom{j}{3}\alpha_j &\equiv& \sum_{\substack{
    3\le j<r \\
    j \equiv 3  \pmod{p-1}
}}
\binom{j}{3}\binom{r}{j}   \pmod{p}\\
& \equiv& \binom{r}{3} \pmod{p},
    \end{eqnarray*}
    and for $a \neq 3$ \begin{eqnarray*}
    \sum_{j \geq 3} \binom{j}{3}\alpha_j &\equiv& \sum_{\substack{
    a\le j<r \\
    j \equiv 3  \pmod{p-1}
}}
\binom{j}{3}\binom{r}{j}   \pmod{p}\\
& \equiv& 0 \pmod{p},
    \end{eqnarray*}
    proving part $(3)$.
\end{proof}

\begin{lemma}\label{lem:2.3}
Let \(p > 2\) and
$r \equiv a \pmod{p(p-1)}$
with \(a \in \{5,\dots,p-1\}\). Then there exist integers
\[
\{\alpha_j : a \le j < r \text{ and } j \equiv a \pmod{p-1}\}
\]
such that the following conditions are satisfied:

\begin{enumerate}
    \item for all $j$,
    \[
    \alpha_j \equiv \binom{r}{j} \pmod{p^2},
    \]
    
    \item for \(n=0,1,2,3\), 
    \[
    \sum_{j \ge n} \binom{j}{n}\alpha_j
    \equiv 0 \pmod{p^{5-n}}.
    \]

\end{enumerate}
\end{lemma}

\begin{proof}
    We know that \begin{eqnarray*}
        (p-1) \sum_{\substack{
    j \geq 0 \\
    j \equiv a \pmod{p-1}
}}
\binom{j}{m}\binom{r}{j} =\binom{r}{m} \sum_{\xi \in \mu'} \xi^{m-a} (1+\xi)^{r-m}.
    \end{eqnarray*}
    \begin{enumerate}
        \item [(i)] For $m=0$, we have \begin{eqnarray*}
        (p-1) \sum_{\substack{
    j \geq 0 \\
    j \equiv a \pmod{p-1}
}}
\binom{r}{j}    &\equiv& \sum_{\xi \in \mu'} \xi^{-a} (1+\xi)^a \pmod{p^2}\\
&\equiv&  (p-1)(1+\delta_{a,p-1})\pmod{p^2},\text{ so} \\
(p-1) \sum_{\substack{
    a \le j <r \\
    j \equiv a \pmod{p-1}
}}
\binom{r}{j}  &\equiv& 0 \pmod{p^2}.
\end{eqnarray*}
\item[(ii)] For $m=1$, we have \begin{eqnarray*}
    (p-1) \sum_{\substack{
    j \geq 0 \\
    j \equiv a \pmod{p-1}
}}
j\binom{r}{j}    &\equiv& r\sum_{\xi \in \mu'} \xi^{1-a} (1+\xi)^{a-1} \pmod{p^2} \\
&\equiv& (p-1)r \pmod{p^2}, \text{ so}\\
(p-1) \sum_{\substack{
a\le j<r \\
    j \equiv a \pmod{p-1}
}}
j\binom{r}{j}    &\equiv& 0\pmod{p^2}.
\end{eqnarray*}
\item[(iii)] For $m=2$, we have \begin{eqnarray*}
    (p-1) \sum_{\substack{
    j \ge 0 \\
    j \equiv a \pmod{p-1}
}}
\binom{j}{2}\binom{r}{j}    &\equiv& \binom{r}{2}\sum_{\xi \in \mu'} \xi^{2-a} (1+\xi)^{a-2} \pmod{p^2} \\
&\equiv& (p-1)\binom{r}{2} \pmod{p^2}, \text{ so}\\
 (p-1) \sum_{\substack{
    a\le j<r \\
    j \equiv a \pmod{p-1}
}}
\binom{j}{2}\binom{r}{j} &\equiv& 0 \pmod{p^2}.
\end{eqnarray*}
    \end{enumerate}
     Now put $\alpha_j =\binom{r}{j}+p^2\delta_j$ for $j \equiv a \pmod{p-1}$, $a \le j<r$ and define  \begin{eqnarray*}
        s_0 &=& -p^{-2} \sum_{\substack{
    a \le j <r \\
    j \equiv a \pmod{p-1}
}}
\binom{r}{j},\\
s_1 &=& -p^{-2}\sum_{\substack{
    a \le j <r\\
    j \equiv a \pmod{p-1}
}}
j\binom{r}{j},\\
s_2 &=& -p^{-2}\sum_{\substack{
    a \le j <r \\
    j \equiv a \pmod{p-1}
}}
\binom{j}{2}\binom{r}{j}.
\end{eqnarray*}
The following system of linear equations 
\begin{eqnarray*}
    \delta_a+\delta_{a+p-1}+\delta_{a+2p-2} &\equiv& s_0 \pmod{p^3}\\
   a \delta_a+(a+p-1)\delta_{a+p-1}+(a+2p-2)\delta_{a+2p-2} &\equiv& s_1 \pmod{p^2}\\
    \binom{a}{2}\delta_a+\binom{a+p-1}{2}\delta_{a+p-1}+\binom{a+2p-2}{2}\delta_{a+2p-2} &\equiv& s_2 \pmod{p}
\end{eqnarray*}
is solvable in the three unknowns
$\delta_a$, $\delta_{a+p-1}$, $\delta_{a+2p-2}$.
We take the other $\delta_j=0$. Then
\begin{eqnarray*}
      \sum_{\substack{
    a\le j <r \\
    j \equiv a \pmod{p-1}
}} \alpha_j &\equiv& 0 \pmod{p^5} ,  \\
\sum_{\substack{
    a \le j <r \\
    j \equiv a \pmod{p-1}
}} j\alpha_j &\equiv& 0 \pmod{p^4}, \\
\sum_{\substack{
    a \le j <r \\
    j \equiv a \pmod{p-1}
}} \binom{j}{2}\alpha_j &\equiv& 0 \pmod{p^3},
     \end{eqnarray*} proving parts $(1)$ and $(2)$ for $n =0,1,2$. Part $(3)$ follows from part $(1)$ and Lemma $2.1$.
\end{proof}

\begin{lemma}\label{lem:2.4}
    Let $p \ge 5$ and $r \equiv a \pmod {p^2(p-1)}$ with $a \in \{5,...,p\}$. Then there exist integers \[
\{\alpha_j : a-1 \le j < r-1 \text{ and } j \equiv a-1 \pmod{p-1}\}
\]
such that the following conditions are satisfied:

\begin{enumerate}
    \item for all $j$,
    \[
    \alpha_j \equiv \binom{r}{j} \pmod{p^2},
    \]
    
    \item for \(n=0,1,2,3,4\), 
    \[
    \sum_{j \ge n} \binom{j}{n}\alpha_j
    \equiv 0 \pmod{p^{5-n}}.
    \]

\end{enumerate}
\end{lemma}
\begin{proof}
We know that \begin{eqnarray*}
        (p-1) \sum_{\substack{
    j \geq 0 \\
    j \equiv a-1 \pmod{p-1}
}}
\binom{j}{m}\binom{r}{j} = \binom{r}{m} \sum_{\xi \in \mu'} \xi^{m-a+1} (1+\xi)^{r-m}.
    \end{eqnarray*}
    \begin{enumerate}
        \item [(i)] For $m=0$, we have \begin{eqnarray}
        \label{final m=0 cong}
        (p-1) \sum_{\substack{
    j \geq 0 \\
    j \equiv a-1 \pmod{p-1}
}}
\binom{r}{j}    &\equiv& \sum_{\xi \in \mu'} \xi^{1-a} (1+\xi)^a \pmod{p^2} \nonumber \\
&\equiv&  (p-1)(a+\delta_{a,p}) \pmod{p^2}, \text{ so} \\
(p-1) \sum_{\substack{
    a-1 \le j <r-1 \\
    j \equiv a-1 \pmod{p-1}
}}
\binom{r}{j}  &\equiv& 0 \pmod{p^2} \nonumber.
\end{eqnarray}
\item[(ii)] For $m=1$, we have \begin{eqnarray}
\label{final cong m=1}
    (p-1) \sum_{\substack{
    j \geq 0 \\
    j \equiv a-1 \pmod{p-1}
}}
j\binom{r}{j}    &\equiv& r\sum_{\xi \in \mu'} \xi^{2-a} (1+\xi)^{a-1} \pmod{p^2}\nonumber  \\
&\equiv& (p-1)r(a-1) \pmod{p^2}, \text{ so} \nonumber \\
(p-1) \sum_{\substack{
a-1\le j<r-1 \\
    j \equiv a-1 \pmod{p-1}
}}
j\binom{r}{j}    &\equiv& 0\pmod{p^2}.
\end{eqnarray}
\item[(iii)] For $m=2$, we have \begin{eqnarray*}
(p-1) \sum_{\substack{
j \ge 0 \\
    j \equiv a-1 \pmod{p-1}
}}
\binom{j}{2}\binom{r}{j}    &\equiv& \binom{r}{2}\sum_{\xi \in \mu'} \xi^{3-a} (1+\xi)^{a-2}\pmod{p^2}\\
&\equiv& (p-1)\binom{r}{2}(a-2)\pmod{p^2}, \text{ so}\\
    (p-1) \sum_{\substack{
    a-1 \le j <r-1 \\
    j \equiv a-1 \pmod{p-1}
}}
\binom{j}{2}\binom{r}{j}    &\equiv& 0 \pmod{p^2}.
\end{eqnarray*}
\item[(iv)] For $m=3$, we have \begin{eqnarray*}
    (p-1) \sum_{\substack{
j \ge 0 \\
    j \equiv a-1 \pmod{p-1}
}}
\binom{j}{3}\binom{r}{j}    &\equiv& \binom{r}{3}\sum_{\xi \in \mu'} \xi^{4-a} (1+\xi)^{a-3}\pmod{p^2}\\
&\equiv& (p-1)\binom{r}{3}(a-3)\pmod{p^2}, \text{ so}\\
    (p-1) \sum_{\substack{
    a-1 \le j <r-1 \\
    j \equiv a-1 \pmod{p-1}
}}
\binom{j}{3}\binom{r}{j}    &\equiv& 0 \pmod{p^2}.
\end{eqnarray*}
\item[(v)] For $m=4$, we have  \begin{eqnarray*}
    (p-1) \sum_{\substack{
j \ge 0 \\
    j \equiv a-1 \pmod{p-1}
}}
\binom{j}{4}\binom{r}{j}    &\equiv& \binom{r}{4}\sum_{\xi \in \mu'} \xi^{5-a} (1+\xi)^{a-4}\pmod{p^2}\\
&\equiv& (p-1)\binom{r}{4}(a-4)\pmod{p^2}, \text{ so}\\
    (p-1) \sum_{\substack{
    a-1 \le j <r-1 \\
    j \equiv a-1 \pmod{p-1}
}}
\binom{j}{4}\binom{r}{j}    &\equiv& 0 \pmod{p^2}.
\end{eqnarray*}
    \end{enumerate}
     Now put $\alpha_j =\binom{r}{j}+p\delta_j$ for $j \equiv a-1 \pmod{p-1}$, $a-1 \le j<r-1$ and define  \begin{eqnarray*}
        s_0 &=& -p^{-2} \sum_{\substack{
    a-1 \le j <r-1 \\
    j \equiv a-1 \pmod{p-1}
}}
\binom{r}{j},\\
s_1 &=& -p^{-2}\sum_{\substack{
    a-1 \le j <r-1\\
    j \equiv a -1\pmod{p-1}
}}
j\binom{r}{j},\\
s_2 &=& -p^{-2}\sum_{\substack{
    a-1 \le j <r-1 \\
    j \equiv a-1 \pmod{p-1}
}}
\binom{j}{2}\binom{r}{j}.
\end{eqnarray*}
The following system of linear equations 
\begin{eqnarray*}
    \delta_{a-1}+\delta_{a+p-2}+\delta_{a+2p-3} &\equiv& s_0 \pmod{p^3}\\
  ( a-1)\delta_{a-1}+(a+p-2)\delta_{a+p-2}+(a+2p-3)\delta_{a+2p-3} &\equiv& s_1 \pmod{p^2}\\
    \binom{a-1}{2}\delta_{a-1}+\binom{a+p-2}{2}\delta_{a+p-2}+\binom{a+2p-3}{2}\delta_{a+2p-3} &\equiv& s_2 \pmod{p}
\end{eqnarray*}
is solvable in the three 
unknowns 
$\delta_{a-1}$, $\delta_{a+p-2}$, $\delta_{a+2p-3}$.
We take the other $\delta_j=0$. Then
\begin{eqnarray*}
      \sum_{\substack{
    a-1\le j <r-1 \\
    j \equiv a-1 \pmod{p-1}
}} \alpha_j &\equiv& 0 \pmod{p^5} ,  \\
\sum_{\substack{
    a-1 \le j <r-1 \\
    j \equiv a-1 \pmod{p-1}
}} j\alpha_j &\equiv& 0 \pmod{p^4}, \\
\sum_{\substack{
    a-1 \le j <r-1 \\
    j \equiv a-1 \pmod{p-1}
}} \binom{j}{2}\alpha_j &\equiv& 0 \pmod{p^3},
     \end{eqnarray*}
     proving part (1) and part (2) for $n = 0,1,2$. Part
     (2) for $n=3,4$ now follows from part (1) and (iv), (v) above.
\end{proof}
\begin{lemma}\label{lem:2.5}
    Let \(p \ge 5\) and $r \equiv a \pmod{p(p-1)}$
with \(a \in \{5,\dots,p-1,p\}\).
Then there exist integers
\[
\{\alpha_j : a-1 \le j < r-1 \text{ and } j \equiv a-1 \pmod{p-1}\}
\]
such that the following conditions are satisfied:

\begin{enumerate}
    \item for all $j$,
    \[
    \alpha_j \equiv \binom{r}{j} \pmod p,
    \]
    
    \item  for \(n=0,1,2,3\), 
    \[
    \sum_{j \ge n} \binom{j}{n}\alpha_j
    \equiv 0 \pmod{p^{4-n}}.
    \]
\end{enumerate}
\end{lemma}

   \begin{proof}
    We know that \begin{eqnarray*}
        (p-1) \sum_{\substack{
    j \geq 0 \\
    j \equiv a-1 \pmod{p-1}
}}
\binom{j}{m}\binom{r}{j} = \binom{r}{m} \sum_{\xi \in \mu'} \xi^{m-a+1} (1+\xi)^{r-m}.
    \end{eqnarray*}
    \begin{enumerate}
        \item [(i)] For $m=0$, we have \begin{eqnarray}
        \label{mod p^2 m=0}
        (p-1) \sum_{\substack{
    j \geq 0 \\
    j \equiv a-1 \pmod{p-1}
}}
\binom{r}{j}    &\equiv& \sum_{\xi \in \mu'} \xi^{1-a} (1+\xi)^a \pmod{p} \nonumber \\
&\equiv&  (p-1)(a+\delta_{a,p}) \pmod{p}, \text{ so}  \\
(p-1) \sum_{\substack{
    a-1 \le j <r-1 \\
    j \equiv a-1 \pmod{p-1}
}}
\binom{r}{j}  &\equiv& 0 \pmod{p}. \nonumber
\end{eqnarray}
\item[(ii)] For $m=1$, we have \begin{eqnarray}
\label{mod p^2 For a=p}
    (p-1) \sum_{\substack{
    j \geq 0 \\
    j \equiv a-1 \pmod{p-1}
}}
j\binom{r}{j}    &\equiv& r\sum_{\xi \in \mu'} \xi^{2-a} (1+\xi)^{a-1} \pmod{p} \nonumber \\
&\equiv& (p-1)r(a-1) \pmod{p}, \text{ so} \nonumber \\
(p-1) \sum_{\substack{
a-1\le j<r-1 \\
    j \equiv a-1 \pmod{p-1}
}}
j\binom{r}{j}    &\equiv& 0\pmod{p}.
\end{eqnarray}
\item[(iii)] For $m=2$, we have \begin{eqnarray*}
(p-1) \sum_{\substack{
j \ge 0 \\
    j \equiv a-1 \pmod{p-1}
}}
\binom{j}{2}\binom{r}{j}    &\equiv& \binom{r}{2}\sum_{\xi \in \mu'} \xi^{3-a} (1+\xi)^{a-2} \pmod{p}\\
&\equiv & (p-1)\binom{r}{2}(a-2)\pmod{p}, \text{ so}\\
    (p-1) \sum_{\substack{
    a-1 \le j <r-1 \\
    j \equiv a-1 \pmod{p-1}
}}
\binom{j}{2}\binom{r}{j}    &\equiv& 0 \pmod{p}.
\end{eqnarray*}
\item[(iv)] For $m=3$, we have \begin{eqnarray*}
    (p-1) \sum_{\substack{
j \ge 0 \\
    j \equiv a-1 \pmod{p-1}
}}
\binom{j}{3}\binom{r}{j}    &\equiv& \binom{r}{3}\sum_{\xi \in \mu'} \xi^{4-a} (1+\xi)^{a-3} \pmod{p}\\
&\equiv & (p-1)\binom{r}{3}(a-3)\pmod{p}, \text{ so}\\
    (p-1) \sum_{\substack{
    a-1 \le j <r-1 \\
    j \equiv a-1 \pmod{p-1}
}}
\binom{j}{3}\binom{r}{j}    &\equiv& 0 \pmod{p}.
\end{eqnarray*} 
    \end{enumerate}
     Now put $\alpha_j =\binom{r}{j}+p\delta_j$ for $j \equiv a-1 \pmod{p-1}$, $a-1 \le j<r-1$ and define  \begin{eqnarray*}
        s_0 &=& -p^{-1} \sum_{\substack{
    a-1 \le j <r-1 \\
    j \equiv a-1 \pmod{p-1}
}}
\binom{r}{j},\\
s_1 &=& -p^{-1}\sum_{\substack{
    a-1 \le j <r-1\\
    j \equiv a -1\pmod{p-1}
}}
j\binom{r}{j},\\
s_2 &=& -p^{-1}\sum_{\substack{
    a-1 \le j <r-1 \\
    j \equiv a-1 \pmod{p-1}
}}
\binom{j}{2}\binom{r}{j}.
\end{eqnarray*}
The following system of linear equations 
     \begin{eqnarray*}
    \delta_{a-1}+\delta_{a+p-2}+\delta_{a+2p-3} &\equiv& s_0 \pmod{p^3}\\
  ( a-1)\delta_{a-1}+(a+p-2)\delta_{a+p-2}+(a+2p-3)\delta_{a+2p-3} &\equiv& s_1 \pmod{p^2}\\
    \binom{a-1}{2}\delta_{a-1}+\binom{a+p-2}{2}\delta_{a+p-2}+\binom{a+2p-3}{2}\delta_{a+2p-3} &\equiv& s_2 \pmod{p}
\end{eqnarray*}
is solvable in the three 
unknowns 
$\delta_{a-1}$, $\delta_{a+p-2}$, $\delta_{a+2p-3}$.
We take the other $\delta_j=0$. Then
\begin{eqnarray*}
      \sum_{\substack{
    a-1\le j <r-1 \\
    j \equiv a-1 \pmod{p-1}
}} \alpha_j &\equiv& 0 \pmod{p^4} ,  \\
\sum_{\substack{
    a-1 \le j <r-1 \\
    j \equiv a-1 \pmod{p-1}
}} j\alpha_j &\equiv& 0 \pmod{p^3}, \\
\sum_{\substack{
    a-1 \le j <r-1 \\
    j \equiv a-1 \pmod{p-1}
}} \binom{j}{2}\alpha_j &\equiv& 0 \pmod{p^2},
\end{eqnarray*}
     proving part $(1)$ and $(2)$.
\end{proof}

\begin{lemma}\label{lem:2.6}
    Let $p \geq 5$, $r \equiv 1 \pmod{p-1}$.

\begin{enumerate}
\item [(1)] If $p^2 \mid p-r$, then there are integers
\[
\{\alpha_j : p-1 \leq j < r-1 \text{ and } j \equiv 0 \pmod{p-1}\}
\]
such that the following conditions are satisfied:
\begin{itemize}
    \item for all $j$,
    \[
    \alpha_j \equiv \binom{r}{j} \pmod{p^2},
    \]
    \item for $0 \leq n \leq 4$, 
    \[
    \sum_{j \geq n} \binom{j}{n}\alpha_j \equiv 0 \pmod{p^{5-n}}.
    \]
\end{itemize}

\item[(2)] If $p^2 \mid p-r$, then there are integers
\[
\{\alpha_{j}: p \leq j < r \text{ and } j \equiv 1 \pmod{p-1}\}
\]
such that the following conditions are satisfied:
\begin{itemize}
    \item for all $j$,
    \[
    \alpha_{j} \equiv \binom{r}{j} \pmod{p^2},
    \]
    \item for $0 \leq n \leq 4$, 
    \[
    \sum_{j\geq n} \binom{j}{n}\alpha_{j} \equiv 0 \pmod{p^{5-n}}.
    \]
\end{itemize}
\end{enumerate}
\end{lemma}
\begin{proof}
     (1) This is a special case of Lemma~\ref{lem:2.4} when $a=p$.

(2)  We know that      \begin{eqnarray*}
        (p-1) \sum_{\substack{
    j \geq 0 \\
    j \equiv a \pmod{p-1}
}}
\binom{j}{m}\binom{r}{j} = \binom{r}{m} \sum_{\xi \in \mu'} \xi^{m-a} (1+\xi)^{r-m}.
    \end{eqnarray*}
    Taking $a=1$ we obtain that  \begin{eqnarray*}
        (p-1) \sum_{\substack{
    j \geq 0 \\
    j \equiv 1 \pmod{p-1}
}}
\binom{j}{m}\binom{r}{j} \equiv \binom{r}{m} \sum_{\xi \in \mu'} \xi^{m-1} (1+\xi)^{r-m}.
\end{eqnarray*}
\begin{enumerate}
        \item [(i)] For $m=0$, we have \begin{eqnarray}
      \label{mod p^2}
        (p-1) \sum_{\substack{
    j \geq 0 \\
    j \equiv 1 \pmod{p-1}
}}
\binom{r}{j}    &\equiv& \sum_{\xi \in \mu'} \xi^{-1} (1+\xi)^p \pmod{p^2}  \label{mod p^2} \nonumber  \\
&\equiv&  (p+1)(p-1) \pmod{p^2}, \text{ so} \\
(p-1) \sum_{\substack{
    p \le j <r \\
    j \equiv 1 \pmod{p-1}
}}
\binom{r}{j}  &\equiv& 0 \pmod{p^2}\nonumber.
\end{eqnarray}
\item[(ii)] For $m=1$, we have \begin{eqnarray}
\label{mod p cube}
    (p-1) \sum_{\substack{
    j \geq 0 \\
    j \equiv 1 \pmod{p-1}
}}
j\binom{r}{j}    &\equiv& r\sum_{\xi \in \mu'} (1+\xi)^{p-1} \pmod{p^2}  \nonumber \\
&\equiv& 2r(p-1) \pmod{p^2}, \text{ so} \nonumber \\
(p-1) \sum_{\substack{
p\le j<r \\
    j \equiv 1 \pmod{p-1}
}}
j\binom{r}{j}    &\equiv& 0\pmod{p^2}.
\end{eqnarray}
\item[(iii)] For $m=2$, we have \begin{eqnarray*}
(p-1) \sum_{\substack{
j \ge 0 \\
    j \equiv 1 \pmod{p-1}
}}
\binom{j}{2}\binom{r}{j}    &\equiv& \binom{r}{2}\sum_{\xi \in \mu'} \xi (1+\xi)^{p-2} \pmod{p^2}\\
&\equiv & (p-1)\binom{r}{2}\pmod{p^2}, \text{ so}\\
    (p-1) \sum_{\substack{
    p \le j <r \\
    j \equiv 1 \pmod{p-1}
}}
\binom{j}{2}\binom{r}{j}    &\equiv& 0 \pmod{p^2}.
\end{eqnarray*}
\item[(iv)] For $m=3$, we have \begin{eqnarray*}
    (p-1) \sum_{\substack{
j \ge 0 \\
    j \equiv 1 \pmod{p-1}
}}
\binom{j}{3}\binom{r}{j}    &\equiv& \binom{r}{3}\sum_{\xi \in \mu'} \xi^{2} (1+\xi)^{p-3} \pmod{p^2}\\
&\equiv & (p-1)\binom{r}{3}\pmod{p^2}, \text{ so}\\
    (p-1) \sum_{\substack{
    p \le j <r \\
    j \equiv1  \pmod{p-1}
}}
\binom{j}{3}\binom{r}{j}    &\equiv& 0 \pmod{p^2}.
\end{eqnarray*}
\item[(v)] For $m=4$, we have \begin{eqnarray*}
    (p-1) \sum_{\substack{
j \ge 0 \\
    j \equiv 1 \pmod{p-1}
}}
\binom{j}{4}\binom{r}{j}    &\equiv& \binom{r}{4}\sum_{\xi \in \mu'} \xi^{3} (1+\xi)^{p-4} \pmod{p^2}\\
&\equiv & (p-1)\binom{r}{4}\pmod{p^2}, \text{ so}\\
    (p-1) \sum_{\substack{
    p \le j <r \\
    j \equiv1  \pmod{p-1}
}}
\binom{j}{4}\binom{r}{j}    &\equiv& 0 \pmod{p^2}. 
\end{eqnarray*} 
    \end{enumerate}
     Now put $\gamma_j =\binom{r}{j}+p^2\delta_j$ for $j \equiv 1 \pmod{p-1}$, $p\le j<r$ and define  \begin{eqnarray*}
        s_0 &=& -p^{-2} \sum_{\substack{
    p\le j <r\\
    j \equiv 1 \pmod{p-1}
}}
\binom{r}{j},\\
s_1 &=& -p^{-2}\sum_{\substack{
    p\le j <r\\
    j \equiv 1\pmod{p-1}
}}
j\binom{r}{j},\\
s_2 &=& -p^{-2}\sum_{\substack{
    p \le j <r \\
    j \equiv 1 \pmod{p-1}
}}
\binom{j}{2}\binom{r}{j}.
\end{eqnarray*}
The following system of linear equations 
     \begin{eqnarray*}
    \delta_{p}+\delta_{2p-1}+\delta_{3p-2} &\equiv& s_0 \pmod{p^3}\\
  p\delta_{p}+(2p-1)\delta_{2p-1}+(3p-2)\delta_{3p-2} &\equiv& s_1 \pmod{p^2}\\
    \binom{p}{2}\delta_{p}+\binom{2p-1}{2}\delta_{2p-1}+\binom{3p-2}{2}\delta_{3p-2} &\equiv& s_2 \pmod{p}
\end{eqnarray*}
is solvable in the three 
unknowns 
$\delta_{p}$, $\delta_{2p-1}$, $\delta_{3p-2}$.
We take the other $\delta_j=0$. Then
\begin{eqnarray*}
      \sum_{\substack{
    p\le j <r \\
    j \equiv 1 \pmod{p-1}
}} \gamma_j &\equiv& 0 \pmod{p^5} ,  \\
\sum_{\substack{
p \le j <r \\
    j \equiv 1 \pmod{p-1}
}} j\gamma_j &\equiv& 0 \pmod{p^4}, \\
\sum_{\substack{
    p\le j <r \\
    j \equiv 1 \pmod{p-1}
}} \binom{j}{2}\gamma_j &\equiv& 0 \pmod{p^3},
\end{eqnarray*}
   proving (2). 
\end{proof}
\begin{lemma}\label{lem:2.7}
    Let \(p \ge 5\). Let \(r \equiv 2 \pmod{p-1}\) and $p \mid \binom{r}{2}$. Then there exist integers
\[
\{\alpha_j : p+1 \le j < r \text{ and } j \equiv 2\pmod{p-1}\}
\]
such that the following conditions are satisfied:

\begin{enumerate}
    \item for all $j$,
    \[
    \alpha_j \equiv \binom{r}{j} \pmod p,
    \]
    \item for $n = 0,1,2,3$, 
\[
\sum_{j \geq n} \binom{j}{n}\alpha_j \equiv 0 \pmod{p^{4-n}}.
\]
 \end{enumerate}
\end{lemma}

\begin{proof}
    We know that \begin{eqnarray*}
        (p-1) \sum_{\substack{
    j \geq 0 \\
    j \equiv 2 \pmod{p-1}
}}
\binom{j}{m}\binom{r}{j} = \binom{r}{m} \sum_{\xi \in \mu'} \xi^{m-2} (1+\xi)^{r-m},
    \end{eqnarray*}
    where $\mu'$ is the set of $(p-1)$-th roots of unity except $-1$.
    \begin{enumerate}
        \item [(i)] For $m=0$, we have \begin{eqnarray*}
        (p-1) \sum_{\substack{
    j \geq 0 \\
    j \equiv 2 \pmod{p-1}
}}
\binom{r}{j}    &\equiv& \sum_{\xi \in \mu'} \xi^{-2} (1+\xi)^2 \pmod{p}\\
&\equiv&  (p-1) \pmod{p},  \text{ so} \\
(p-1) \sum_{\substack{
    p+1 \le j <r \\
    j \equiv 2 \pmod{p-1}
}}
\binom{r}{j}  &\equiv& 0 \pmod{p}, \text{ since $p\mid \binom{r}{2}$}.
\end{eqnarray*}
\item[(ii)] For $m=1$, we have \begin{eqnarray*}
    (p-1) \sum_{\substack{
    j \geq 0 \\
    j \equiv 2 \pmod{p-1}
}}
j\binom{r}{j}    &\equiv& r\sum_{\xi \in \mu'} \xi^{-1} (1+\xi) \pmod{p} \\
&\equiv& (p-1)r \pmod{p},\,\,\text{ so}\\
(p-1) \sum_{\substack{
p+1\le j<r \\
    j \equiv 2 \pmod{p-1}
}}
j\binom{r}{j}    &\equiv& 0\pmod{p}.
\end{eqnarray*}
\item[(iii)] For $m=2$, we have \begin{eqnarray*}
(p-1) \sum_{\substack{
    j \ge 0 \\
    j \equiv 2 \pmod{p-1}
}}
\binom{j}{2}\binom{r}{j}    &\equiv& \binom{r}{2}\sum_{\xi \in \mu'} 1 \pmod{p}\\
&\equiv& (p-2)\binom{r}{2} \equiv 0
\pmod{p},\,\,\text{ so}\\
(p-1) \sum_{\substack{
    p+1\le j<r \\
    j \equiv 2 \pmod{p-1}
}}
\binom{j}{2}\binom{r}{j}    &\equiv& 0 \pmod{p}.
\end{eqnarray*}
\item[(iii)] For $m=3$, we have \begin{eqnarray*}
(p-1) \sum_{\substack{
    j \ge 0 \\
    j \equiv 2 \pmod{p-1}
}}
\binom{j}{3}\binom{r}{j}    &\equiv& \binom{r}{3}\sum_{\xi \in \mu'} \xi(1+\xi)^{r-3} \pmod{p}, \text{ so}\\
(p-1) \sum_{\substack{
    p+1\le j<r \\
    j \equiv 2 \pmod{p-1}
}}
\binom{j}{3}\binom{r}{j}    &\equiv& 0 \pmod{p}.
\end{eqnarray*}

    \end{enumerate}
    Now, put $\alpha_j =\binom{r}{j}+p\delta_j$ for $p+1 \le j<r$ such that $j \equiv 2 \pmod{p-1}$ and define  \begin{eqnarray*}
        s_0 &=& -p^{-1} \sum_{\substack{
    p+1 \le j <r \\
    j \equiv 2 \pmod{p-1}
}}
\binom{r}{j},\\
s_1 &=& -p^{-1}\sum_{\substack{
    p+1 \le j <r\\
    j \equiv 2\pmod{p-1}
}}
j\binom{r}{j},\\
s_2 &=& -p^{-1}\sum_{\substack{
    p+1 \le j <r \\
    j \equiv 2 \pmod{p-1}
}}
\binom{j}{2}\binom{r}{j}.
\end{eqnarray*}
Then, the following system of linear equations 
\begin{eqnarray*}
    \delta_{p+1}+\delta_{2p}+\delta_{3p-1} &\equiv& s_0 \pmod{p^3}\\
   (p+1) \delta_{p+1}+2p\delta_{2p}+(3p-1)\delta_{3p-1} &\equiv& s_1 \pmod{p^2}\\
    \binom{p+1}{2}\delta_{p+1}+\binom{2p}{2}\delta_{2p}+\binom{3p-1}{2}\delta_{3p-1} &\equiv& s_2 \pmod{p}
\end{eqnarray*}
is solvable for $\delta_{p+1}$, $\delta_{2p}$, $\delta_{3p-1}$. We take the other $\delta_j=0$.
It follows that 
\begin{eqnarray*}
      \sum_{\substack{
    a\le j <r \\
    j \equiv a \pmod{p-1}
}} \alpha_j &\equiv& 0 \pmod{p^4} ,  \\
\sum_{\substack{
    a \le j <r \\
    j \equiv a \pmod{p-1}
}} j\alpha_j &\equiv& 0 \pmod{p^3}, \\
\sum_{\substack{
    a \le j <r \\
    j \equiv a \pmod{p-1}
}} \binom{j}{2}\alpha_j &\equiv& 0 \pmod{p^2}.
     \end{eqnarray*}
     This proves part $(1)$ and part $(2)$.
\end{proof}

\begin{lemma}\label{lem:2.8}
    Let $p \geq 5$, $r \equiv 2 \pmod{p-1}$ and $p \mid r$. Then there exist integers 
 \[
\{\alpha_j : p \leq j < r-1 \text{ and } j \equiv 1 \pmod{p-1}\}
\]
such that the following conditions are satisfied:
\begin{enumerate}
    \item for all $j$,
    \[
    \alpha_j \equiv 
    \begin{cases}
         \binom{r}{p}-2 & \text{if } j = p \\
         \binom{r}{j} & \text{if } j > p
    \end{cases} \pmod{p}, 
    \]
    \item for $0 \leq n \leq 3$, 
    \[
    \sum_{j \geq n} \binom{j}{n}\alpha_j \equiv 0 \pmod{p^{4-n}}.
    \]
\end{enumerate}
\end{lemma}
\begin{proof}
    We know that \[
    (p-1) \sum_{\substack{
    j \geq 0 \\
    j \equiv 1 \pmod{p-1}
}}
\binom{j}{m}\binom{r}{j} \equiv \binom{r}{m} \sum_{\xi \in \mu'} \xi^{m-1} (1+\xi)^{r-m}.  
    \]
    \begin{enumerate}
        \item [(i)] For $m=0$, we have \begin{eqnarray*}
        (p-1) \sum_{\substack{
    j \geq 0 \\
    j \equiv 1 \pmod{p-1}
}}
\binom{r}{j}    &\equiv& \sum_{\xi \in \mu'} \xi^{-1} (1+\xi)^2 \pmod{p}\\
&\equiv&  2(p-1) \pmod{p}, \text{ so} \\
(p-1) \sum_{\substack{
    p \le j <r-1 \\
    j \equiv 1 \pmod{p-1}
}}
\binom{r}{j}  &\equiv& 2(p-1) \pmod{p}, \text{ since $p\mid r$}.
\end{eqnarray*}
\item[(ii)] For $m=1$, we have \begin{eqnarray*}
    (p-1) \sum_{\substack{
    j \geq 0 \\
    j \equiv 1 \pmod{p-1}
}}
j\binom{r}{j}    &\equiv& r\sum_{\xi \in \mu'}  (1+\xi) \pmod{p} \\
(p-1) \sum_{\substack{
p\le j<r-1 \\
    j \equiv 1 \pmod{p-1}
}}
j\binom{r}{j}    &\equiv& 0\pmod{p}.
\end{eqnarray*}
\item[(iii)] For $m=2$, we have \begin{eqnarray*}
    (p-1) \sum_{\substack{
    p\le j<r-1 \\
    j \equiv 1 \pmod{p-1}
}}
\binom{j}{2}\binom{r}{j}    &\equiv& 0 \pmod{p}.
\end{eqnarray*}
\item[(iv)] For $m=3$, we have \begin{eqnarray*}
    (p-1) \sum_{\substack{
    p\le j<r-1 \\
    j \equiv 1 \pmod{p-1}
}}
\binom{j}{3}\binom{r}{j}    &\equiv& 0 \pmod{p}.
\end{eqnarray*}
    \end{enumerate}
    Now put $\alpha_p=\binom{r}{p}-2+p\delta_p$, $\alpha_j =\binom{r}{j}+p\delta_j$ for $j>p$ and  \begin{eqnarray*}
        s_0 &=& -p^{-1}\biggl( -2+\sum_{\substack{
    p \le j <r-1 \\
    j \equiv 1 \pmod{p-1}
}}
\binom{r}{j}\biggl),\\
s_1 &=& -p^{-1}\biggl(-2p+\sum_{\substack{
    p \le j <r-1 \\
    j \equiv 1 \pmod{p-1}
}}
j\binom{r}{j}\biggl),\\
s_2 &=& -p^{-1}\biggl(-2\binom{p}{2}+\sum_{\substack{
    p \le j <r-1 \\
    j \equiv 1 \pmod{p-1}
}}
\binom{j}{2}\binom{r}{j}\biggl).
\end{eqnarray*}
The following system of linear equations 
\begin{eqnarray*}
    \delta_p+\delta_{2p-1}+\delta_{3p-2} &\equiv& s_0 \pmod{p^3}\\
   p \delta_p+(2p-1)\delta_{2p-1}+(3p-2)\delta_{3p-2} &\equiv& s_1 \pmod{p^2}\\
    \binom{p}{2}\delta_p+\binom{2p-1}{2}\delta_{2p-1}+\binom{3p-2}{2}\delta_{3p-2} &\equiv& s_2 \pmod{p}
\end{eqnarray*}
is solvable in the three 
unknowns 
$\delta_{p}$, $\delta_{2p-1}$, $\delta_{3p-2}$.
We take the other $\delta_j=0$. Then
     \begin{eqnarray*}
      \sum_{\substack{
    p \le j <r-1 \\
    j \equiv 1 \pmod{p-1}
}} \alpha_j &\equiv& 0 \pmod{p^4} ,  \\
\sum_{\substack{
    p \le j <r-1 \\
    j \equiv 1 \pmod{p-1}
}} j\alpha_j &\equiv& 0 \pmod{p^3}, \\
\sum_{\substack{
    p \le j <r-1 \\
    j \equiv 1 \pmod{p-1}
}} \binom{j}{2}\alpha_j &\equiv& 0 \pmod{p^2},
     \end{eqnarray*}
proving parts $(1)$ and $(2)$.
\end{proof}

\begin{lemma}\label{lem:2.9}
    Let $p \geq 5$, $r \equiv 2 \pmod{p-1}$. Then there exist integers

        \[
\{\alpha_j : p \leq j < r-1 \text{ and } j \equiv 1 \pmod{p-1}\}
\]
such  the following conditions are satisfied:
\begin{enumerate}
    \item for all $j$,
    \[
    \alpha_j \equiv 
    \begin{cases}
         \binom{r-1}{p}+(r-1) & \text{if } j = p \\
         \binom{r-1}{j} & \text{if } j > p
    \end{cases} \pmod{p}, 
    \]
    
    \item for $0 \leq n \leq 3$,
    \[
    \sum_{j \geq n} \binom{j}{n}\alpha_j \equiv 0 \pmod{p^{4-n}}.
    \]
\end{enumerate}
    
\end{lemma}
\begin{proof}
    We have  \[
    (p-1) \sum_{\substack{
    j \geq 0 \\
    j \equiv 1 \pmod{p-1}
}}
\binom{j}{m}\binom{r-1}{j} \equiv \binom{r-1}{m} \sum_{\xi \in \mu'} \xi^{m-1} (1+\xi)^{r-1-m}.  
    \]
    \begin{enumerate}
        \item [(i)] For $m=0$, we have \begin{eqnarray*}
        (p-1) \sum_{\substack{
    j \geq 0 \\
    j \equiv 1 \pmod{p-1}
}}
\binom{r-1}{j}    &\equiv& \sum_{\xi \in \mu'} \xi^{-1} (1+\xi) \pmod{p}\\
&\equiv&  (p-1) \pmod{p}, \text{ so} \\
(p-1) \sum_{\substack{
    p \le j <r-1 \\
    j \equiv 1 \pmod{p-1}
}}
\binom{r-1}{j}  &\equiv& -(p-1)(r-1) \pmod{p}.
\end{eqnarray*}
\item[(ii)] For $m=1$, by Lemma~\ref{lem:2.1} we have \begin{eqnarray*}
   \sum_{\substack{j \equiv 1 \!\pmod{p-1} \\ 1<j<r-1}}
{j}\binom{r-1}{j}
\equiv
p (r-1)\frac{p+1-r}{p-1}\equiv0
\pmod{p} .
\end{eqnarray*}
\item[(iii)] For $m=2$, we have \begin{eqnarray*}
    (p-1) \sum_{\substack{
    j\ge 0 \\
    j \equiv 1 \pmod{p-1}
}}
\binom{j}{2}\binom{r-1}{j}   &\equiv& \binom{r-1}{2} \sum_{\xi \in \mu'} \xi (1+\xi)^{p-2} \pmod{p}\\
&\equiv& (p-1)\binom{r-1}{2} \pmod{p}, \text{ so}\\
(p-1) \sum_{\substack{
    p\le j<r-1 \\
    j \equiv 1 \pmod{p-1}
}}
\binom{j}{2}\binom{r-1}{j}&\equiv& 0 \pmod{p}.
\end{eqnarray*}
\item[(iv)] For $m=3$, we have \begin{eqnarray*}
    (p-1) \sum_{\substack{
    j\ge 0 \\
    j \equiv 1 \pmod{p-1}
}}
\binom{j}{3}\binom{r-1}{j}   &\equiv& \binom{r-1}{3} \sum_{\xi \in \mu'} \xi^2 (1+\xi)^{p-3} \pmod{p}\\
&\equiv& (p-1)\binom{r-1}{3} \pmod{p}, \text{ so}\\
(p-1) \sum_{\substack{
    p\le j<r-1 \\
    j \equiv 1 \pmod{p-1}
}}
\binom{j}{3}\binom{r-1}{j}&\equiv& 0 \pmod{p}.
\end{eqnarray*}
    \end{enumerate}
    Now put $\alpha_p=\binom{r-1}{p}+(r-1)+p\delta_p$, $\alpha_j =\binom{r-1}{j}+p\delta_j$ for $j>p$ and  \begin{eqnarray*}
        s_0 &=& -p^{-1}\biggl( r-1+\sum_{\substack{
    p \le j <r-1 \\
    j \equiv 1 \pmod{p-1}
}}
\binom{r-1}{j}\biggl),\\
s_1 &=& -p^{-1}\biggl(p(r-1)+\sum_{\substack{
    p \le j <r-1 \\
    j \equiv 1 \pmod{p-1}
}}
j\binom{r-1}{j}\biggl),\\
s_2 &=& -p^{-1}\biggl(\binom{p}{2}(r-1)+\sum_{\substack{
    p \le j <r-1 \\
    j \equiv 1 \pmod{p-1}
}}
\binom{j}{2}\binom{r-1}{j}\biggl).
\end{eqnarray*}
The following system of linear equations 
   \begin{eqnarray*}
    \delta_p+\delta_{2p-1}+\delta_{3p-2} &\equiv& s_0 \pmod{p^3}\\
    p\delta_p+(2p-1)\delta_{2p-1}+(3p-2)\delta_{3p-2} &\equiv& s_1 \pmod{p^2}\\
    \binom{p}{2}\delta_p+\binom{2p-1}{2}\delta_{2p-1}+\binom{3p-2}{2}\delta_{3p-2} &\equiv& s_2 \pmod{p}
\end{eqnarray*}
is solvable in the three 
unknowns 
$\delta_{p}$, $\delta_{2p-1}$, $\delta_{3p-2}$.
We take the other $\delta_j=0$. Then \begin{eqnarray*}
      \sum_{\substack{
    p \le j <r-1 \\
    j \equiv 1 \pmod{p-1}
}} \alpha_j &\equiv& 0 \pmod{p^4} ,  \\
\sum_{\substack{
    p \le j <r-1 \\
    j \equiv 1 \pmod{p-1}
}} j\alpha_j &\equiv& 0 \pmod{p^3}, \\
\sum_{\substack{
    p \le j <r-1 \\
    j \equiv 1 \pmod{p-1}
}} \binom{j}{2}\alpha_j &\equiv& 0 \pmod{p^2},
     \end{eqnarray*}
     proving parts $(1)$ and $(2)$.
\end{proof}
\begin{comment}{\color{red}\begin{cor}\label{lem: rmk 2.2}
Let $p \geq 5$, $r \equiv 2 \pmod{p-1}$ and $p \mid r-1$. Then there exist integers

        \[
\{\alpha_j : p \leq j < r-1 \text{ and } j \equiv 1 \pmod{p-1}\}
\]
such  the following conditions are satisfied:
\begin{enumerate}
    \item for all $j$,
    \[
    \alpha_j \equiv \binom{r-1}{j} \pmod{p}, 
    \]
    
    \item for $0 \leq n \leq 3$,
    \[
    \sum_{j \geq n} \binom{j}{n}\alpha_j \equiv 0 \pmod{p^{4-n}}.
    \]
\end{enumerate}
\begin{proof}
    We note that \begin{eqnarray*}
        (p-1) \sum_{\substack{
    p \le j <r-1 \\
    j \equiv 1 \pmod{p-1}
}}
\binom{r-1}{j}  &\equiv& 0 \pmod{p}.
\end{eqnarray*}
Here we put $\alpha_j=\binom{r}{j}+p\delta_j$ for $p \le j <r-1$ and 
\begin{eqnarray*}
        s_0 &=& -p^{-1}\sum_{\substack{
    p \le j <r-1 \\
    j \equiv 1 \pmod{p-1}
}}
\binom{r-1}{j},\\
s_1 &=& -p^{-1}\color{red}\sum_{\substack{
    p \le j <r-1 \\
    j \equiv 1 \pmod{p-1}
}}
j\binom{r-1}{j},\\
s_2 &=& -p^{-1}\color{red}\sum_{\substack{
    p \le j <r-1 \\
    j \equiv 1 \pmod{p-1}
}}
\binom{j}{2}\binom{r-1}{j}.
\end{eqnarray*}
Rest of the proof is same as previous lemma.
\end{proof}
 \end{cor}}
\end{comment}

\subsection{A quotient of  $V_r = \mathrm{Sym}^r \bar{\mathbb F}_p^2$}
\subsubsection{Definition of $P$}
Let $V_r$ denote the $(r+1)$-dimensional $\bar{\mathbb F}_p$-vector space of homogeneous polynomials in two variables $X$ and $Y$ of degree $r$ over $\bar{\mathbb F}_p$ on which the group $\Gamma= \mathrm{GL_2}(\mathbb F_p)$ acts on $V_r$ by the formula $\begin{pmatrix}
a & b\\
c & d
\end{pmatrix}
\cdot F(X,Y)=F(aX+cY,bX+dY)$.

Let $X_{r-1} \subset V_r$ be the $\bar{\mathbb F}_p[\Gamma]$-submodule generated by the monomial $X^{r-1}Y$. Let $\theta=X^pY-XY^p$. We define three submodules of $V_r$ as follows.
\[
V_r^{*}:=\{f\in V_r : \theta \mid f\}
\]
\[
V_r^{**}:=\{f\in V_r: \theta^2\mid f\}
\]
\[
V_r^{***}:=\{f\in V_r : \theta^3 \mid f\}.
\]
We set $X_{r-1}^*:= X_{r-1}\cap V_r^*$, $X_{r-1}^{**}:= X_{r-1}\cap V_r^{**}$ and $X_{r-1}^{***}:= X_{r-1}\cap V_r^{***}$.

By definition of the lattice $\Theta_{k,a_p}$, there is a surjection $\mathrm{ind}_{KZ}^G\mathrm{Sym}^{k-2}{\overline{\mathbb Z}_p}^2\twoheadrightarrow \Theta_{k,a_p} $, which induces a surjective map 
\[
\mathrm{ind}_{KZ}^G\mathrm{Sym}^{k-2}V_r\twoheadrightarrow \bar\Theta_{k,a_p}
\]
for $r=k-2$. By Remark $4.4$ of \cite{BG09} we know that when $v(a_p)=2$, then the above map factors through 
\[
P:= \frac{V_r}{X_{r-1}+V_r^{***}}.
\]
\subsubsection{A filtration on $P$}
To invoke the results of Glover~\cite{Glo78}, we adopt a slight abuse of notation by realizing \(V_r\) as the space of homogeneous polynomials of degree \(r\) in two variables \(X\) and \(Y\) with coefficients in \(\mathbb{F}_p\). After determining the structure of \(P\) over \(\mathbb{F}_p\), the analogous statement over \(\overline{\mathbb{F}}_p\) follows immediately by scalar extension.

Throughout, we choose representatives \(1 \leq a \leq p-1\) for the residue classes of \(r \bmod (p-1)\). Observe that the case \(a=1\) corresponds to \(a=p\) in the notation of the introduction. For convenience, write
$V_r^{(i)}$ when there are $i$ $*$'s. The submodules $V_r^{(i)}$ give a filtration on $V_r$\[
V_r \supset V_r^{(1)} \supset ...\supset V_r^{(i)} \supset V_r^{(i+1)} \supset ...
\]
where each subquotient $\frac{V_r^{(i)}}{V_r^{(i+1)}}$ contains two Jordan-H\"older factors. The following lemma describes these two JH factors. For an integer $n$, let $\left[n\right] \in \{1,\ldots, p-1\} $ be such that $n \equiv \left[n\right] \pmod{p-1}.$

\begin{lemma}\label{lem:2.10}
Let $p \ge 2$, $m \ge 0$, and $
m(p+1)+p \le r \equiv a \pmod{p-1}$, $1 \le a \le p-1$.
Then we have a short exact sequence of $\Gamma$-modules
\begin{equation}\label{eq:2.5}
0 \longrightarrow V_{[a-2m]}\otimes D^{m}
\longrightarrow V_r^{(m)}/V_r^{(m+1)}
\longrightarrow V_{p-1-[a-2m]}\otimes D^{\,a-m}
\longrightarrow 0,
\end{equation}
and this sequence splits if and only if $
a \equiv 2m \pmod{p-1}.
$ If
$
r' = r-m(p+1)
\text{ and}\,\,
a'=[a-2m],
$
then the rightmost map in \eqref{eq:2.5} is given on a basis by
\[
\theta^{m}X^{r'-i}Y^{i}
\pmod{V_r^{(m+1)}}
\longmapsto
\begin{cases}
0,
& \text{if } 0\le i<a' \text{ or } i=r',\\[1ex]
(-1)^{\,r'-i}
\displaystyle\binom{p-1-a'}{\,i-a'\,}
X^{p-1-i}Y^{\,i-a'},
& \text{if } a'\le i\le p-1.
\end{cases}
\]
\end{lemma}
\begin{proof}
   \cite[Lemma 2.11]{GV22}. 
\end{proof}

As the map $
\mathrm{ind}_{KZ}^G\mathrm{Sym}^{k-2}V_r\twoheadrightarrow \bar\Theta_{k,a_p}
$ factors through $P$, we only need to  make the JH factors of the first three subquotients explicit.

\begin{prop}\label{prop:3.1}
Let \(p \geq 5\), and suppose \(r \equiv a \pmod{p-1}\) with \(1 \leq a \leq p-1\).

\begin{enumerate}
    \item[(i)]     For \(r \geq p\), the \(\Gamma\)-module structure of \(V_r / V_r^*\) is given by
    \[
    0 \longrightarrow V_a \longrightarrow \frac{V_r}{V_r^*}
    \longrightarrow V_{p-a-1} \otimes D^a \longrightarrow 0,
    \]
    and this sequence splits if and only if \(a = p-1\).

    \item[(ii)] For \(r \geq 2p+1\), the \(\Gamma\)-module structure of \(V_r^* / V_r^{**}\) is given by
    \begin{eqnarray*}
        0 \longrightarrow V_{p-2} \otimes D \longrightarrow \frac{V_r^*}{V_r^{**}} 
    \longrightarrow V_1 \longrightarrow 0, &\text{if }a=1,\\
     0 \longrightarrow V_{p-1} \otimes D \longrightarrow \frac{V_r^*}{V_r^{**}}
    \longrightarrow V_0 \otimes D \longrightarrow 0, 
    &\text{if } a=2,\\
    0 \longrightarrow V_{a-2} \otimes D \longrightarrow \frac{V_r^*}{V_r^{**}}
    \longrightarrow V_{p-a+1} \otimes D^{a-1} \longrightarrow 0, 
    & \text{if } 3 \leq a \leq p-1,
     \end{eqnarray*}
    and the above sequences split if and only if \(a=2\).
    \item[(iii)] For $r\geq 3p+2$,  the \(\Gamma\)-module structure of \(V_r^{**} / V_r^{***}\) is given by
    \begin{eqnarray*}
         0 \longrightarrow V_{p-4} \otimes D^2 \longrightarrow \frac{V_r^{**}}{V_r^{***}} \longrightarrow V_3 \otimes  D^{p-2} \longrightarrow 0, &\text{if } a=1,\\
    0 \longrightarrow V_{p-3}\otimes D^2 \longrightarrow \frac{V_r^{**}}{V_r^{***}} \longrightarrow V_{2} \longrightarrow 0, &\text{if } a=2,\\
     0 \longrightarrow V_{p-2}\otimes D^2 \longrightarrow \frac{V_r^{**}}{V_r^{***}} \longrightarrow V_{1}\otimes D \longrightarrow 0, &\text{if } a=3,\\
    0 \longrightarrow V_{p-1}\otimes D^2 \longrightarrow \frac{V_r^{**}}{V_r^{***}}\longrightarrow V_{0}\otimes D^2 \longrightarrow 0, & \text{if } a=4,\\
    0 \longrightarrow V_{a-4} \otimes D^2 \longrightarrow \frac{V_r^{**}}{V_r^{***}}\longrightarrow V_{p-a+3} \otimes D^{a-2} \longrightarrow 0, & \text{if } 5\leq a \leq p-1, 
     \end{eqnarray*}
     and the above sequences split if and only if \(a=4\).
\end{enumerate}
\end{prop}

The filtration above induces the following finite filtration on $V_r/V_r^{***}$,
namely,
$$
0 \subset
\frac{V_r^{**}}{V_r^{***}}
\subset
\frac{V_r^{*}}{V_r^{***}}
\subset
\frac{V_r}{V_r^{***}}.
$$
The following diagram is the associated graded of the filtration above. We arrange and label the corresponding Jordan-H\"older factors as follows:

\[
\begin{tikzcd}[column sep=3cm, row sep=0.9cm]
\tilde{J}_4 \arrow[d] &
\tilde{J}_2 \arrow[d] &
\tilde{J}_0 \arrow[d] \\
\dfrac{V_r^{**}}{V_r^{***}} \arrow[d] &
\dfrac{V_r^{*}}{V_r^{**}} \arrow[d] &
\dfrac{V_r}{V_r^{*}} \arrow[d] \\
\tilde{J}_5 &
\tilde{J}_3 &
\tilde{J}_1.
\end{tikzcd}
\]
Passing to the image of 
the above filtration in $P$, we obtain the following filtration on $P$
$$
W_4:=0
\subset
W_3:=
\frac{X_{r-1}+V_r^{**}}
{X_{r-1}+V_r^{***}}
\subset
W_2:=
\frac{X_{r-1}+V_r^{*}}
{X_{r-1}+V_r^{***}}
\subset
W_1:=P=\frac{X_{r-1}+V_r}
{X_{r-1}+V_r^{***}}
.
$$
The following exact sequences are useful for computing the structure of the module $P$. 
\begin{enumerate}
\item[(i)] \[0 \longrightarrow \frac{X_{r-1}+V_r^{*}}{X_{r-1}+V_r^{***}} \longrightarrow P=\frac{X_{r-1}+V_r}{X_{r-1}+V_r^{***}} \longrightarrow \frac{V_r}{X_{r-1}+V_r^{*}} \longrightarrow 0,\]
  
   \item[(ii)] \[
   0 \longrightarrow \frac{X_{r-1}+V_r^{**}}{X_{r-1}+V_r^{***}}  \longrightarrow \frac{X_{r-1}+V_r^{*}}{X_{r-1}+V_r^{***}}\longrightarrow \frac{X_{r-1}+V_r^{*}}{X_{r-1}+V_r^{**}} \longrightarrow 0.\]
\end{enumerate}
Let us define
\[
 J_0 \coloneqq \frac{\langle{X^r}\rangle+X_{r-1}+V_r^{*}}{X_{r-1}+V_r^{*}} =0\subset W_1/W_2= \frac{X_{r-1}+V_r}{X_{r-1}+V_r^{*}}  ,
\]
\[
J_2 \coloneqq \frac{\langle{\theta X^{r-p-1}}\rangle+X_{r-1}+V_r^{**}}{X_{r-1}+V_r^{**}}\subset W_2/W_3= \frac{X_{r-1}+V_r^{*}}{X_{r-1}+V_r^{**}} ,
\]
\[
J_4 \coloneqq 
\frac{\langle \theta^2 X^{r-2p-2}\rangle +X_{r-1}+V_r^{***}}{X_{r-1}+V_r^{***}} \subset W_3= \frac{X_{r-1}+V_r^{**}}{X_{r-1}+V_r^{***}} .
 \]
We set $$J_{2i+1} \coloneqq (W_{i+1}/W_{i+2})/J_{2i} ,$$ for $i=0,1,2$. 

  Let $U_j$ be the image of $\mathrm{ind}_{KZ}^GW_j$ under the map $\mathrm{ind}_{KZ}^GP \twoheadrightarrow \bar{\Theta}$, $F_{2i}$ be the image of $\mathrm{ind}_{KZ}^GJ_{2i}$  under the map $\mathrm{ind}_{KZ}^G(W_{i+1}/W_{i+2}) \twoheadrightarrow U_{i+1}/U_{i+2}$ and $F_{2i+1}=({U_{i+1}/U_{i+2}})/{F_{2i}}$, for $i=0,1,2$.
  Hence, we have the following commutative diagrams of $G$ modules 
  \[
\begin{tikzcd}
0 \arrow[r] & \mathrm{ind}_{KZ}^G W_2 \arrow[r] \arrow[two heads,d] & \mathrm{ind}_{KZ}^G P \arrow[r] \arrow[two heads,d] & \mathrm{ind}_{KZ}^G (W_1/W_2) \arrow[r] \arrow[two heads,d] & 0 \\
0 \arrow[r] & U_2 \arrow[r] & \bar{\Theta}_{k,a_p} \arrow[r] & U_1/U_2\arrow[r] & 0
\end{tikzcd}
\]

  \[
\begin{tikzcd}
0 \arrow[r] & \mathrm{ind}_{KZ}^G J_0 \arrow[r] \arrow[two heads,d] & \mathrm{ind}_{KZ}^G(W_1/W_2)\arrow[r] \arrow[two heads,d] & \mathrm{ind}_{KZ}^G J_1\arrow[r] \arrow[two heads,d] & 0 \\
0 \arrow[r] & F_0 \arrow[r] & U_1/U_2 \arrow[r] & F_1\arrow[r] & 0,
\end{tikzcd}
\]
  \[
\begin{tikzcd}
0 \arrow[r] & \mathrm{ind}_{KZ}^GW_3 \arrow[r] \arrow[two heads,d] & \mathrm{ind}_{KZ}^G W_2 \arrow[r] \arrow[two heads,d] & \mathrm{ind}_{KZ}^G (W_2/W_3) \arrow[r] \arrow[two heads,d] & 0 \\
0 \arrow[r] & U_3 \arrow[r] & U_2 \arrow[r] & U_2/U_3\arrow[r] & 0
\end{tikzcd}
\]
  
\[
\begin{tikzcd}
0 \arrow[r] & \mathrm{ind}_{KZ}^G J_2 \arrow[r] \arrow[two heads,d] & \mathrm{ind}_{KZ}^G({W_2}/{W_3}) \arrow[r] \arrow[two heads,d] & \mathrm{ind}_{KZ}^G J_3 \arrow[r] \arrow[two heads,d] & 0 \\
0 \arrow[r] & F_2 \arrow[r] & U_2/U_3 \arrow[r] & F_3\arrow[r] & 0,
\end{tikzcd}
\]

\[
\begin{tikzcd}
0 \arrow[r] & \mathrm{ind}_{KZ}^G J_4 \arrow[r] \arrow[two heads,d] & \mathrm{ind}_{KZ}^G W_3 \arrow[r] \arrow[two heads,d] & \mathrm{ind}_{KZ}^G J_5 \arrow[r] \arrow[two heads,d] & 0 \\
0 \arrow[r] & F_4 \arrow[r] & U_3\arrow[r] & F_5 \arrow[r] & 0.
\end{tikzcd}
\]  
The semi-simplification of $\bar{\Theta}_{k,a_p}$ is completely determined by the $\Gamma$-modules $F_0,F_1,F_2,F_3,F_4$ and $F_5$.

\section{The case $5 \le a \le p-1$}\label{sec: 5 to p-1}

Let $r \equiv a \pmod{p-1}$ with $5 \le a \le p-1$. We assume that $r \ge 3p+2$ unless specified.

   \begin{theorem}\label{thm:3.1}
Let $3p+2 \leq k - 2 = r \equiv a \pmod{p-1}$ with $5\le a \le p-1$, and assume that $v(a_p) = 2$.
\begin{enumerate}
\item If $p \nmid (r - a)(r-a+1)$, then $\overline{V}^{\,{ss}}_{k,a_p}
\cong \mu_\lambda \cdot \omega^{a-1} \oplus \mu_{\lambda^{-1}} \cdot \omega^2$ is reducible, where
$$\lambda =
\overline{\frac{(a-1)(a-2)}{(r-a)(r-a+1)}\cdot \frac{a_p}{p^2}}
\in \overline{\mathbb{F}}_p^{\times}.$$
\item If $p\mid r-a+1$, then $\overline{V}^{\,{ss}}_{k,a_p}
\cong \mathrm{ind}\!\left(\omega_2^{{a+p}}\right)$.
\item If $p \mid r - a$ and $p^2 \nmid (r-a)$, then $\overline{V}^{\,{ss}}_{k,a_p}
\cong \mu_\lambda \cdot \omega^{a} \oplus \mu_{\lambda^{-1}} \cdot \omega$ is reducible, where  $$\lambda = 
\overline{ \frac{a  p}{a-r} \cdot \frac{a_p}{p^2}}\in \overline{\mathbb{F}}_p^{\times}.$$
\item If $p^2 \mid r-a$, then  $\overline{V}^{\,{ss}}_{k,a_p}\cong \mathrm{ind}\!\left(\omega_2^{{a+1}}\right)$.
\end{enumerate}
\end{theorem}

\begin{lemma}\label{lem:3.2}
    Let $p \le r\equiv a \pmod{p-1}$ with $5 \le a \le p-1$. Then \[
X_{r-1}^{**}/X_{r-1}^{***} = 0.
\]
\end{lemma}
\begin{proof}
    By  \cite[Lemma 4.4]{GV22} we have $\frac{X_{r-1}^{**}}{X_{r-1}^{***}} =\frac{X_{r}^{**}}{X_{r}^{***}} $. Now by  \cite[Proposition 4.8]{GV22} we have $\frac{X_{r}^{**}}{X_{r}^{***}} =0$.
\end{proof}
\begin{rmk}\label{rmk:A}
    The proof of Lemma~\ref{lem:3.2} also works for $a=1$ and $a=4$.
\end{rmk}
\begin{lemma}\label{lem:3.3}
    Let { $2p+3 \le r \equiv a \pmod{p-1}$} with $5 \leq a \leq p-1$. Then we have \[
    X^{r-2}Y^2 \equiv  \frac{\binom{r-a+1}{2}}{\binom{a-1}{2}} \cdot \theta^2 X^{r-2p-2} \pmod {X_{r-1}+V_r^{(3)}}.
    \]
\end{lemma}
\begin{proof}
    We know that \begin{eqnarray*}
         \sum_{\substack{
    j \geq 0 \\
    j \equiv 2 \pmod{p-1}
}}
\binom{r }{j} \cdot X^{\,r-j} Y^j &\in& X_{r}, \text{ so}\\
\binom{r}{2} \cdot X^{r-2}Y^2+\sum_{\substack{
  2 < j <r-2 \\
    j \equiv 2 \pmod{p-1}
}}
\binom{r }{j} \cdot X^{\,r-j} Y^j &\in& X_{r-1}.
    \end{eqnarray*}
  By \cite[Lemma 2.14]{GV22} we have  \begin{eqnarray*}
 \sum_{\substack{
  j \geq 0\\
    j \equiv 2 \pmod{p-1}
}}
 \binom{j}{m}\binom{r }{j} \equiv \binom{r}{m} \binom{\left[a-m\right]}{\left[2-m\right]} + \binom{r}{m} \delta_{p-1,\left[2-m\right]} \pmod p.
 \end{eqnarray*}
 \begin{enumerate}
     \item[(i)] For $m=0$, we have \begin{eqnarray*}
     \sum_{\substack{
  2 < j <r-2 \\
    j \equiv 2 \pmod{p-1}
}}
\binom{r }{j} \equiv \binom{a}{2}- \binom{r}{2} \pmod{p}.
\end{eqnarray*}
\item[(ii)] For $m=1$, we have \[
\sum_{\substack{
  2 < j <r-2 \\
    j \equiv 2 \pmod{p-1}
}}
j\binom{r }{j} \equiv r(a-r) \pmod{p}.
\]
\item[(iii)] For $m=2$, we have \[
\sum_{\substack{
  2 < j <r-2 \\
    j \equiv 2 \pmod{p-1}
}}
\binom{j}{2}\binom{r }{j} \equiv 0 \pmod{p}.
\]
\end{enumerate}
Consider the element
\[
F(X,Y):=
\sum_{\substack{
2<j<r-2\\
j\equiv 2 \pmod{p-1}
}}
\binom{r}{j}X^{r-j}Y^j
+A X^{r-2p}Y^{2p}
+B X^{r-p-1}Y^{p+1}.
\]
We determine constants \(A,B\) such that $F(X,Y)\in V_r^{(3)}$. By the congruences above, we must solve: \begin{eqnarray*}
    A+B+\binom{a}{2}- \binom{r}{2} &=& 0\\
    B+r(a-r)&=& 0.
\end{eqnarray*}
The solutions are:
\[
A=\binom{r}{2}-\binom{a}{2}+r(a-r),\,\,B=-r(a-r).
\]
Hence
\begin{align}
\label{A,B}
-\binom{r}{2}X^{r-2}Y^2
&+\biggl(\binom{r}{2}-\binom{a}{2}+r(a-r)\biggr)X^{r-2p}Y^{2p}
-r(a-r)X^{r-p-1}Y^{p+1}
\in X_{r-1}+V_r^{(3)}.
\end{align}

Similarly, we know that \begin{eqnarray*}
         \sum_{\substack{
    j \geq 0 \\
    j \equiv 2 \pmod{p-1}
}}
\binom{r-1 }{j} \cdot X^{\,r-j} Y^j &\in& X_{r-1}, \text{ so}\\
\binom{r-1}{2} \cdot X^{r-2}Y^2+\sum_{\substack{
  2 < j <r-2 \\
    j \equiv 2 \pmod{p-1}
}}
\binom{r-1}{j} \cdot X^{\,r-j} Y^j &\in& X_{r-1}.
    \end{eqnarray*}
    By \cite[Lemma 2.14]{GV22} we have \begin{eqnarray*}
        \sum_{\substack{
  j \geq 0\\
    j \equiv 2 \pmod{p-1}
}}
\binom{j}{m}\binom{r-1 }{j} \equiv \binom{r-1}{m} \binom{\left[a-1-m\right]}{\left[2-m\right]} + \binom{r-1}{m} \delta_{p-1,\left[2-m\right]} \pmod p.
 \end{eqnarray*}
 \begin{enumerate}
     \item[(i)] For $m=0$, we have \begin{eqnarray*}
     \sum_{\substack{
  2 < j <r-2 \\
    j \equiv 2 \pmod{p-1}
}}
\binom{r-1}{j} \equiv \binom{a-1}{2}- \binom{r-1}{2} \pmod{p},
\end{eqnarray*}
\item[(ii)] For $m=1$, we have \[
\sum_{\substack{
  2 < j <r-2 \\
    j \equiv 2 \pmod{p-1}
}}
j\binom{r-1}{j} \equiv (r-1)(a-r) \pmod{p},
\]
\item[(iii)] For $m=2$, we have \[
\sum_{\substack{
  2 < j <r-2 \\
    j \equiv 2 \pmod{p-1}
}}
\binom{j}{2}\binom{r }{j} \equiv 0 \pmod{p}.
\]
\end{enumerate}
 Let us consider 
\[
G(X,Y)\coloneqq\sum_{\substack{
  2 < j <r-2 \\
    j \equiv 2 \pmod{p-1}
}}
\binom{r-1 }{j} \cdot X^{\,r-j} Y^j+A'X^{r-2p}Y^{2p}+B'X^{r-p-1}Y^{p+1}. 
\] 
Using these congruences, we can similarly find values of $A',B'$ such that $G(X,Y)$ lies in $V_r^{(3)}$. We have to solve the equations \begin{eqnarray*}
    A'+B'+\binom{a-1}{2}-\binom{r-1}{2} &=& 0\\
    B'+(r-1)(a-r) &=& 0.
\end{eqnarray*}
The solutions are:
\[
A'= \binom{r-1}{2}-\binom{a-1}{2}+(r-1)(a-r),\,\, B'=-(r-1)(a-r)
\]
Hence
\begin{align}
\label{A',B'}
-\binom{r-1}{2}X^{r-2}Y^2
&+\biggl(\binom{r-1}{2}-\binom{a-1}{2}+(r-1)(a-r)\biggr)X^{r-2p}Y^{2p} \nonumber \\
&-(r-1)(a-r)X^{r-p-1}Y^{p+1}
\in X_{r-1}+V_r^{(3)}.
\end{align}
Subtracting \eqref{A',B'} from \eqref{A,B}, we obtain
\[
-(r-1)X^{r-2}Y^2-(a-r)X^{r-p-1}Y^{p+1}
\in X_{r-1}+V_r^{(3)}.
\]
Multiplying by $(1-a)$ yields
\[
(r-1)(a-1)X^{r-2}Y^2
+(a-r)(a-1)X^{r-p-1}Y^{p+1}
\in X_{r-1}+V_r^{(3)}.
\]
Adding this to \eqref{A,B} above, we deduce that
\begin{eqnarray*}
    \biggl(-\binom{r}{2}+(r-1)(a-1)\biggr)X^{r-2}Y^2 
+\biggl(\binom{r}{2}-\binom{a}{2}+r(a-r)\biggr)X^{r-2p}Y^{2p} \qquad \qquad\\
 -(r-a+1)(a-r)X^{r-p-1}Y^{p+1}
\in X_{r-1}+V_r^{(3)}.
\end{eqnarray*}
Using the identity ${\theta^2}X^{r-2p-2}=X^{r-2}Y^2-2X^{r-p-1}Y^{p+1}+X^{r-2p}Y^{2p}
$, we obtain
\begin{align*}
&\biggl(
-\binom{r}{2}
+(r-1)(a-1)
-\binom{r}{2}
+\binom{a}{2}
-r(a-r)
\biggr)X^{r-2}Y^2 \\
&\qquad
+\frac{(a-r)(r-a+1)}{2}\theta^2X^{r-2p-2}
\in X_{r-1}+V_r^{(3)}.
\end{align*}
A straightforward simplification of the coefficients then gives
\[
\binom{a-1}{2}X^{r-2}Y^2
-\binom{r-a+1}{2}\theta^2X^{r-2p-2}
\in X_{r-1}+V_r^{(3)}.
\]
Since $\binom{a-1}{2}$ is invertible, it follows that
\[
X^{r-2}Y^2
\equiv
\frac{\binom{r-a+1}{2}}{\binom{a-1}{2}}
\,\theta^2X^{r-2p-2}
\pmod{X_{r-1}+V_r^{(3)}},
\]
proving the lemma.
\end{proof}

\begin{prop}\label{prop:A_1}
    Let $3p+2\le r \equiv a \pmod{p-1}$ with $5 \le a \le p-1$ and suppose that $v_p(a_p)=2$. Then\begin{enumerate}
        \item [(i)] $F_5=0$.
        \item[(ii)] If $p \mid (r-a)(r-a+1)$, then $U_3=0$.
    \end{enumerate}
\end{prop}
\begin{proof}
  (i) Let
\[
f_0:=
\left[
\mathrm{Id},
\frac{\theta^2X^{r-a-2p+2}Y^{a-4}}{p^2}
\right]
\in
\mathrm{ind}_{KZ}^G
\mathrm{Sym}^r\overline{\mathbb{Q}}_p^{\,2}.
\]
Since \(r\geq 3p+2\), we have
\[
T^-f_0\equiv 0 \pmod p.
\]
We now compute \(T^+f_0\). A straightforward calculation shows that
\[
T^+f_0
\equiv
\sum_{\lambda\in\mathbb F_p}
\left[
g^0_{1,\left[\lambda\right]},
(-[\lambda])^{a-4}X^{r-2}Y^2
\right]
\pmod p.
\]
By Lemma~\ref{lem:3.3}, the image of this element in
\(\mathrm{ind}_{KZ}^G P\) is
\[
\overline{\frac{(r-a)(r-a+1)}{(a-1)(a-2)}}
\sum_{\lambda\in\mathbb F_p}
\left[
g^0_{1,\left[\lambda \right]},
(-[\lambda])^{a-4}\theta^2X^{r-2p-2}
\right].
\]
Therefore,
\[
\overline{(T-a_p)f_0}
=
\overline{\frac{(r-a)(r-a+1)}{(a-1)(a-2)}}
\sum_{\lambda\in\mathbb F_p}
\left[
g^0_{1,\left[\lambda\right]},
(-[\lambda])^{a-4}\theta^2X^{r-2p-2}
\right]
-
\overline{\biggl(\frac{a_p}{p^2}\biggl)}
\left[
\mathrm{Id},
\theta^2X^{r-a-2p+2}Y^{a-4}
\right]
\]
in \(\mathrm{ind}_{KZ}^G P\).

The polynomial
\(
\theta^2X^{r-a-2p+2}Y^{a-4}
\)
by Lemma~\ref{lem:2.10}
maps to \(X^{p-a+3}\) in the cosocle \(J_5\). Consequently, the image of
\((T-a_p)f_0\) in \(\mathrm{ind}_{KZ}^GJ_5\) is
\[
-\frac{a_p}{p^2}
\left[
\mathrm{Id},
X^{p-a+3}
\right].
\]
Since
\(
\left[\mathrm{Id},X^{p-a+3}\right]
\)
generates \(\mathrm{ind}_{KZ}^GJ_5\) as a \(G\)-module, it follows that the image of \((T-a_p)f_0\) generates \(\mathrm{ind}_{KZ}^GJ_5\). In particular, $F_5=0$.

(ii) Suppose that $p \mid (r-a)(r-a+1)$. By Lemma~\ref{lem:3.3}, the polynomial $X^{r-2}Y^2$ maps to zero in $P$. Consequently,
$
T^{+}f_{0}=0 $ in $\mathrm{ind}_{KZ}^{G}P.
$
It follows that $(T-a_p)f_0$ is integral, and its image in $\mathrm{ind}_{KZ}^{G}P$ is $
-a_pf_0=- \frac{a_p}{p^2}[\mathrm{Id},\theta^2 X^{r-2p-a+2}Y^{a-4}].
$
By Proposition~\ref{prop:3.1} (iii) and Lemma~ \ref{lem:3.2}, the $\Gamma$-module
\[
W_3 \cong \frac{ V_r^{**}/V_r^{***} }{X_{r-1}^{**}/X_{r-1}^{***}} \cong V_r^{**}/V_r^{***}
\]
is a non-split extension of the weight $J_5$ by $J_4$. Furthermore, the polynomial
$
\theta^2 X^{r-2p-a+2}Y^{a-4}
$
maps to
$
X^{p-a+3}\neq 0
$
in $J_5$. Hence it generates $W_3$ as a $\Gamma$-module, and therefore
$
[\mathrm{Id},\theta^2 X^{r-2p-a+2}Y^{a-4}]
$
generates $\mathrm{ind}_{KZ}^{G}W_3$ as a $G$-module. It follows that the induced surjection
\[
\mathrm{ind}_{KZ}^{G}W_3 \twoheadrightarrow U_3
\]
is necessarily the zero map.
\end{proof}
\begin{prop}\label{prop:A_2}

    Let $3p+2 \le r\equiv a \pmod{p-1}$ with $5\le a\le p-1$. Assume that $p \nmid (r-a)(r-a+1)$ and $v_p(a_p)=2$, so that $u \coloneqq \frac{(a-1)(a-2)}{(r-a)(r-a+1)}\cdot \frac{a_p}{p^2}$ is a $p$-adic unit. Then 
    \begin{enumerate}
    \item[(i)] $F_4$ is a quotient of $\pi(a-4,\bar{u},\omega^2)$,
    \item[(ii)] $F_3$ is a quotient of $\pi(p-a+1,\bar{u}^{-1}, \omega^{a-1})$. 
\end{enumerate}
    
\end{prop}
\begin{proof}
     (i) Let
\[
f_0=\left[\mathrm{Id},\frac{\theta^2X^{r-2p-2}}{p^2}\right]
\in \mathrm{ind}_{KZ}^G \mathrm{Sym}^r \overline{\mathbb{Q}}_p^{\,2}.
\]
Since \(r\geq 3p+2\), we have
\[
T^-f_0\equiv 0 \pmod p.
\]
Moreover,
\[
T^+f_0
 = \sum_{\lambda\in\mathbb F_p}
 \left[g^0_{1,\left[\lambda \right]},X^{r-2}Y^2\right].
\]
Therefore
\[
(T-a_p)f_0
 \equiv
 \sum_{\lambda\in\mathbb F_p}
 \left[g^0_{1,\left[\lambda\right]},X^{r-2}Y^2\right]
 -
 \overline{\left(\frac{a_p}{p^2}\right)}
 \left[\mathrm{Id},\theta^2X^{r-2p-2}\right]
 \pmod p.
\]
By Lemma~\ref{lem:3.3}, the image of \((T-a_p)f_0\) in
\(\mathrm{ind}_{KZ}^G P\) is
\[
\overline{\frac{(r-a)(r-a+1)}{(a-1)(a-2)}}
\sum_{\lambda\in\mathbb F_p}
\left[g^0_{1,\left[\lambda \right]},\theta^2X^{r-2p-2}\right]
-
\overline{\left(\frac{a_p}{p^2}\right)}
\left[\mathrm{Id},\theta^2X^{r-2p-2}\right].
\]
By \cite[(4.2)]{Glo78} and \cite[Lemma~5.3]{Bre03b}, this further maps to
\[
\overline{\frac{(r-a)(r-a+1)}{(a-1)(a-2)}}
\,(T-\bar u)
\left[\mathrm{Id},X^{a-4}\right] \in \mathrm{ind}_{KZ}^G J_4.
\]
Since \(\left[\mathrm{Id},X^{a-4}\right]\) generates
\(\mathrm{ind}_{KZ}^G J_4\) as a \(G\)-module, the map $\mathrm{ind}_{KZ}^G J_4 \twoheadrightarrow F_4$ factors through $\pi(a-4,\bar{u}, \omega^2)$.

    (ii) Let us consider $f=f_0+f_1 \in \mathrm{ind}_{KZ}^G \mathrm{Sym}^r \bar{\mathbb{Q}}_p^2$, where 
    \begin{eqnarray*}
        f_1 &=& \sum_{\lambda \in \mathbb{F}_p} \left[ g^0_{1, \left[\lambda\right]} ,\frac{XY^{r-1}-X^p Y^{r-p}+X^{r-a+1}Y^{a-1}-X^{r-a-p+2}Y^{a+p-2}}{p^2} \right],\\
f_0 &=&\left[\mathrm{Id}, \frac{(p-1)}{p\cdot a_p} \sum_{\substack{
    0<j<r-1 \\
    j \equiv a-1 \pmod{p-1}
}}
\alpha_j \cdot X^{\,r-j} Y^j \right],
     \end{eqnarray*}
    where $\alpha_j$ are integers as in Lemma~\ref{lem:2.2} (applied for $r-1$). Both $T^-f_0$ and $T^+f_0$ die mod $p$ and $T^+f_1$ dies mod ${p}$ using  Lemma~\ref{lem:3.3}. We have
    \[T^-f_1=\left[\mathrm{Id}, \frac{(p-1)}{p} \sum_{\substack{
    0<j<r-1 \\
    j \equiv a-1 \pmod{p-1}
}}
\binom{r-1}{j} \cdot X^{\,r-j} Y^j \right].\]
 By \cite[Lemma 2.12]{GV22} we have $X^{r-j}Y^{j}-X^{p}Y^{r-p} \equiv (j-r)\theta X^{r-p-a+1}Y^{a-2} \pmod{V_r^{**}} $. We get 

\begin{eqnarray*}
T^-f_1-a_pf_0 &=&
\left[\mathrm{Id}, \frac{(p-1)}{p} \sum_{\substack{
    0<j<r-1 \\
    j \equiv a-1 \pmod{p-1}
}}
\left({\binom{r-1}{j}-\alpha_j}\right) \cdot X^{\,r-j} Y^j\right]\\
&=& \biggl[\mathrm{Id}, \frac{(p-1)}{p} \sum_{\substack{
    0<j<r-1 \\
    j \equiv a-1 \pmod{p-1}
}}
\left({\binom{r-1}{j}-\alpha_j}\right) \cdot (X^{\,r-j} Y^j-X^{p}Y^{r-p})\\
 && \quad +\frac{(p-1)}{p} \sum_{\substack{
    0<j<r-1 \\
    j \equiv a-1 \pmod{p-1}
}}
\left({\binom{r-1}{j}-\alpha_j}\right) \cdot X^{p} Y^{r-p}\biggl]\\
&\equiv& \biggl[\mathrm{Id}, \frac{(p-1)}{p} \sum_{\substack{
    0<j<r-1 \\
    j \equiv a-1 \pmod{p-1}
}}
(j-r)\left({\binom{r-1}{j}-\alpha_j}\right) \cdot \theta X^{r-p-a+1}Y^{a-2}\\
 && \quad +
\frac{(p-1)}{p} \sum_{\substack{
    0<j<r-1 \\
    j \equiv a-1 \pmod{p-1}
}}
\left({\binom{r-1}{j}-\alpha_j}\right) \cdot X^{p} Y^{r-p}\biggl] \pmod{V_r^{**}}\\
&=& \biggl[\mathrm{Id},\frac{(r-a)(r-a+1)}{(a-1)(a-2)}\cdot\theta X^{r-p-a+1}Y^{a-2}\\
 && \quad +\frac{(p-1)}{p} \sum_{\substack{
    0<j<r-1 \\
    j \equiv a-1 \pmod{p-1}
}}
\left({\binom{r-1}{j}-\alpha_j}\right) \cdot X^{p} Y^{r-p}\biggl],
\end{eqnarray*}
by Lemma~\ref{lem:2.1} with $i=0,1$ and Lemma~\ref{lem:2.2}.
On the other hand, we have \[
-a_pf_1 = \frac{a_p}{p^2}\sum_{\lambda \in \mathbb{F}_p} \left[ g^0_{1, \lambda} ,{\theta Y^{r-p-1}-\theta X^{r-p-a+1}Y^{a-2}} \right].
\]
By ~\ref{lem:2.10}, the polynomial
\(
\theta X^{r-a-p+1}Y^{a-2}
\)
maps to \(X^{p-a+1}\) in the cosocle \(J_3\) and $X^p Y^{r-p}= \theta Y^{r-p-1}+XY^{r-1}$. Thus $(T-a_p)(-f_0-f_1)$ is integral and maps to \[
\overline {\frac{a_p}{p^2}}\sum_{\lambda\in \mathbb{F_p}}\left[ g^0_{1,\lambda},X^{p-a+1}\right]-\left[\mathrm{Id},\overline{\frac{(r-a)(r-a+1)}{(a-1)(a-2)}}\cdot X^{p-a+1}\right]
=\overline {\frac{a_p}{p^2}}(T-\bar{u}^{-1})\left[\mathrm{Id}, X^{p-a+1}\right]\in \mathrm{ind}_{KZ}^G J_3.\]
Hence, $F_3$ is a quotient of $\pi(p-a+1,\bar{u}^{-1}, \omega^{a-1}).$
\end{proof}

\begin{prop}\label{prop:A_3}
   Let $3p+2 \le r\equiv a \pmod{p-1}$  with $5\le a\le p-1$. Assume that $p \nmid (r-a)$ and $v_p(a_p)=2$. Then \begin{itemize}
       \item [(i)] $F_1=0$.
       \item [(ii)] $F_2=0$.
   \end{itemize}
\end{prop}
\begin{proof}
  (i)  Let us consider $f=f_0+f_1+f_{-1} \in  \mathrm{ind}_{KZ}^G \mathrm{Sym}^r \bar{\mathbb{Q}}_p^2$, where 
    \begin{eqnarray*}
        f_1 &=& \sum_{\lambda\in {\mathbb{F}_p}}\left[g^0_{1,\left[\lambda\right]}, \frac{Y^r-X^{r-a}Y^a}{p} \right]\\
     f_0 &=&\left[\mathrm{Id}, \frac{(p-1)}{p\cdot a_p}\sum_{\substack{
    0<j<r \\
    j \equiv a \pmod{p-1}
}}
\alpha_j \cdot X^{\,r-j} Y^j\right]\\
 %\color{red}
f_{-1} &=& %\color{red} 
\delta_{a,p-1}\biggl[\alpha, \frac{X^r-X^{a}Y^{r-a}}{p} \biggl],
   \end{eqnarray*}
    where $\alpha_j$ are integers as in Lemma~\ref{lem:2.2}. Then $-a_pf_1$, $T^+f_1$, $T^-f_0$, $T^+f_0$, $ %\color{red} 
    -a_pf_{-1}$ and %$\color{red} 
    $T^-f_{-1}$ die mod $p$. %{\color{red}
    We have \begin{eqnarray*}
        T^+f_{-1} = \delta_{a,p-1} \biggl[\mathrm{Id}, \frac{(1-p)X^r}{p}\biggl] \pmod{X_{r-1}}.
    \end{eqnarray*}  We compute 
\begin{eqnarray*}
        T^-f_1-a_pf_0&=&\biggl[\mathrm{Id},\delta_{a,p-1}\frac{p-1}{p}X^r+ \frac{(p-1)}{p} \sum_{\substack{
    0<j<r \\
    j \equiv a \pmod{p-1}
}}
\left({\binom{r}{j}-\alpha_j}\right) \cdot X^{\,r-j} Y^j\biggl].
\end{eqnarray*}
Using $X^{r-j}Y^{j} \equiv X^{r-a}Y^a \bmod {V_r^{*}}$, we get that \begin{eqnarray*}
 T^-f_1-a_pf_0+T^+f_{-1}&=&\biggl[\mathrm{Id}, \frac{(p-1)}{p} \sum_{\substack{
    0<j<r \\
    j \equiv a \pmod{p-1}
}}
\left({\binom{r}{j}-\alpha_j}\right) \cdot X^{\,r-a} Y^a\biggl] \\
&=&\biggl[\mathrm{Id}, \frac{(r-a)}{a} \cdot X^{r-a}Y^a \biggl],
\end{eqnarray*}
by Lemma~\ref{lem:2.1} with $i=0$ and Lemma~\ref{lem:2.2}.
So, $(T-a_p)f$ is integral and maps to $\frac{r-a}{a} \biggl[\mathrm{Id},X^{p-a-1}\biggl]$ in $\mathrm{ind}_{KZ}^G J_1$ by Lemma~\ref{lem:2.10}. This is a generator as a $G$-module as $p \nmid (r-a)$. Hence, $F_1=0$.

     (ii) By \cite[Proposition 4.18]{GV22} if $r \equiv r_0 \pmod{p}$ and $r_0 \neq a$, then we have $\frac{X_{r-1}^{(1)}}{X_{r-1}^{(2)}} \cong V_{a-2} \otimes D$. 
\end{proof}
\begin{prop}\label{prop:A_4}
Let $3p
+2 \le r\equiv a \pmod{p-1}$ with $5\le a\le p-1$. Assume that $v_p(a_p)=2$. \begin{itemize}
       \item [(i)]If $p \mid (r-a)$, then $F_3=0$.
       \item [(ii)] If $p^2 \mid r-a$, then $F_2=0$ as well.
   \end{itemize}
\end{prop}
\begin{proof}
 (i)   Let us consider $f_0,f_1 \in \mathrm{ind}_{KZ}^G \mathrm{Sym}^r \bar{\mathbb{Q}}_p^2$ given by 
    \[
    f_1=\sum_{\lambda \in \mathbb{F}_p^{\times}}\biggl[g^0_{1,[\lambda]}, \frac{[\lambda]^{p-2}(Y^r-X^{r-a}Y^a)}{p}\biggl]+\biggl[g^0_{1,0}, \frac{r(1-p)(XY^{r-1}-X^{r-a+1}Y^{a-1})}{p^2}\biggl],
    \]
    \[
    f_0= \biggl[\mathrm{Id},\frac{(p-1)}{p\cdot a_p}\sum_{\substack{
    0<j<r-1 \\
    j \equiv a-1 \pmod{p-1}
}}
\alpha_j \cdot X^{\,r-j} Y^j\biggl],
    \]
    where $\alpha_j$ are integers as in Lemma~\ref{lem:2.5}. Then $T^+f_1$, $T^+f_0$ and $T^-f_0$ die mod $p$. We have \[
    -a_pf_1=-\frac {r \cdot a_p}{p^2}\biggl[g^0_{1,0}, \theta X^{r-p-a+1}Y^{a-2}-\theta Y^{r-p-1
    }\biggl] \pmod{V_r^{**}},
    \]
    by \cite[Lemma 2.12]{GV22}.
    We compute
\begin{eqnarray}
\label{T-f1} 
T^-f_1=
\left[\mathrm{Id},
\frac{(p-1)}{p}
\sum_{\substack{
0<j<r-1 \\
j \equiv a-1 \pmod{p-1}
}}
\binom{r}{j}\,X^{r-j}Y^j
\right].
\end{eqnarray}
    Thus, \[T^-f_1-a_pf_0=\biggl[\mathrm{Id}, \frac{(p-1)}{p}  \sum_{\substack{
    0<j<r-1 \\
    j \equiv a-1 \pmod{p-1}
}}
\left(\binom{r}{j}-\alpha_j\right) \cdot X^{\,r-j} Y^j\biggl].\]
 We can write the above expression as 
 \begin{eqnarray*}
     T^-f_1-a_pf_0=\biggl[\mathrm{Id}, \frac{(p-1)}{p}  \sum_{\substack{
    0<j<r-1 \\
    j \equiv a-1 \pmod{p-1}
}}
\left(\binom{r}{j}-\alpha_j\right) \cdot X^{\,r-j} Y^j- \frac{(p-1)(a-r)}{p} \cdot X^pY^{r-p}\\
+ \frac{(p-1)(a-r)}{p} \cdot \theta Y^{r-p-1}\biggl] \pmod{X_{r-1}}.
\end{eqnarray*}
 We claim that \[F(X,Y)=\frac{(p-1)}{p}  \sum_{\substack{
    0<j<r-1 \\
    j \equiv a-1 \pmod{p-1}
}}
\left(\binom{r}{j}-\alpha_j\right) \cdot X^{\,r-j} Y^j- \frac{(p-1)(a-r)}{p} \cdot X^pY^{r-p} \in V_r^{**}.\]
%\textit{Proof of the claim.} 
Indeed, as in the proof of Lemma~\ref{lem:2.4}, 
\[
\frac{1}{p}   \sum_{\substack{
    0<j<r-1 \\
    j \equiv a-1 \pmod{p-1}
}}
\left(\binom{r}{j}-\alpha_j\right)- \frac{(a-r)}{p} \equiv 0 \pmod{p}
\]
and
\[\frac{1}{p}   \sum_{\substack{
    0<j<r-1 \\
    j \equiv a-1 \pmod{p-1}
}}
j \cdot\left(\binom{r}{j}-\alpha_j\right)- \frac{(a-r)(r-p)}{p} \equiv 0 \pmod{p}. \] 
The claim follows from \cite[Lemma 2.7]{GV22}. 
So, mod $X_{r-1} +V_r^{**}$, we have \[
(T-a_p)(f_0+f_1)= -\frac{a \cdot a_p}{p^2}\biggl[g^0_{1,0},\theta X^{r-p-a+1}Y^{a-2}-\theta Y^{r-p-1
    }\biggl]+\biggl[\mathrm{Id},\frac{(p-1)(a-r)}{p} \cdot \theta Y^{r-p-1}\biggl].
\]
Hence, $(T-a_p)(f_0+f_1)$ maps to $-\frac{a\cdot a_p}{p^2}\biggl[g^0_{1,0},X^{p-a+1}\biggl]$ in $\mathrm{ind}_{KZ}^GJ_3$ by Lemma~\ref{lem:2.10}. Since $\biggl[g^0_{1,0},X^{p-a+1}\biggl]$ generates $\mathrm{ind}_{KZ}^GJ_3$ as a $G$-module, it follows that the image of $(T-a_p)(f_0+f_1)$ generates $\mathrm{ind}_{KZ}^GJ_3$. In particular, $F_3=0$.

(ii) Let $f_0, f_1 \in \mathrm{ind}_{KZ}^G \mathrm{Sym}^r \overline{\mathbb{Q}}_p^{\,2}$ be as in 
part (i), except that this time let
%\[
%f_1=\sum_{\lambda \in \mathbb{F}_p^{\times}}
%\left[g^0_{1,[\lambda]}, 
%\frac{[\lambda]^{p-2}\bigl(Y^r-X^{r-a}Y^a\bigr)}{p}
%\right]
%+
%\left[g^0_{1,0}, 
%\frac{-r\bigl(XY^{r-1}-X^{r-a+1}Y^{a-1}\bigr)}{p^2}
%\right],
%\]
%
%\[
%f_0=
%\left[\mathrm{Id},
%\frac{p-1}{p a_p}
%\sum_{\substack{
%0<j<r-1 \\
%j \equiv a-1 \pmod{p-1}
%}}
%\alpha_j \cdot X^{r-j}Y^j
%\right],
%\]
%where 
$\alpha_j$ be integers as in Lemma~\ref{lem:2.4}. As above 
$T^+f_1$, $T^+f_0$ and $T^-f_0$ die mod $p$.
Moreover,
\[
-a_pf_1
=
-\frac{r\,a_p}{p^2}
\biggl[g^0_{1,0},
XY^{r-1}-X^{r-a+1}Y^{a-1}
\biggl]
\]
and $T^-f_1$ is as in \eqref{T-f1}.
Hence,
\[
T^-f_1-a_pf_0=
\left[\mathrm{Id},
\frac{p-1}{p}
\sum_{\substack{
0<j<r-1 \\
j \equiv a-1 \pmod{p-1}
}}
\left(\binom{r}{j}-\alpha_j\right)
X^{r-j}Y^j
\right].
\]
 Since $\alpha_j \equiv \binom{r}{j} \pmod{p^2}$, it follows that $T^-f_1-a_pf_0 \equiv 0 \pmod p$.
Therefore,
\[
(T-a_p)(f_0+f_1)
=
-\frac{a\,a_p}{p^2}
\biggl[g^0_{1,0},
XY^{r-1}-X^{r-a+1}Y^{a-1}
\biggl].
\]
Let $v$ be the image of $-\frac{a\,a_p}{p^2}(XY^{r-1}-X^{r-a+1}Y^{a-1})$ in $V_r^{*}/V_r^{**}$. Then the reduction $\overline {(T-a_p)f}$ maps to $\biggl[g^0_{1,0},v\biggl] \in \mathrm{ind}_{KZ}^G (V_r^{*}/V_r^{**})$. Now as in part $(i)$, $v$ maps to $-\overline{\frac{a a_p}{p^2}} \cdot X^{p-a+1}\in J_3$. As the short exact sequence  
\[
0 \longrightarrow V_{[a-2]}\otimes D
\longrightarrow \frac{V_r^{*}}{V_r^{**}}
\longrightarrow V_{p-1-[a-2]}\otimes D^{a-1}
\longrightarrow 0
\]
is non-split, $v$ generates the whole module $V_r^{*}/V_r^{**}$ as a $G$-module, so the element $\biggl[g^0_{1,0},v\biggl]$ generates $\mathrm{ind}_{KZ}^G (V_r^{*}/V_r^{**})$. Thus $F_2=0$ and $F_3=0$.
\end{proof}
\begin{lemma}\label{lem:3.8}
Let $r \equiv a \pmod{p-1}$ and $p \mid r-a$. Then
    \begin{eqnarray*}
    \sum_{\substack{1<j\leq {r-a+1} \\
j \equiv 1 \pmod{p-1}
}}
 \binom{r}{j} &\equiv& a-r \pmod{p^2}\\
     p \sum_{\substack{1<j\leq {r-a+1} \\
j \equiv 1 \pmod{p-1}
}}
j \binom{r}{j} &\equiv& 0 \pmod{p^3}\\
 p^2 \sum_{\substack{
1<j\leq {r-a+1} \\
j \equiv 1 \pmod{p-1}
}}
\binom{j}{2}\binom{r}{j} &\equiv& 0 \pmod{p^4}\\
 p^3 \sum_{\substack{
1<j\leq {r-a+1} \\
j \equiv 1 \pmod{p-1}
}}
\binom{j}{3}\binom{r}{j} &\equiv& 0 \pmod{p^5}.
 \end{eqnarray*}
\end{lemma}
\begin{proof}
    Similar to the proofs of the lemmas proved in Section~\ref{roots of 1}.
\end{proof}
\begin{prop}\label{prop:A_5}
    Let $3p+2 \le r\equiv a \pmod{p-1}$ with $5\le a\le p-1$. Assume that $p \mid (r-a)$ but $p^2 \nmid (r-a)$ and $v_p(a_p)=2$, so that $u' \coloneqq\frac{a \cdot a_p}{(a-r) p}$ is a $p$-adic unit. Then \begin{enumerate}
  \item[(i)] $F_1$ is a quotient of $\pi(p-a-1,\bar{u'}^{-1}, \omega^{a})$. 
    \item[(ii)] $F_2$ is a quotient of $\pi(a-2,\bar{u'},\omega)$.
\end{enumerate}
\end{prop}
\begin{proof}
  (i)  Let us consider $f_0,f_1, %{\color{red}}
  f_{-1} 
  \in \mathrm{ind}_{KZ}^G \mathrm{Sym}^r \bar{\mathbb{Q}}_p^2$ given by 
    \begin{eqnarray*}
    f_1 &=&\sum_{\lambda\in {\mathbb{F}^{\times}_p}}\biggl[g^0_{1,\left[\lambda \right]}, \frac{Y^r-X^{r-a}Y^a}{p^2}\biggl]+(1-p)\cdot\biggl[g^0_{1,0}, \frac{Y^r-X^{r-a}Y^a}{p^2}\biggl]\\
    f_0 &=&\biggl[\mathrm{Id}, \frac{(p-1)}{p^2\cdot a_p}\sum_{\substack{
    0<j<r \\
    j \equiv a \pmod{p-1}
}}
\alpha_j \cdot X^{\,r-j} Y^j\biggl]\\
f_{-1} &=& %\color{red}
\frac{\delta_{a,p-1}(1-p)}{p^2}\biggl[\alpha, X^r-X^{a}Y^{r-a} \biggl],
    \end{eqnarray*}
    where $\alpha_j$ are as in Lemma~\ref{lem:2.3}. 
    Then $T^+f_1$ dies mod $X_{r-1}$ and $T^-f_0$, $T^+f_0$, $T^-f_{-1}$ die mod $p$.  We have \begin{eqnarray*}
        T^+f_{-1}=\frac{\delta_{a,p-1}(1-p)}{p^2} \biggl[\mathrm{Id},   X^r\biggl] \pmod{X_{r-1}}.
    \end{eqnarray*}We compute \begin{eqnarray*}
    T^-f_1-a_pf_0 &=&\biggl[\mathrm{Id},\frac{\delta_{a,p-1}(p-1)}{p^2}X^r
    + \frac{(p-1)}{p^2}\sum_{\substack{
    0<j<r \\
    j \equiv a \pmod{p-1}}}\left(\binom{r}{j}-\alpha_j\right)\cdot X^{r-j}Y^j\biggl] \\
   T^-f_1-a_pf_0 +T^+f_{-1}&=& \biggl[\mathrm{Id}, \frac{(p-1)}{p^2}\cdot \frac{(a-r)\cdot p}{a} X^{r-a}Y^a\biggl] \pmod{V_r^{*}},
     \end{eqnarray*}
  where the second equality follows from Remark~\ref{rmk:2.1}. Also \[
  -a_pf_1= \frac{a_p}{p^2} \sum_{\lambda\in {\mathbb{F}_p}}\biggl[g^0_{1,\left[\lambda \right]}, X^{r-a}Y^a\biggl] \pmod{X_{r-1}}.
  \]
Hence, we have \begin{eqnarray*}
    (T-a_p)(f_0+f_1+f_{-1}) &=& \frac{a_p}{p^2} \sum_{\lambda\in {\mathbb{F}_p}}\biggl[g^0_{1,\left[\lambda \right]}, X^{r-a}Y^a\biggl]-\frac{(a-r)}{a \cdot p}\biggl[\mathrm{Id}, X^{r-a}Y^a\biggl]\\
    && \qquad +\frac{a_p \delta_{a,p-1}}{p^2}\biggl[ \alpha,X^{a}Y^{r-a}\biggl]\pmod{X_{r-1}+V_r^{*}}.
    \end{eqnarray*}
 Thus $(T-a_p)(f_0+f_1+f_{-1})$ is integral and, by Lemma~\ref{lem:2.10}, maps to \[
\overline {\frac{a_p}{p^2}}\sum_{\lambda\in \mathbb{F}_p}\left[ g^0_{1,\left[\lambda \right]},X^{p-a-1}\right] +\overline{\frac{a_p \delta_{a,p-1}}{p^2}}\biggl[ \alpha,X^{p-a-1}\biggl]-\overline{\frac{(a-r)}{a \cdot p}}\cdot\left[\mathrm{Id},X^{p-a-1}\right]=\overline {\frac{a_p}{p^2}}(T-\bar{u'}^{-1})\left[\mathrm{Id}, X^{p-a-1}\right]\in \mathrm{ind}_{KZ}^G J_1,\]
where we have used that $X^{r-p+1}Y^{p-1} \equiv X^{p-1}Y^{r-p+1} \pmod{V_r^{*}}$.
Hence, $F_1$ is a quotient of $\pi(p-a-1,\bar{u}^{-1}, \omega^{a}).$

  (ii)   Let $f_0,f_1,f_2 \in \mathrm{ind}_{KZ}^G \mathrm{Sym}^r \overline{\mathbb{Q}}_p^{\,2}$ be defined by \begin{eqnarray*}
        f_0 &=& \bigg[\mathrm{Id}, \frac{X^{r-1}Y-X^{r-p}Y^p}{p^2}\bigg]\\
   f_1 &=& -\frac{1}{ap \cdot a_p} \bigg(\sum_{\lambda \in \mathbb{F}_p^{\times}} \Phi_{g^0_{1,\left[\lambda\right]}}+(1-p)\Phi_{g^0_{1,0}}\bigg)\\
    f_2&=&\frac{1}{ap} \bigg(\frac{1}{1-p} \sum_{\lambda \in \mathbb{F}_p^{\times}} \Psi_{g^0_{1,\left[\lambda\right]}}+\Psi_{g^0_{1,0}}\bigg),
\end{eqnarray*}
    where 
\begin{eqnarray*}
    \Phi_{g} &=&
    \bigg[g,\sum_{\substack{
1<j\leq {r-a+1} \\
j \equiv 1 \pmod{p-1}
}}
\binom{r}{j} X^{r-j}Y^j-S_0\cdot X^{r-p}Y^{p} \bigg] \text{ with}\,\,  {S_0 =\sum_{\substack{
1<j\leq {r-a+1} \\
j \equiv 1 \pmod{p-1}
}}
\binom{r}{j} }\\
\Psi_g &=&\sum_{\mu \in \mathbb{F}_p^{\times}} \left[\mu\right]^{-a+1}\biggl( \left[gg^0_{1,\left[\mu\right]}, Y^r-X^{r-a}Y^a\right]\biggl) .
\end{eqnarray*}
Note that $S_0 \equiv a-r \pmod{p^2}$ by Lemma~\ref{lem:3.8}. %by \cite[Lemma 2.14]{GV22}.
   Then $T^-f_0 \equiv 0 \pmod p$. Computing $T^+f_0$, we have 
    \[
  T^+f_0= \frac{1-p}{p} \sum_{\lambda \in {\mathbb{F}_p}} \bigg[g^0_{1,\left[\lambda\right]},X^{r-1}Y\bigg] \pmod{X_{r-1}}
  \]
  Since $r\ge 3p+2$ and $v_p(pa_p)=3$ we have $T^-f_1 \equiv 0 \pmod p$. Now computing $T^+{\phi_g}$ we have  \begin{eqnarray*}
        T^+{\phi_g}&= &\sum_{\mu\in {\mathbb{F}_p}} \bigg[gg^{0}_{1,\left[\mu \right]},\sum_{\substack{
1<j\leq {r-a+1} \\
j \equiv 1 \pmod{p-1}
}}
\binom{r}{j} X^{r-j}({-\left[\mu \right]X+pY})^j-S_0\cdot X^{r-p}{(-\left[\mu \right]X+pY)}^{p} \bigg]\\
&=& \sum_{\mu\in \mathbb{F}_p^{\times}} \bigg[gg^{0}_{1,\left[\mu \right]},\sum_{i=0}^{r} {p^{i}}(-\left[\mu \right])^{1-i}
\sum_{\substack{
1<j\leq {r-a+1} \\
j \equiv 1 \pmod{p-1}
}}
\binom{j}{i}\binom{r}{j} X^{r-i}Y^i\\
&& \qquad+{\left[\mu \right]} {S_0}\cdot X^r-S_0\cdot {p^2}\cdot X^{r-1}Y \bigg]
+\,O(p^4).
\end{eqnarray*}
By definition of $S_0$, the coefficient of 
$X^r$ %$\bigg[gg^{0}_{1,\left[\mu \right]},X^r\biggl]$ 
is $0$. The coefficients of
$X^{r-1}Y$, $X^{r-2}Y^2$, $X^{r-3}Y^3$ are $0$ %$\left[gg^{0}_{1,\left[\mu \right]}, X^{r-1}Y\right]$, $\left[gg^{0}_{1,\left[\mu \right]}, X^{r-2}Y^2\right]$, $\left[gg^{0}_{1,\left[\mu \right]}, X^{r-3}Y^3\right]$ in $T^+\phi_g$, for $\mu \neq 0$, are 
 mod $p^3$, $p^4$, $p^5$ respectively, by Lemma~\ref{lem:3.8}.
 Thus, $T^+f_1 \equiv 0 \pmod{X_{r-1}}$.
 So we have \begin{eqnarray*}
     (T-a_p)f_1 &=&\frac{1}{ap} \biggl(\sum_{\lambda \in {\mathbb{F}_p}^{\times}}\bigg[g^0_{1,\left[\lambda\right]},\sum_{\substack{
1<j\leq {r-a+1} \\
j \equiv 1 \pmod{p-1}
}}
\binom{r}{j} X^{r-j}Y^j-S_0\cdot X^{r-p}Y^{p} \bigg]\\
&&\qquad +(1-p)\bigg[g^0_{1,0},\sum_{\substack{
1<j\leq {r-a+1} \\
j \equiv 1 \pmod{p-1}
}}
\binom{r}{j} X^{r-j}Y^j-S_0\cdot X^{r-p}Y^{p} \bigg]\biggl) \pmod{X_{r-1}}.
\end{eqnarray*}

Also, we have $(T-a_p) \left[g, Y^r-X^{r-a}Y^a\right]= \left[g\alpha,Y^r\right]+\, O(p^2)$. Hence,
 \begin{eqnarray*}
(T-a_p)\Psi_g &=& \sum_{\mu \in \mathbb{F}_p^{\times}} \left[\mu\right]^{-a+1}\left[gg^0_{1,\left[\mu\right]}\alpha,Y^r\right]+ \,O(p^2)\\
&=& \sum_{\mu \in \mathbb{F}_p^{\times}}\left[\mu\right]^{-a+1} \left[g, (\left[\mu \right]X+Y)^r\right]+O(p^2)\\
&=& (p-1) \biggl[g, \sum_{\substack{
0<j\leq {r-a+1} \\
j \equiv 1 \pmod{p-1}
}}
 \binom{r}{j} X^{r-j} Y^{j} \biggl]+\,O(p^2).
  \end{eqnarray*}
  Using this computation, we have 
  \begin{eqnarray*}
(T-a_p)f_2%&=&\frac{1}{ap(1-p)} \sum_{\lambda \in {\mathbb{F_p}^{\times}}} (p-1) \biggl[g^0_{1,\left[ \lambda \right]}, \sum_{\substack{
%0<j\leq {r-a+1} \\
%j \equiv 1 \pmod{p-1}
%}}
 %\binom{r}{j} X^{r-j} Y^{j} \biggl]\\
 %&& \quad + \frac{p-1}{ap} %\biggl[g^0_{1,0}, \sum_{\substack{
%0<j\leq {r-a+1} \\
%j \equiv 1 \pmod{p-1}
%}}
 &=& \frac{-1}{ap} \biggl(\sum_{\lambda \in \mathbb{F}_p^{\times}} \biggl[g^0_{1,\left[ \lambda \right]}, \sum_{\substack{
0<j\leq {r-a+1} \\
j \equiv 1 \pmod{p-1}
}}
 \binom{r}{j} X^{r-j} Y^{j} \biggl]\\
 && \quad + (1-p) \biggl[g^0_{1,0}, \sum_{\substack{
0<j\leq {r-a+1} \\
j \equiv 1 \pmod{p-1}
}}
 \binom{r}{j} X^{r-j} Y^{j} \biggl] \biggl).
  \end{eqnarray*}
  We can write \[ 
  T^+f_0=
\frac{1 - p}{ap}
\left[
g^{0}_{1,0},
\binom{r}{1}
X^{r-1}Y
\right]
+
\frac{1}{ap}
\sum_{\lambda \in \mathbb{F}_p^\times}
\left[
g^{0}_{1,[\lambda]},
\binom{r}{1}
X^{r-1}Y
\right]
+
\frac{a - 
r}{ap}
\sum_{\lambda \in \mathbb{F}_p}
\left[
g^{0}_{1,[\lambda]},
X^{r-1}Y
\right] \pmod{X_{r-1}}.
\]
Now combining these together, we have 
\[
(T-a_p)(f_0+f_1+f_2)= \frac{a-r}{ap}\sum_{\lambda \in \mathbb{F}_p}
\biggl[
g^{0}_{1,[\lambda]},
X^{r-1}Y-X^{r-p}Y^p
\biggl]- \frac{a_p}{p^2} 
\biggl[\mathrm{Id},X^{r-1}Y-X^{r-p}Y^p
\biggl].
\]
Moreover, \(\theta X^{r-p-1}\) is identified with \(X^{a-2}\) in $J_2=V_{[a-2]}\otimes D$.
So the image of \((T-a_p)f\) in $\operatorname{ind}_{KZ}^{G} J_2$
is given by

\[
\overline{\frac{a-r}{ap}}\sum_{\lambda \in \mathbb{F}_p}
\biggl[
g^{0}_{1,[\lambda]},
X^{a-2}
\biggl]- \frac{a_p}{p^2} 
\biggl[\mathrm{Id},X^{a-2}
\biggl].
\] 
Since \(\left[\mathrm{Id},X^{a-2}\right]\) generates
\(\mathrm{ind}_{KZ}^G J_2\) as a \(G\)-module, the map $\mathrm{ind}_{KZ}^G J_2\twoheadrightarrow F_2$ factors through $\pi(a-2,\bar{u}, \omega)$.
Hence, $F_2$ is a quotient of  $\pi(a-2,\bar{u'},\omega)$.
\end{proof}

\begin{proof}[Proof of Theorem~\ref{thm:3.1}] 
The above propositions may be summarized by the following table:

\begin{table}[h]
\centering
\begin{tabular}{|>{\centering\arraybackslash}m{5cm}|
                >{\centering\arraybackslash}m{1.3cm}|
                >{\centering\arraybackslash}m{1.3cm}|
                >{\centering\arraybackslash}m{1.3cm}|
                >{\centering\arraybackslash}m{1.3cm}|
                >{\centering\arraybackslash}m{1.3cm}|}
\hline
& $F_5$ & $F_4$ & $F_3$ & $F_2$ & $F_1$ \\ \hline
$p\nmid (r-a)(r-a+1)$
& & $\checkmark$& $\checkmark$ & & \\ \hline

$p\mid (r-a+1)$
& & &$\checkmark$ & & \\ \hline

$p\mid (r-a),\; p^2\nmid (r-a)$
& & & &$\checkmark$ & $\checkmark$\\ \hline

$p^2\mid (r-a)$
& & & &  &$\checkmark$ \\ \hline

\end{tabular}
\end{table} In particular, we have:
\begin{itemize}
    \item [(1)] If $p \nmid (r-a)(r-a+1)$, then by Proposition~\ref{prop:A_1} and Proposition~\ref{prop:A_3} we get $\overline{\Theta}^{\,{ss}} \cong (F_3 \oplus F_4)^{\,{ss}}$. Moreover, Proposition~\ref{prop:A_2} tells us that we have surjections $\pi(a-4,\bar{u},\omega^2) \twoheadrightarrow F_4$ and $\pi(p-a+1,\bar{u}^{-1},\omega^{a-1}) \twoheadrightarrow F_3$. Now as $\bar\Theta$ lies in image of mod $p$ LLC, we have $\overline{\Theta}^{\,{ss}} \cong \pi(a-4,\bar{u},\omega^2)^{\,{ss}} \oplus \pi(p-a+1,\bar{u}^{-1},\omega^{a-1})^{\,{ss}}$.
    \item [(2)] If $p \mid r-a+1$, then by Proposition~\ref{prop:A_1} and Proposition~\ref{prop:A_3}, we have $\overline{\Theta} \cong F_3 \cong \pi(p-a+1,0,\omega^{a-1})$.
    \item [(3)] If $p \mid r-a$ but $p^2 \nmid r-a$, then by    Proposition~\ref{prop:A_1} and Proposition~\ref{prop:A_4} we have $\overline{\Theta}^{\,{ss}} \cong (F_1 \oplus F_2)^{\,{ss}}$. Moreover, Proposition~\ref{prop:A_5} tells us that we have surjections $\pi(a-2,\bar{u'},\omega) \twoheadrightarrow F_2$ and $\pi(p-a-1,\bar{u'}^{-1},\omega^{a}) \twoheadrightarrow F_1$. We must have $\overline{\Theta}^{\,{ss}} \cong \pi(a-2,\bar{u'},\omega)^{\,{ss}} \oplus \pi(p-a-1,\bar{u'}^{-1},\omega^{a})^{\,{ss}}$.
    \item [(4)] If $p^2 \mid r-a$, then by Proposition~\ref{prop:A_1} and Proposition~\ref{prop:A_4}  we have $\overline{\Theta} \cong F_1 \cong \pi(p-a-1,0,\omega^{a})$.
\end{itemize}
Theorem~\ref{thm:3.1} now follows in all cases from the mod $p$ LLC.
\end{proof}

 \section{The case $a=1$}\label{sec: a=1}
 The remaining congruence classes $a \pmod{p-1}$ require a separate treatment.
In this section we will treat the case $r \equiv 1\pmod{p-1}$.
\begin{theorem}\label{thm: con 1}
    Let $p \geq 5$, $r\geq 3p+2$ and $r \equiv1 \pmod{p-1}$. Assume that $v_p(a_p)=2$. Then \begin{enumerate}
        \item [(1)] If $p \nmid r(r+1)$, then $\overline{V}^{ss}_{k,a_p} \cong \mu_{\lambda} \oplus \mu_{\lambda^{-1}} \cdot \omega^2 $, where $$\lambda = \overline{\frac{2}{r(r+1)}\cdot \frac{a_p}{p^2}}\in \overline{\mathbb{F}}_p^\times .$$
        \item[(2)] If $p \mid r+1$, then $\overline{V}^{ss}_{k,a_p} \cong \mathrm{ind}\bigl(\omega_2^{2}\bigl).$
        \item[(3)] If $p \mid r-p$ but $p^2\nmid r-p$, then $\overline{V}^{ss}_{k,a_p} \cong  \mathrm{ind}\bigl(\omega_2^{2}\bigl).$
        \item[(4)] If $p^2 \mid p-r$, then $\overline{V}^{ss}_{k,a_p} \cong \mu_{\lambda}\cdot \omega\oplus \mu_{\lambda^{-1}} \cdot \omega$, where $$\lambda+\lambda^{-1}= \overline{\frac{p-r}{a_p}} \in \overline{\mathbb{F}}_p.$$
        
    \end{enumerate}
    
\end{theorem}

%\begin{lemma}\label{lem:4.2}
   % Let $p \le r\equiv 1 \pmod{p-1}$. Then \[
%X_{r-1}^{**}/X_{r-1}^{***} = 0.
%\]
%\end{lemma}
%\begin{proof}
  %  By  \cite[Lemma 4.4]{GV22} we have $\frac{X_{r-1}^{**}}{X_{r-1}^{***}} =\frac{X_{r}^{**}}{X_{r}^{***}} $. Now by  \cite[Proposition 4.8]{GV22} we have $\frac{X_{r}^{**}}{X_{r}^{***}} =0$.
%\end{proof}
\begin{lemma} \label{theta-square for a = p}
%\label{lem:4.3}
Let $2p+3 \le r \equiv 1\pmod{p-1}$. Then $$X^{r-2}Y^2 \equiv \frac{r(r+1)}{2} \theta^2 X^{r-2p-2} \pmod{ X_{r-1}+V_r^{(3)}}.$$  
\end{lemma}
\begin{proof}
Similar to  the
proof of Lemma~\ref{lem:3.3} 
(substitute $a =p$ in that proof). 
\end{proof}

\begin{prop}\label{prop:B_1}
     Let $p \ge 5$, $r \ge 3p+2$ and $r \equiv 1 \bmod(p-1)$. Then\begin{enumerate}
        \item [(i)] $F_5=0$.
        \item[(ii)] If $p \mid r(r+1)$, then $U_3=0$.
    \end{enumerate}
\end{prop}

\begin{proof}
   The proof is identical to the proof of Proposition~\ref{prop:A_1} (taking $a=p$), except that we use Remark~\ref{rmk:A} instead of Lemma~\ref{lem:3.2} and Lemma~\ref{theta-square for a = p} instead of Lemma~\ref{lem:3.3}.
\end{proof}
\begin{prop}\label{prop:B_2}
    Let $3p+2 \le r \equiv 1 \pmod{p-1}$. Assume that $v_p(a_p)=2$ and $p\nmid r(r+1)$, so that $u 
 := \frac{2}{r(r+1)} \cdot \frac{a_p}{p^2}$ is a $p$-adic unit. Then  \begin{enumerate}
        \item [(i)] $F_4$ is a quotient of $\pi(p-4, \bar u,\omega^2)$.
        \item [(ii)] $F_3$ is a quotient of $\pi(1,{\bar u}^{-1},1)$.
    \end{enumerate}
\end{prop}
\begin{proof}
    (i) Similar to the proof of Proposition~\ref{prop:A_2}, taking $a = p$ there.
   \begin{comment}
    Let
\[
f_0=\left[\mathrm{Id},\frac{\theta^2X^{r-2p-2}}{p^2}\right]
\in \mathrm{ind}_{KZ}^G \mathrm{Sym}^r \overline{\mathbb{Q}}_p^{\,2}.
\]
Since \(r\geq 3p+2\), we have
\[
T^-f_0\equiv 0 \pmod p.
\]
Moreover,
\[
T^+f_0
 = \sum_{\lambda\in\mathbb F_p}
 \left[g^0_{1,\left[\lambda \right]},X^{r-2}Y^2\right].
\]
Therefore
\[
(T-a_p)f_0
 \equiv
 \sum_{\lambda\in\mathbb F_p}
 \left[g^0_{1,\left[\lambda\right]},X^{r-2}Y^2\right]
 -
 \overline{\left(\frac{a_p}{p^2}\right)}
 \left[\mathrm{Id},\theta^2X^{r-2p-2}\right]
 \pmod p.
\]
By Lemma~\ref{theta-square for a = p}, the image of \((T-a_p)f_0\) in
\(\mathrm{ind}_{KZ}^G P\) is
\[
\overline{\frac{r(r+1)}{2}}
\sum_{\lambda\in\mathbb F_p}
\left[g^0_{1,\left[\lambda \right]},\theta^2X^{r-2p-2}\right]
-
\overline{\left(\frac{a_p}{p^2}\right)}
\left[\mathrm{Id},\theta^2X^{r-2p-2}\right].
\]
By \cite[(4.2)]{Glo78} and \cite[Lemma~5.3]{Bre03b}, this further maps to
\[
\overline{\frac{r(r+1)}{2}}
\,(T-\bar u)
\left[\mathrm{Id},X^{p-4}\right] \in \mathrm{ind}_{KZ}^G J_4.
\]
Since \(\left[\mathrm{Id},X^{p-4}\right]\) generates
\(\mathrm{ind}_{KZ}^G J_4\) as a \(G\)-module, the map $\mathrm{ind}_{KZ}^G J_4 \twoheadrightarrow F_4$ factors through $\pi(p-4,\bar{u}, \omega^2)$.
\end{comment} 

(ii) Let us consider $f_0\in \mathrm{ind}_{KZ}^G \mathrm{Sym}^r \bar{\mathbb{Q}}_p^2$ defined by \begin{eqnarray*}
    f_2 &=& \sum_{\lambda \in \mathbb{F}_p} \left[ g^0_{2, \left[\lambda\right]} ,\frac{XY^{r-1}-X^p Y^{r-p}+X^{r-p+1}Y^{p-1}-X^{r-2p+2}Y^{2p-2}}{p^2} \right],\\
f_1 &=&\left[g^0_{1,0}, \frac{(p-1)}{p\cdot a_p} \sum_{\substack{
    0<j<r-1 \\
    j \equiv 0 \pmod{p-1}
}}
\alpha_j \cdot X^{\,r-j} Y^j \right],\\
     f_0 &=& \biggl[ \mathrm{Id}, \frac{X^r- X^pY^{r-p}}{p}\biggl],
     \end{eqnarray*}
      where $\alpha_j$ are integers as in Lemma~\ref{lem:2.2} (applied for $r-1$). Both $T^-f_1$ and $T^+f_1$ die mod $p$ and $T^+f_2$ dies mod ${p}$ by Lemma~\ref{theta-square for a = p}. Also $T^-f_0$ and $-a_pf_0$ die mod $p$. We compute \begin{eqnarray*}
    T^-f_2 &=& \biggl[g^0_{1,0},\frac{p-1}{p}\sum_{\substack{
    0\le j < r-1 \\
    j \equiv 0 \pmod{p-1}
}}
\binom{r-1}{j}\cdot X^{\,r-j} Y^j \biggl],\\
%T^-f_2-a_pf_1 &=& \biggl[g^0_{1,0},\frac{p-1}{p} \cdot X^r+\frac{p-1}{p}\sum_{\substack{
   % 0 < j < r-1 \\
  %  j \equiv 0 \pmod{p-1}
%}}
%\biggl(\binom{r-1}{j}-\alpha_j\biggl)\cdot X^{\,r-j} Y^j \biggl].
T^+f_0 &=& \biggl[g^0_{1,0}, \frac{X^r}{p}\biggl].
\end{eqnarray*}  

     Thus, we have \begin{eqnarray*}
         T^-f_2-a_pf_1+T^+f_0 &=& \biggl[g^0_{1,0},\frac{p-1}{p}\sum_{\substack{
    0 < j < r-1 \\
    j \equiv 0 \pmod{p-1}
}}
\biggl(\binom{r-1}{j}-\alpha_j\biggl)\cdot X^{\,r-j} Y^j \biggl].
\end{eqnarray*}
The rest of the proof is similar to the proof of Proposition~\ref{prop:A_2} (ii). We 
have 
\begin{eqnarray*}
    T^-f_2-a_pf_1+T^+f_0 &=& \biggl[g^0_{1,0},\frac{p-1}{p}\sum_{\substack{
    0 < j < r-1 \\
    j \equiv 0 \pmod{p-1}
}}
\biggl(\binom{r-1}{j}-\alpha_j\biggl)\cdot (X^{\,r-j} Y^j-X^pY^{r-p})\\
&& \qquad + \frac{p-1}{p}\sum_{\substack{
    0 < j < r-1 \\
    j \equiv 0 \pmod{p-1}
}}
\biggl(\binom{r-1}{j}-\alpha_j\biggl)\cdot X^p Y^{r-p}\biggl]\\
&=& \biggl[g^0_{1,0},\frac{p-1}{p}\sum_{\substack{
    0 < j < r-1 \\
    j \equiv 0 \pmod{p-1}
}}
(j-r)\biggl(\binom{r-1}{j}-\alpha_j\biggl)\cdot \theta X^{r-2p+1}Y^{p-2}\\
&& \qquad + \frac{p-1}{p}\sum_{\substack{
    0 < j < r-1 \\
    j \equiv 0 \pmod{p-1}
}}
\biggl(\binom{r-1}{j}-\alpha_j\biggl)\cdot X^p Y^{r-p}\biggl]\\
&=& \biggl[g^0_{1,0},\frac{r(r+1)}{2} \cdot \theta X^{r-2p+1}Y^{p-2} \\
&& \qquad + \frac{p-1}{p}\sum_{\substack{
    0 < j < r-1 \\
    j \equiv 0 \pmod{p-1}
}}
\biggl(\binom{r-1}{j}-\alpha_j\biggl)\cdot X^p Y^{r-p}\biggl],
\end{eqnarray*}
by Lemma~\ref{lem:2.1} with $i=0,1$ and Lemma~\ref{lem:2.2}.
On the other hand, we have \[
-a_pf_2 = \frac{a_p}{p^2}\sum_{\lambda \in \mathbb{F}_p} \left[ g^0_{2, \left[\lambda \right]} ,{\theta Y^{r-p-1}-\theta X^{r-2p+1}Y^{p-2}} \right].
\]
By Lemma~\ref{lem:2.10}, the polynomial
\(
\theta X^{r-2p+1}Y^{p-2}
\)
maps to \(X\) in the cosocle \(J_3\) and $X^p Y^{r-p}= \theta Y^{r-p-1}+XY^{r-1}$. Thus $(T-a_p)(-f_0-f_1-f_2)$ is integral and maps to \[
\overline {\frac{a_p}{p^2}}\sum_{\lambda\in \mathbb{F}_p}\left[ g^0_{2,\left[\lambda \right]},X\right]-\left[g^0_{1,0},\overline{\frac{r(r+1)}{2}}\cdot X\right]
=\overline {\frac{a_p}{p^2}}(T-\bar{u}^{-1})\left[g^0_{1,0}, X \right]\in \mathrm{ind}_{KZ}^G J_3.\]
Hence, $F_3$ is a quotient of $\pi(1,\bar{u}^{-1}, 1).$
\end{proof}

\begin{prop}\label{prop:B_3}
Let $3p+2 \le r \equiv 1 \pmod{p-1}$. Assume that $v_p(a_p)=2$ and $p^2 \mid (r-p)$, then $F_2$ is a quotient of $\frac{\mathrm{ind}_{KZ}^G(V_{p-2}\otimes D)}{T^2-cT+1}$, where $c=\frac{p-r}{a_p}$.
\end{prop}
\begin{proof}
    Let us consider $f_0,f_1,f_2 \in \mathrm{ind}_{KZ}^G \mathrm{Sym}^r \bar{\mathbb{Q}}_p^2$ defined by 
    \begin{eqnarray*}
        f_2 &=& \sum_{{\lambda \in \mathbb{F}_p},\,\mu \in \mathbb{F}_p^{\times} }\biggl[g^0_{2,\left [\lambda \right]+p\left[\mu \right]}, \frac{Y^r-X^{r-p}Y^p}{p^2}\biggl]+(1-p)\sum_{\lambda \in \mathbb{F}_p}\biggl[g^0_{2,\left [\lambda \right]}, \frac{Y^r-X^{r-p}Y^p}{p^2}\biggl]\\
        f_1 &=& \sum_{\lambda \in {\mathbb{F}_p}} \biggl[g^0_{1,\left[\lambda \right]},\frac{p-1}{p^2 \cdot a_p} \sum_{\substack{
1<j<r \\
j \equiv 1 \pmod{p-1}
}}
 \alpha_j \cdot X^{r-j} Y^{j}\biggl]\\
 f_0 &=& \frac{r}{p} \biggl[\mathrm{Id},\frac{X^{r-1}Y-X^{r-p}Y^p}{p^2}\biggl],
 \end{eqnarray*}
 where the $\alpha_j$ are the  integers in part $2$ of Lemma~\ref{lem:2.6}.
 It is clear that $T^+f_2$, $T^+f_1$, and $T^-f_0$ die mod $p$ as $p\mid (r-p)$. Also $T^-f_1$ dies mod $p$ (as $p  \mid \alpha_{r-p+1}$ since $p \mid \binom{r}{r-p+1}$; this is only needed when $p = 5)$. We have
 \begin{eqnarray*}
     T^-f_2 &=& \sum_{\lambda \in {\mathbb{F}_p}} \biggl[g^0_{1,\left[\lambda\right]},\frac{(p-1)r}{p^2} X^{r-1}Y+ \frac{p-1}{p^2}\sum_{\substack{
1<j<r \\
j \equiv 1 \pmod{p-1}
}}
 \binom{r}{j} X^{r-j} Y^{j}\biggl], \text{ so}\\
  T^-f_2-a_pf_1 &=& \sum_{\lambda\in {\mathbb{F}_p}} \biggl[g^0_{1,\left[\lambda\right]}, \frac{(p-1)r}{p^2} X^{r-1}Y+\frac{p-1}{p^2}\sum_{\substack{
1<j<r \\
j \equiv 1 \pmod{p-1}
}}
 \biggl(\binom{r}{j}-\alpha_j\biggl) X^{r-j} Y^{j}\biggl].
 \end{eqnarray*}
 Also \[
 T^+f_0=\frac{r(1-p)}{p^2} \sum_{\lambda \in \mathbb{F}_p} \biggl[g^0_{1,\left[\lambda \right]}, X^{r-1}Y\biggl].
 \]
 So, using $X^{r-j}Y^j-X^{r-p}Y^p \equiv (j-p) \theta X^{r-2p} Y^{p-1}$ by \cite[Lemma 2.12]{GV22} and noting that the congruence \eqref{mod p cube} also holds mod $p^3$, we have 
 \begin{eqnarray*}
     T^-f_2-a_pf_1+T^+f_0 &=& \sum_{\lambda \in {\mathbb{F}_p}} \biggl[g^0_{1,\left[\lambda\right]}, \frac{p-1}{p^2} \sum_{\substack{
1<j<r \\
j \equiv 1 \pmod{p-1}
}}
\biggl(\binom{r}{j}-\alpha_j\biggl). X^{r-j} Y^{j}\biggl]\\
 &=& \sum_{\lambda \in {\mathbb{F}_p}} \biggl[g^0_{1,\left[\lambda\right]},  \frac{p-1}{p^2}\sum_{\substack{
1<j<r \\
j \equiv 1 \pmod{p-1}
}}
\biggl(\binom{r}{j}-\alpha_j\biggl). \biggl(X^{r-j} Y^{j}-X^{r-p}Y^p\biggl)\\
 && \qquad + \frac{p-1}{p^2} \sum_{\substack{
1<j<r \\
j \equiv 1 \pmod{p-1}
}}
 \biggl(\binom{r}{j}-\alpha_j\biggl)X^{r-p}Y^p\biggl]\\
 &=& \sum_{\lambda \in {\mathbb{F}_p}} \biggl[g^0_{1,\left[\lambda\right]}, \frac{(p-1)(p-r)}{p^2}X^{r-p}Y^p\biggl]\\
 & \equiv & \sum_{\lambda \in {\mathbb{F}_p}} \biggl[g^0_{1,\left[\lambda\right]}, \frac{(p-r)}{p^2}(X^{r-1}Y-X^{r-p}Y^p)\biggl] \pmod{X_{r-1}},
 \end{eqnarray*}
 where we note that
 the congruence $\eqref{mod p^2}$ holds mod $p^3$.
 Hence, we have \begin{eqnarray*}
     (T-a_p)(f_0+f_1+f_2) &=& \frac{-a_p}{p^2} {\sum_{{\lambda \in \mathbb{F}_p},\mu \in \mathbb{F}_p^{\times}}
     \biggl[g^0_{2,\left [\lambda \right]+p\left[\mu \right]}, {Y^r-X^{r-p}Y^p}\biggl]-\frac{a_p}{p^2}\sum_{\lambda \in \mathbb{F}_p}\biggl[g^0_{2,\left [\lambda \right]}, {Y^r-X^{r-p}Y^p}\biggl]}\\
    && \qquad + \sum_{\lambda \in {\mathbb{F}_p}} \biggl[g^0_{1,\left[\lambda\right]}, \frac{(p-r)}{p^2}(X^{r-1}Y-X^{r-p}Y^p)\biggl]- \frac{a_p}{p^2}\biggl[\mathrm{Id},{X^{r-1}Y-X^{r-p}Y^p}\biggl]\\
    &=& \frac{-a_p}{p^2} {\sum_{{\lambda \in \mathbb{F}_p},\mu \in \mathbb{F}_p} \biggl[g^0_{2,\left [\lambda \right]+p\left[\mu \right]}, {X^{r-1}Y-X^{r-p}Y^p}\biggl]}\\ %-\frac{a_p}{p^2}\sum_{\lambda \in \mathbb{F}_p}\biggl[g^0_{2,\left [\lambda \right]}, {X^{r-1}Y-X^{r-p}Y^p}\biggl]}\\
   && \qquad + \sum_{\lambda \in {\mathbb{F}_p}} \biggl[g^0_{1,\left[\lambda\right]}, \frac{(p-r)}{p^2}(X^{r-1}Y-X^{r-p}Y^p)\biggl]- \frac{a_p}{p^2}\biggl[\mathrm{Id},{X^{r-1}Y-X^{r-p}Y^p}\biggl] \\
    &=& \frac{-a_p}{p^2} \biggl({\sum_{{\lambda \in \mathbb{F}_p},\mu \in {\mathbb{F}_p}} \biggl[g^0_{2,\left [\lambda \right]+p\left[\mu \right]}, \theta X^{r-p-1}\biggl]}-\sum_{\lambda \in {\mathbb{F}_p}} \biggl[g^0_{1,\left[\lambda\right]}, \frac{(p-r)}{a_p}\theta X^{r-p-1}\biggl]\\
    && \qquad + \biggl[\mathrm{Id},\theta X^{r-p-1}\biggl]\biggl).
   \end{eqnarray*}
   Moreover, \(\theta X^{r-p-1}\) is identified with \(X^{p-2}\) in $J_2=V_{p-2}\otimes D$.
So the image of \((T-a_p)f\) in $\operatorname{ind}_{KZ}^{G} J_2$
is given by
\[
-\frac{a_p}{p^2}(T^2-\frac{p-r}{a_p}T+1)\biggl[\mathrm{Id}, X^{p-2}\biggl].
\]
Since \(\left[\mathrm{Id},X^{p-2}\right]\) generates
\(\mathrm{ind}_{KZ}^G J_2\) as a \(G\)-module, the map $\mathrm{ind}_{KZ}^G J_2\twoheadrightarrow F_2$ factors through $(T^2-\frac{p-r}{a_p}T+1)$.
Hence, $F_2$ is a quotient of  $\frac{\mathrm{ind}_{KZ}^G(V_{p-2}\otimes D)}{(T^2-\frac{p-r}{a_p}T+1)}$.
\end{proof}
\begin{prop}\label{prop:B_4}
Let $3p+2 \le r \equiv 1 \pmod{p-1}$. Assume that $v_p(a_p)=2$ and $p \mid (r-p)$ but $p^2 \nmid (r-p)$. Then $F_2=0$.
\end{prop}
\begin{proof}
     Let us consider $f_0,f_1,f_2 \in \mathrm{ind}_{KZ}^G \mathrm{Sym}^r \bar{\mathbb{Q}}_p^2$ defined by 
     \begin{eqnarray*}
    f_2 &=& \sum_{\lambda\in \mathbb{F}_p^\times} \biggl[g^0_{2,p \left[ \lambda\right]}, \frac{\left[ \lambda\right]^{p-2}(Y^r-X^{r-p}Y^p)}{p}\biggl]+\frac{r(1-p)}{p^2} \biggl[g^0_{2,0}, X Y^{r-1}-X^{r-p+1}Y^{p-1}\biggl]\\
    f_1 &=& \biggl[g^0_{1,0}, \frac{p-1}{p \cdot a_p} \sum_{\substack{
0<j<r-1 \\
j \equiv 0 \pmod{p-1}
}}
 \alpha_j \cdot  X^{r-j} Y^{j}\biggl]\\
 f_0 &=& \biggl[ \mathrm{Id}, \frac{1-p}{p} (X^r-X^pY^{r-p})\biggl],
     \end{eqnarray*}
     where $\alpha_j$ are integers in Lemma~\ref{lem:2.5} (for $a=p$). It is clear that $T^+f_2$, $-a_pf_2$, $T^+f_1$, $T^-f_1$, $T^-f_0$ and $-a_pf_0$ die mod $p$. We have 
    \begin{eqnarray*}
        T^-f_2 &=& \biggl[g^0_{1,0}, \frac{p-1}{p} X^r + \frac{p-1}{p} 
\sum_{\substack{
0<j<r-1 \\
j \equiv 0 \pmod{p-1}
}}
 \binom{r}{j} X^{r-j} Y^{j}  \biggl]\\
 T^-f_2-a_pf_1 &=& \biggl[g^0_{1,0}, \frac{p-1}{p} X^r + \frac{p-1}{p} 
\sum_{\substack{
0<j<r-1 \\
j \equiv 0 \pmod{p-1}
}}
 \biggl(\binom{r}{j}-\alpha_j\biggl) X^{r-j} Y^{j}  \biggl].
     \end{eqnarray*}
     We compute \[
     T^+f_0= \frac{1-p}{p} \biggl[g^0_{1,0}, X^r\biggl].
     \]
     Hence, since $X^{r-j}Y^j-X^pY^{r-p} \equiv j\theta X^{r-2p+1}Y^{p-2} \pmod{V_r^{**}}$ by \cite[Lemma $2.12$]{GV22} and the congruences \eqref{mod p^2 For a=p} and \eqref{mod p^2 m=0} hold mod $p^2$,  we have \begin{eqnarray*}
         T^-f_2-a_pf_1+T^+f_0 &=& \biggl[g^0_{1,0}, \frac{p-1}{p} 
\sum_{\substack{
0<j<r-1 \\
j \equiv 0 \pmod{p-1}
}}
 \biggl(\binom{r}{j}-\alpha_j\biggl) X^{r-j} Y^{j}  \biggl]\\
 &=& \biggl[g^0_{1,0}, \frac{p-1}{p} 
\sum_{\substack{
0<j<r-1 \\
j \equiv 0 \pmod{p-1}
}}
 \biggl(\binom{r}{j}-\alpha_j\biggl) (X^{r-j} Y^{j}-X^{p}Y^{r-p})\\
 && \qquad + \frac{p-1}{p}  \sum_{\substack{
0<j<r-1 \\
j \equiv 0 \pmod{p-1}
}} \biggl(\binom{r}{j}-\alpha_j\biggl)X^{p}Y^{r-p} \biggl]\\
 &\equiv& \biggl[g^0_{1,0}, \frac{(p-1)(p-r)}{p} X^{p}Y^{r-p} \biggl]\\
 &\equiv & \biggl[g^0_{1,0}, \frac{(p-r)}{p} (XY^{r-1}-X^{p}Y^{r-p}) \biggl] \pmod{X_{r-1}}.
     \end{eqnarray*}
     Hence, we have $(T-a_p)(f_0+f_1+f_2)=-\biggl[g^0_{1,0}, \frac{(p-r)}{p} \theta Y^{r-p-1} \biggl]$.  
      Moreover, \(\theta Y^{r-p-1}\) is identified with \(Y^{p-2}\) in $J_2=V_{p-2}\otimes D$.
So the image of \((T-a_p)f\) in $\operatorname{ind}_{KZ}^{G} J_2$
is given by $\biggl[g^0_{1,0}, \frac{(r-p)}{p} Y^{p-2} \biggl]$. As $\frac{r-p}{p}$ is an unit, \(\frac{(r-p)}{p}\left[g^0_{1,0},Y^{p-2}\right]\) generates
\(\mathrm{ind}_{KZ}^G J_2\) as a \(G\)-module. In particular, $F_2=0$.
\end{proof}

\begin{prop}\label{prop:B_5}
   Let $3p+2 \le r \equiv 1 \pmod{p-1}$. Assume that $v_p(a_p)=2$ and $p^2 \mid (r-p)$. Then $F_3=0$.
\end{prop}

\begin{proof}
Let us consider $f_0,f_1,f_2 \in \mathrm{ind}_{KZ}^G \mathrm{Sym}^r \bar{\mathbb{Q}}_p^2$ defined by
\begin{eqnarray*}
    f_2 &=&  \sum_{\lambda\in \mathbb{F}_p^\times}\biggl[g^0_{2,p \left[ \lambda\right]}, \frac{\left[ \lambda\right]^{p-2}(Y^r-X^{r-p}Y^p)}{p^2}\biggl]+\frac{r(1-p)}{p^3} \biggl[g^0_{2,0}, X Y^{r-1}-X^{r-p+1}Y^{p-1}\biggl]\\
    f_1 &=& \biggl[g^0_{1,0}, \frac{p-1}{p^2 \cdot a_p} \sum_{\substack{
0<j<r-1 \\
j \equiv 0 \pmod{p-1}
}}
 \alpha_j \cdot X^{r-j} Y^{j}\biggl]\\
 f_0 &=& \biggl[ \mathrm{Id}, \frac{1-p}{p^2} (X^r-X^pY^{r-p})\biggl],
 \end{eqnarray*}
 where $\alpha_j$ are integers in part 1 of Lemma~\ref{lem:2.6}.
It is clear that $T^+f_2$, $T^-f_1$, $T^+f_1$ and $T^-f_0$ die mod $p$. As in the proof of the previous proposition, we have \begin{eqnarray*}
     T^-f_2-a_pf_1+T^+f_0 & \equiv & \biggl[g^0_{1,0}, \frac{p-1}{p^2} 
\sum_{\substack{
0<j<r-1 \\
j \equiv 0 \pmod{p-1}
}}
 \biggl(\binom{r}{j}-\alpha_j\biggl) \cdot X^{r-j} Y^{j}  \biggl].
\end{eqnarray*}% We compute 
%\begin{eqnarray*}
   %T^-f_2 &=& \biggl[g^0_{1,0}, \frac{p-1}{p^2} X^r + \frac{p-1}{p^2} 
%\sum_{\substack{
%0<j<r-1 \\
%j% \equiv 0 \pmod{p-1}
%}}
 %\binom{r}{j} X^{r-j} Y^{j}  \biggl]\\
 %T^-f_2-a_pf_1 &=& \biggl[g^0_{1,0}, \frac{p-1}{p^2} X^r + \frac{p-1}{p^2} 
%\sum_{\substack{
%0<j<r-1 \\
%j \equiv 0 \pmod{p-1}
%}}
 %\biggl(\binom{r}{j}-\alpha_j \biggl) X^{r-j} Y^{j}  \biggl].
%\end{eqnarray*}
%Moreover, we have \[
%T^+f_0= \frac{1-p}{p^2} \biggl[g^0_{1,0}, X^r\biggl].
%\]
 Again $X^{r-j}Y^j-X^pY^{r-p} \equiv j\theta X^{r-2p+1}Y^{p-2} \pmod{V_r^{**}}$ by \cite[Lemma $2.12$]{GV22} and since $p^2 \mid r-p$ the congruences \eqref{final cong m=1} and \eqref{final m=0 cong}  hold mod $p^3$,  we have 
 \begin{eqnarray*}
         %T^-f_2-a_pf_1+T^+f_0 &=& \biggl[g^0_{1,0}, \frac{p-1}{p^2} 
%\sum_{\substack{
%0<j<r-1 \\
%j \equiv 0 \pmod{p-1}
%}}
 %\biggl(\binom{r}{j}-\alpha_j\biggl) %\cdot X^{r-j} %Y^{j}  \biggl]\\
 T^-f_2-a_pf_1+T^+f_0 &=& \biggl[g^0_{1,0}, \frac{p-1}{p^2} 
\sum_{\substack{
0<j<r-1 \\
j \equiv 0 \pmod{p-1}
}}
 \biggl(\binom{r}{j}-\alpha_j\biggl)\cdot (X^{r-j} Y^{j}-X^{p}Y^{r-p})\\
 && \qquad + \frac{p-1}{p^2} 
\sum_{\substack{
0<j<r-1 \\
j \equiv 0 \pmod{p-1}
}}
 \biggl(\binom{r}{j}-\alpha_j\biggl)X^{p}Y^{r-p}\biggl]\\
 &\equiv& \biggl[g^0_{1,0}, \frac{(p-1)(p-r)}{p^2} X^{p}Y^{r-p} \biggl]\\
 &\equiv& \biggl[g^0_{1,0}, \frac{(p-r)}{p^2} (XY^{r-1}-X^{p}Y^{r-p}) \biggl]\pmod{X_{r-1}}.
 \end{eqnarray*}
 Finally we have \begin{eqnarray*}
     (T-a_p)(f_0+f_1+f_2) &=& -\frac{a_p}{p^2}\sum_{\lambda\in \mathbb{F}_p^{\times}} \biggl[g^0_{2,p \left[ \lambda\right]}, \left[ \lambda\right]^{p-2}\theta X^{r-p-1}\biggl]-\frac{ra_p}{p^3} \biggl[g^0_{2,0}, X Y^{r-1}-X^{r-p+1}Y^{p-1}\biggl]\\
     && \qquad -\frac{(p-r)}{p^2}\biggl[g^0_{1,0},  \theta Y^{r-p-1} \biggl]+\frac{a_p}{p^2}\biggl[\mathrm{Id},\theta Y^{r-p-1}\biggl]\pmod{X_{r-1}}.
      \end{eqnarray*}
    Thus $(T-a_p)(f_0+f_1+f_2)$ is integral and, by \cite[Lemma $2.12$]{GV22} and Lemma~\ref{lem:2.10}, maps to $\overline{\frac{-a_p}{p^2}}\biggl[g^0_{2,0}, X\biggl]$ in \(\mathrm{ind}_{KZ}^G J_3\), which generates this as a $G$-module. Hence, we have $F_3=0$.
\end{proof}
\begin{prop}\label{prop:B_6}
    Let $2p+1\le r \equiv 1 \pmod{p-1}$. Then 
    \begin{enumerate}
        \item [(i)] $F_1=0$.
        \item [(ii)] If $p \nmid r$, then $F_2=0$.
    \end{enumerate}
\end{prop}
\begin{proof}
    (i) By \cite[Lemma $4.3$]{GV22} we have $\frac{X_{r-1}}{X_{r-1}^{*}} \cong \frac{V_r}{V_r^*}$.

    (ii) By \cite[Lemma $4.34$] {GV22} if $p \nmid r$, we have $\frac{X_{r-1}^*}{X_{r-1}^{**}} \cong V_{p-2} \otimes D$.
\end{proof}

\vspace{1cm}
%{\color{red} We have to write something}
\begin{proof}[Proof of Theorem~\ref{thm: con 1}] 
The above propositions may be summarized by the following table:

\begin{table}[h]
\centering
\begin{tabular}{|>{\centering\arraybackslash}m{5cm}|
                >{\centering\arraybackslash}m{1.3cm}|
                >{\centering\arraybackslash}m{1.3cm}|
                >{\centering\arraybackslash}m{1.3cm}|
                >{\centering\arraybackslash}m{1.3cm}|
                >{\centering\arraybackslash}m{1.3cm}|}
\hline
& $F_5$ & $F_4$ & $F_3$ & $F_2$ & $F_1$ \\ \hline
$p\nmid r(r+1)$
& & $\checkmark$& $\checkmark$ & & \\ \hline

$p\mid (r+1)$
& & &$\checkmark$ & & \\ \hline

$p\mid (r-p),\; p^2\nmid (r-p)$
& & &$\checkmark$ & &\\ \hline

$p^2\mid (r-p)$
& & & &$\checkmark$  & \\ \hline
\end{tabular}
\end{table}

In particular, we have:
\begin{itemize}
    \item [(1)] If $p \nmid r(r+1)$, then by Proposition~\ref{prop:B_1} and Proposition~\ref{prop:B_6} we get $\overline{\Theta}^{\,{ss}} \cong (F_3 \oplus F_4)^{\,{ss}}$. Moreover, Proposition~\ref{prop:B_2} tells us that we have surjections $\pi(p-4,\bar{u},\omega^2) \twoheadrightarrow F_4$ and $\pi(1,\bar{u}^{-1},1) \twoheadrightarrow F_3$. Hence, we have $\overline{\Theta}^{\,{ss}} \cong \pi(p-4,\bar{u},\omega^2)^{\,{ss}} \oplus \pi(1,\bar{u}^{-1},1)^{\,{ss}}$.
    \item [(2)] If $p \mid r+1$, then by Proposition~\ref{prop:B_1} and Proposition~\ref{prop:B_6}, we have $\overline{\Theta} \cong F_3 \cong \pi(1,0,1)$.
    \item [(3)] If $p \mid r-p$ but $p^2 \nmid r-p$, then by    Proposition~\ref{prop:B_1}, Proposition~\ref{prop:B_4} and Proposition~\ref{prop:B_6}, we have $\overline{\Theta} \cong F_3 \cong \pi(1,0,1)$.
    \item [(4)] If $p^2 \mid r-p$, then by Proposition~\ref{prop:B_1}, Proposition~\ref{prop:B_3}, Proposition~\ref{prop:B_5} and Proposition~\ref{prop:B_6},  we have $\overline{\Theta}^{\,{ss}} \cong \pi(p-2,{\lambda},\omega)^{\,{ss}} \oplus \pi(p-2,{\lambda}^{-1},\omega)^{\,{ss}}$, where $\lambda+\lambda^{-1}=\overline{\frac{
    p-r}{a_p}}$.
\end{itemize}
Theorem~\ref{thm: con 1} now follows in all cases from the mod $p$ LLC.
\end{proof}

\section{The case $a=2$}\label{sec: a=2}
In this section we will discuss the case $r \equiv 2 \pmod {p-1}$.
\begin{theorem}\label{thm:a=2}
Let $3p+2 \le r\equiv 2 \pmod{p-1}$. Assume that $v_p(a_p)=2$. Then \begin{enumerate}
    \item [(1)] If $p \nmid r(r-1)$, then $\overline{V}^{\,{ss}}_{k,a_p}
\cong \mathrm{ind}\bigl(\omega_2^{3}\bigr)$.
\item [(2)] If $p\mid r$ but $p^2 \nmid r-2p$, then $\overline{V}^{\,{ss}}_{k,a_p}
\cong \mu_{\lambda}\cdot \omega\oplus\mu_{\lambda^{-1}} \cdot \omega^2 $, where $$\lambda=\overline{\frac{2p}{r-2p} \cdot \frac{a_p}{p^2}} \in \overline{\mathbb{F}}_p^\times.$$
 \item[(3)] If $p^2 \mid r-2p$, then $\overline{V}^{\,{ss}}_{k,a_p} \cong \mathrm{ind}\bigl(\omega_2^{p+2}\bigl)$.
\item[(4)] If $p\mid r-1$ but $p^2 \nmid r-p-1$, then $\overline{V}^{\,{ss}}_{k,a_p} \cong \mu_{\lambda}\cdot \omega \oplus\mu_{\lambda^{-1}} \cdot \omega^2$, where $$\lambda = \overline{\frac{p}{p+1-r}\cdot \frac{a_p}{p^2}} \in \overline{\mathbb{F}}_p^\times.$$
\item[(5)] If $p^2 \mid r-p-1$, then $\overline{V}^{ss}_{k,a_p} \cong \mathrm{ind} \bigl(\omega_2^{2p+1}\bigl)$. 
\end{enumerate}
\end{theorem}
\begin{lemma}\label{lem:cong 2 f_1}
     Let $2p+3 \le r \equiv 2\pmod{p-1}$. Then \begin{itemize}
         \item [(i)] $\binom{r}{2} \theta^2 X^{r-2p-2} \in X_{r-1}+V_r^{(3)}$.
         \item [(ii)] $(r-1)( X^{r-2}Y^2-X^{r-p-1}Y^{p+1}) \in X_{r-1}+V_r^{(3)} $.
     \end{itemize}
\end{lemma}
\begin{proof}

 (i) We know that \begin{eqnarray*}
         \sum_{\substack{
    j \geq 0 \\
    j \equiv 2 \pmod{p-1}
}}
\binom{r }{j} \cdot X^{\,r-j} Y^j &\in& X_{r}, \text{ so}\\
\binom{r}{2} \cdot X^{r-2}Y^2+\sum_{\substack{
  2 < j <r \\
    j \equiv 2 \pmod{p-1}
}}
\binom{r }{j} \cdot X^{\,r-j} Y^j+Y^r &\in& X_{r-1}.
    \end{eqnarray*}
  By \cite[Lemma 2.14]{GV22} we have  \begin{eqnarray*}
 \sum_{\substack{
  j \geq 0\\
    j \equiv 2 \pmod{p-1}
}}
 \binom{j}{m}\binom{r }{j} \equiv \binom{r}{m} + \binom{r}{m} \delta_{p-1,\left[2-m\right]} \pmod p.
 \end{eqnarray*}
 \begin{enumerate}
     \item[(i)] For $m=0$, we have \begin{eqnarray*}
     \sum_{\substack{
  2 < j < r \\
    j \equiv 2 \pmod{p-1}
}}
\binom{r }{j} \equiv - \binom{r}{2} \pmod{p}.
\end{eqnarray*}
\item[(ii)] For $m=1$, we have \[
\sum_{\substack{
  2 < j <r \\
    j \equiv 2 \pmod{p-1}
}}
j\binom{r }{j} \equiv -2\binom{r}{2} \pmod{p}.
\]
\item[(iii)] For $m=2$, we have \[
\sum_{\substack{
  2 < j <r \\
    j \equiv 2 \pmod{p-1}
}}
\binom{j}{2}\binom{r }{j} \equiv 0 \pmod{p}.
\]
\end{enumerate}
Consider the element
\[
F(X,Y):=
\sum_{\substack{
2<j<r\\
j\equiv 2 \pmod{p-1}
}}
\binom{r}{j}X^{r-j}Y^j
+A X^{r-2p}Y^{2p}
+B X^{r-p-1}Y^{p+1}.
\]
We determine constants \(A,B\) such that $F(X,Y)\in V_r^{(3)}$. Note that $X^3$ and $Y^3$ divide $F(X,Y)$. By the congruences above, we must solve: \begin{eqnarray*}
    A+B- \binom{r}{2} &=& 0\\
    B-2 \binom{r}{2}&=& 0.
\end{eqnarray*}
The solutions are:
\[
A=-\binom{r}{2},\,\,B=2\binom{r}{2}.
\]
Hence
\begin{align}
\label{l_1,l_2}
\binom{r}{2} \theta^2X^{r-2p-2} \in X_{r-1}+V_r^{(3)}.
\end{align}

  (ii) Similarly, we have \begin{eqnarray*}
        \sum_{\substack{
j \geq 0 \\
j \equiv 2 \pmod{p-1}
}}
 \binom{r-1}{j} X^{r-j} Y^{j} &\in& X_{r-1}, \text{ so}\\
 \binom{r-1}{2}X^{r-2}Y^2+ \sum_{\substack{
2 <j< r \\
j \equiv 2 \pmod{p-1}
}}
 \binom{r-1}{j} X^{r-j} Y^{j} &\in&  X_{r-1}.
    \end{eqnarray*}
    
    By \cite[Lemma $2.14$]{GV22}, we have
  \begin{eqnarray*}
       \sum_{\substack{
  j \geq 0\\
    j \equiv 2 \pmod{p-1}
}}
\binom{j}{m}\binom{r-1}{j} &\equiv& \binom{r-1}{m} \binom{\left[1-m\right]}{\left[2-m\right]} + \binom{r-1}{m} \delta_{p-1,\left[2-m\right]} \pmod p, \text{ so}\\
\sum_{\substack{
  2<j<r \\
    j \equiv 2 \pmod{p-1}
}}
\binom{j}{m}\binom{r-1}{j} &\equiv& -\binom{2}{m}\binom{r-1}{2}+\binom{r-1}{m} \binom{\left[1-m\right]}{\left[2-m\right]}+ \binom{r-1}{m} \delta_{p-1,\left[2-m\right]} \pmod p.
\end{eqnarray*}
\begin{enumerate}
     \item[(i)] For $m=0$ we have \begin{eqnarray*}
     \sum_{\substack{
  2 < j <r \\
    j \equiv 2 \pmod{p-1}
}}
\binom{r -1}{j} \equiv  -\binom{r-1}{2} \pmod{p},
\end{eqnarray*}
\item[(ii)] For $m=1$ we have \[
\sum_{\substack{
  2 < j <r \\
    j \equiv 2 \pmod{p-1}
}}
j\binom{r -1}{j} \equiv -(r-1)^2 \pmod{p},
\]
\item[(iii)] For $m=2$ we have \[
\sum_{\substack{
  2 < j <r \\
    j \equiv 2 \pmod{p-1}
}}
\binom{j}{2}\binom{r-1}{j} \equiv 0 \pmod{p}.
\]
\end{enumerate}
Let us consider\[
G(X,Y) \coloneqq \sum_{\substack{
2 <j< r \\
j \equiv 2 \pmod{p-1}
}}
 \binom{r-1}{j} X^{r-j} Y^{j} +AX^{r-2p}Y^{2p}+BX^{r-p-1}Y^{p+1},
\]
We determine constants \(A,B\) such that $G(X,Y)\in V_r^{(3)}$. Note that $X^3$ and $Y^3$ divide $G(X,Y)$. By the congruences above, we must solve:\begin{eqnarray*}
    -\binom{r-1}{2}+A+B &=& 0\\
    -(r-1)^2+B&=&0.
\end{eqnarray*}
The solutions are \begin{eqnarray*}
    A &=& \binom{r-1}{2}-(r-1)^2\\
    B &=& (r-1)^2.
\end{eqnarray*}
Hence,
\begin{align*}
&\sum_{\substack{
2<j<r\\
j\equiv 2 \pmod{p-1}
}}
\binom{r-1}{j}X^{r-j}Y^j
+\left(\binom{r-1}{2}-(r-1)^2\right)X^{r-2p}Y^{2p}
+(r-1)^2X^{r-p-1}Y^{p+1}
\in V_r^{(3)},
\end{align*}
and therefore
\begin{align}
\label{l_1',l_2'}
-\binom{r-1}{2}X^{r-2}Y^2
&+\left(\binom{r-1}{2}-(r-1)^2\right)X^{r-2p}Y^{2p}
+(r-1)^2X^{r-p-1}Y^{p+1} 
&\in X_{r-1}+V_r^{(3)}.
\end{align}
Finally adding \eqref{l_1,l_2} and \eqref{l_1',l_2'} we have \begin{eqnarray*}
    (r-1)X^{r-2}Y^2 \equiv (r-1) X^{r-p-1}Y^{p+1} \pmod{X_{r-1}+V_r^{(3)}}. \qquad \qquad \qquad  \qedhere
\end{eqnarray*} 
%Since $p\mid r$, we have
%\[
%X^{r-2}Y^2
%\equiv
%X^{r-p-1}Y^{p+1}
%\pmod{X_{r-1}+V_r^{(3)}}.
%\]
\end{proof}
\begin{lemma}\label{lem:im 2}
Let $3p+2 \le r\equiv 2 \pmod{p-1}$ and $p \mid r$, then 
    \[\frac{1}{p}\sum_{\substack{
2 <j< r-2 \\
j \equiv 2 \pmod{p-1}
}}
 \biggl(\binom{r}{j}-\alpha_j\biggl) X^{r-j} Y^{j}-{A}X^{r-2p}Y^{2p}-{B}X^{r-p-1}Y^{p+1} \in V_r^{(3)},
 \]
 where $A= 1+ \frac{1}{p} \binom{r}{2}$, $B= \frac{-2}{p} \binom{r}{2}\ $ and $\alpha_j$ are integers as in Lemma~\ref{lem:2.7}.
  \end{lemma}
\begin{proof}
By Lemma~\ref{lem:2.1} with $(i,a)=(0,2)$, $(1,2)$ and $(2,p+1)$, respectively, we have

%\begin{comment} 
%   We know that \[
%  (p-1)\sum_{\substack{
%j \geq 0 \\
%j \equiv 2 \pmod{p-1}
%}}
% \binom{j}{m}\binom{r}{j} = \binom{r}{m} \sum_{\xi \in \mu'} \xi^{m-2}(1+\xi)^{r-m}.
%    \]
%    So, we have \begin{eqnarray*}
%       (p-1) \sum_{\substack{
%j \geq 0 \\
%j \equiv 2 \pmod{p-1}
%}}
% \binom{j}{m}\binom{r}{j} \equiv %\binom{r}{m} \sum_{\xi \in \mu'} %\xi^{m-2}.(1+\xi)^{2p-m} \pmod{p^2}.
%    \end{eqnarray*}
%\end{comment}

\begin{eqnarray*}
%       \item [(i)] For $m=0$ we have\begin{eqnarray*}
%       \begin{comment}
%           (p-1) \sum_{\substack{
%j \geq 0 \\
%j \equiv 2 \pmod{p-1}
%}}
% \binom{r}{j} &\equiv&  \sum_{\xi \in \mu'} \xi^{-2}(1+\xi)^{2p} \pmod{p^2}\\
% &\equiv& (p-1)\biggl(\binom{2p}{2}+\binom{2p}{p+1}+1\biggl) \pmod{p^2}\\
%  (p-1) \sum_{\substack{
%2<j<r-2  \\
%j \equiv 2 \pmod{p-1}
%}}
% \binom{r}{j} &\equiv& (p-%1)\biggl(\binom{2p}{2}+\binom{2p}{p+1}-\binom{r}{2}\biggl) %\pmod{p^2}\\
%  \end{comment}
 \frac{(p-1)}{p} \sum_{\substack{
2<j<r-2  \\
j \equiv 2 \pmod{p-1}
}}
 \binom{r}{j} &\equiv& (p-1) \biggl(1-\frac{1}{p}\binom{r}{2}\biggl) \pmod{p} \\
 %
 %\item[(ii)] For $m=1$ we have 
 %\begin{eqnarray*}
  %      
 %       \begin{comment}
  %          (p-1) \sum_{\substack{
%j \geq 0 \\
%j \equiv 2 \pmod{p-1}
%}}
%j \binom{r}{j} &\equiv&  r\sum_{\xi \in \mu'} \xi^{-1}(1+\xi)^{2p-1} \pmod{p^2}\\
%&\equiv& r(p-1)\biggl( 2p-1 + %\binom{2p-1}{p}+1\biggl) %\pmod{p^2}\\
% (p-1) \sum_{\substack{
%2<j<r-2 \\
%j \equiv 2 \pmod{p-1}
%}}
%j \binom{r}{j} &\equiv&  r(p-%1)\biggl( 2p-1 + \binom{2p-1}%{p}\biggl)-2(p-1)\binom{r}{2 } %\pmod{p^2}\\
%\end{comment}
\frac{(p-1)}{p} \sum_{\substack{
2<j<r-2 \\
j \equiv 2 \pmod{p-1}
}}
j \binom{r}{j} &\equiv& -2\frac{(p-1)}{p} \binom{r}{2} \pmod{p} \\
% \item[(iii)] For $m=2$ we have 
 \frac{(p-1)}{p} \sum_{\substack{
2<j<r-2 \\
j \equiv 2 \pmod{p-1}
}}
\binom{j}{2} \binom{r}{j} &\equiv& 0 \pmod{p}.
 \end{eqnarray*}
 Consider \[ 
  H(X,Y)=\frac{1}{p}\sum_{\substack{
2 <j< r-2 \\
j \equiv 2 \pmod{p-1}
}}
\biggl( \binom{r}{j}-\alpha_j \biggl) X^{r-j} Y^{j}-AX^{r-2p}Y^{2p}-BX^{r-p-1}Y^{p+1}.\]
We determine constants \(A,B\) such that $H(X,Y)\in V_r^{(3)}$. Note that $X^3$ and $Y^3$ divide $H(X,Y)$. By the congruences above, we must solve:
 \begin{eqnarray*}
     1-\frac{1}{p}\binom{r}{2}-A-B &=& 0\\
     -\frac{2}{p}\binom{r}{2}-B &=& 0.
\end{eqnarray*}
Hence, we have \begin{eqnarray*}
    A &=& 1+\frac{1}{p}\binom{r}{2}\\
    B &=& -\frac{2}{p}\binom{r}{2},
\end{eqnarray*}
proving the lemma.
\end{proof}
\begin{lemma}\label{lem: im 3}
Let $3p+2 \le r \equiv 2 \pmod{p-1}$ and $p \mid r-1$, then \[
\frac{1}{p}\sum_{\substack{
2 <j< r-2 \\
j \equiv 2 \pmod{p-1}
}}
 \biggl(\binom{r}{j}-\alpha_j\biggl) X^{r-j} Y^{j}-AX^{r-2p}Y^{2p}-BX^{r-p-1}Y^{p+1} \in V_r^{(3)},
\]   
where $A=-\frac{1}{2}+\frac{1}{p} \binom{r}{2}$, $B= 1-\frac{2}{p} \binom{r}{2}\ $ and $\alpha_j$ are integers as in Lemma~\ref{lem:2.7}.
\end{lemma}
\begin{proof}
By Lemma~\ref{lem:2.1} with $(i,a)=(0,2)$, $(1,2)$ and $(2,p+1)$, respectively, we have 
        %We know that \[
  %(p-1)\sum_{\substack{
%j \geq 0 \\
%j \equiv 2 \pmod{p-1}
%}}
 %\binom{j}{m}\binom{r}{j} = \binom{r}{m} \sum_{\xi \in \mu'} \xi^{m-2}.(1+\xi)^{r-m}.
 %   \]
       % So, we have \begin{eqnarray*}
   %    (p-1) \sum_{\substack{
%j \geq 0 \\
%j \equiv 2 \pmod{p-1}
%}}
% \binom{j}{m}\binom{r}{j} \equiv \binom{r}{m} \sum_{\xi \in \mu'} \xi^{m-2}.(1+\xi)^{p+1-m} \pmod{p^2}.
   % \end{eqnarray*}
    \begin{eqnarray}
              % \item [(i)] For $m=0$ we have \begin{eqnarray*}
       % (p-1) \sum_{\substack{
%j \geq 0 \\
%j \equiv 2 \pmod{p-1}
%}}
 %\binom{r}{j} &\equiv&  \sum_{\xi \in \mu'} \xi^{-2}.(1+\xi)^{p+1} \pmod{p^2}\\
 %&\equiv& (p-1)\biggl(\binom{p+1}{2}+1\biggl) \pmod{p^2}\\
 % (p-1) \sum_{\substack{
%2<j<r-2  \\
%j \equiv 2 \pmod{p-1}
%}}
% \binom{r}{j} &\equiv& (p-1)\biggl(\binom{p+1}{2}-\binom{r}{2}\biggl) \pmod{p^2}\\
\label{lem: needed for a=3 cl 1}
 \frac{p-1}{p} \sum_{\substack{
2<j<r-2  \\
j \equiv 2 \pmod{p-1}
}}
 \binom{r}{j} &\equiv& (p-1) \biggl(\frac{1}{2}-\frac{1}{p}\binom{r}{2}\biggl) \pmod{p}\\
 %\item[(ii)] For $m=1$ we have \begin{eqnarray*}
    %    (p-1) \sum_{\substack{
%j \geq 0 \\
%j \equiv 2 \pmod{p-1}
%}}
%j \binom{r}{j} &\equiv&  r\sum_{\xi \in \mu'} \xi^{-1}.(1+\xi)^{p} \pmod{p^2}\\
%&\equiv& r(p-1)( p+1) \pmod{p^2}\\
 %(p-1) \sum_{\substack{
%2<j<r-2 \\
%j \equiv 2 \pmod{p-1}
%}} 
%j \binom{r}{j} &\equiv&  r(p-1) p-2(p-1)\binom{r}{2 } \pmod{p^2}\\
\label{lem: needed for a=3 cl 2}
\frac{p-1}{p} \sum_{\substack{
2<j<r-2 \\
j \equiv 2 \pmod{p-1}
}}
j \binom{r}{j} &\equiv& (p-1)-2\frac{(p-1)}{p} \binom{r}{2} \pmod{p}\\
 %\item[(iii)] For $m=2$ we have \[
 \frac{p-1}{p} \sum_{\substack{
2<j<r-2 \\
j \equiv 2 \pmod{p-1}
}}
\binom{j}{2} \binom{r}{j} &\equiv& 0 \pmod{p} \nonumber.
 \end{eqnarray}
   Consider \[ 
  H(X,Y)=\frac{1}{p}\sum_{\substack{
2 <j< r-2 \\
j \equiv 2 \pmod{p-1}
}}
\biggl( \binom{r}{j}-\alpha_j\biggl) X^{r-j} Y^{j}-AX^{r-2p}Y^{2p}-BX^{r-p-1}Y^{p+1}.\]
We determine constants \(A,B\) such that $H(X,Y)\in V_r^{(3)}$. Note that $X^3$ and $Y^3$ divide $H(X,Y)$. By the congruences above, we must solve:
 \begin{eqnarray*}
     \frac{1}{2}-\frac{1}{p}\binom{r}{2}-A-B &=& 0\\
     1-\frac{2}{p}\binom{r}{2}-B &=& 0.
\end{eqnarray*}
Hence, we have \begin{eqnarray*}
    A &=& -\frac{1}{2}+\frac{1}{p}\binom{r}{2}\\
    B &=& 1-\frac{2}{p}\binom{r}{2},
\end{eqnarray*}
proving the lemma.
\end{proof}

%\begin{lemma}\label{lem: important for r-2p}

%\end{proof}

\begin{prop}\label{lem:G_1}
Let $3p+2 \le r \equiv 2 \pmod{p-1}$ and $v_p(a_p)=2$, then $F_1=0$.
\end{prop}
\begin{proof}
 Consider $f_0 \in \mathrm{ind}_{KZ}^G \mathrm{Sym}^r \bar{\mathbb{Q}}_p^2$ defined by 
     \begin{eqnarray*}
         f_0=\biggl[\mathrm{Id},\frac{\theta^2X^{r-3p-1}Y^{p-1}-\theta^2X^{r-2p-2}}{p^2}\bigg].
    \end{eqnarray*}
     Clearly $T^-f_0$ dies mod $p$. We have 
     \begin{eqnarray*}
          T^+f_0= \biggl[g^0_{1,0}, -X^{r-2}Y^2\biggl].
     \end{eqnarray*}
     Hence we have \begin{eqnarray*}
         (T-a_p)f_0=-\biggl[g^0_{1,0},X^{r-2}Y^2\biggl]-\frac{a_p}{p^2}\biggl[\mathrm{Id}, \theta^2 X^{r-3p-1}Y^{p-1}-\theta^2X^{r-2p-2}\biggl].
     \end{eqnarray*}
     So $(T-a_p)f$ is integral and maps to $- \biggl[g^0_{1,0},X^{p-3}\biggl]$ in $\mathrm{ind}_{KZ}^G J_1$ by Lemma~\ref{lem:2.10}. This is a generator of $\mathrm{ind}_{KZ}^G J_1$ as $G$-module. Hence, $F_1=0$.
     \end{proof}

\begin{prop}\label{lem:G_2}
    Let $2p+1 \le r \equiv 2 \pmod{p-1}$ and $p \nmid r(r-1)$. Then \begin{itemize}
        \item [(i)] $F_2=0$, $F_3=0$.
        \item[(ii)] $F_4=0$. 
    \end{itemize}
   \end{prop}
   \begin{proof}\label{lem:G_3}
       (i) By \cite[Proposition $4.22$]{GV22}, we have $\frac{X_{r-1}^{*}}{X_{r-1}^{**}} \cong \frac{V_r^{*}}{V_r^{**}}$, noting that ${\mathcal J}(2,1) = \{0, 1\}$ there.

       (ii) By  \cite[Lemma 4.4]{GV22}, we have $\frac{X_{r-1}^{**}}{X_{r-1}^{***}} =\frac{X_{r}^{**}}{X_{r}^{***}} $. Now by  \cite[Proposition 4.8]{GV22}, we have $\frac{X_{r}^{**}}{X_{r}^{***}} \cong V_{p-3} \otimes D^2$. So the socle of $\frac{V_r^{**}}{V_r^{***}}$ dies. Alternatively, use Lemma~\ref{lem:cong 2 f_1},  part (i). 
   \end{proof}

\begin{prop}\label{prop: alpha 1}
    Let $3p+2 \le r \equiv2 \pmod{p-1}$. If $p\mid r(r-1)$, then $F_5=0$. %and \begin{itemize}
       % \item [(i)] if $p \mid r$ and $p^2 \mid r-2p$, then $U_3=0$
       % \item[(ii)] if $p \mid r-1$ and $p^2 \mid r-p-1$, then $U_3=0$. 
    %\end{itemize}
\end{prop}
\begin{proof}
 (i)  Let $p \mid r$ and consider $f_0,f_1,f_2 \in \mathrm{ind}_{KZ}^G \mathrm{Sym}^r \bar{\mathbb{Q}}_p^2$ defined by
    \begin{eqnarray*}
        f_2 &=& \sum_{\mu \in {\mathbb{F}_p^\times}} \biggl[g^0_{2,p\left[\mu \right]}, \frac{Y^r-X^{r-2p}Y^{2p}}{p}\biggl]+(1-p)\biggl[g^0_{2,0}, \frac{Y^r-X^{r-2p}Y^{2p}}{p}\biggl]\\
         f_1 &=& \biggl[g^0_{1,0},\frac{p-1}{p\cdot a_p} \sum_{\substack{
2<j<r \\
j \equiv 2 \pmod{p-1}
}}
 \alpha_j \cdot X^{r-j} Y^{j}\biggl]\\
 f_0 &=& \biggl[\mathrm{Id},\frac{\theta^2 X^{r-2p-2}-\theta^2 X^{r-3p-1}Y^{p-1}}{p^2}\biggl],
    \end{eqnarray*}
    where $\alpha_j$ are integers as in Lemma~\ref{lem:2.7}.  Clearly $T^+f_2$, $-a_pf_2$, $T^+f_1$, $T^-f_1$ and $T^-f_0$ die in mod $p$. We compute that \begin{eqnarray*}
        T^-f_2&=&  \biggl[g^0_{1,0},\frac{p-1}{p} \sum_{\substack{
0<j<r \\
j \equiv 2 \pmod{p-1}
}}
 \binom{r}{j} X^{r-j} Y^{j}\biggl], \text{ so}\\
 T^-f_2-a_pf_1 &=&\biggl[g^0_{1,0},\frac{p-1}{p} \binom{r}{2}X^{r-2}Y^2 +\frac{p-1}{p} \sum_{\substack{
2<j<r \\
j \equiv 2 \pmod{p-1}
}}
 \biggl(\binom{r}{j}-\alpha_j\biggl) X^{r-j} Y^{j}\biggl]\\
  &=&  \biggl[g^0_{1,0},\frac{p-1}{p} \binom{r}{2}X^{r-2}Y^2 + \biggl(\frac{p-1}{p} \sum_{\substack{
2<j<r \\ 
j \equiv 2 \pmod{p-1}
}}
 \biggl(\binom{r}{j}-\alpha_j\biggl) X^{r-j} Y^{j}\\
 && \qquad -{(p-1)}\biggl(1+\frac{1}{p}\binom{r}{2}\biggl)X^{r-2p}Y^{2p}+\frac{2(p-1) }{p}\binom{r}{2} X^{r-p-1}Y^{p+1}\biggl)\\
 && \qquad + {(p-1)}\biggl(1+\frac{1}{p}\binom{r}{2}\biggl)X^{r-2p}Y^{2p}-\frac{2(p-1) }{p}\binom{r}{2} X^{r-p-1}Y^{p+1}\biggl]\\
 &=& \biggl[g^0_{1,0},(p-1)X^{r-2p}Y^{2p}+ \frac{p-1}{p}\binom{r}{2}\theta^2 X^{r-2p-2}\biggl] \pmod{V_r^{(3)}},
 \end{eqnarray*}
since the term in parantheses dies by Lemma~\ref{lem:im 2}.
 We have \[
 T^+f_0= \biggl[g^0_{1,0}, X^{r-2}Y^2\biggl].
 \]
 Hence, we have \begin{eqnarray*}
     T^-f_2-a_pf_1+T^+f_0  &\equiv& \biggl[g^0_{1,0},(p-1)X^{r-2p}Y^{2p}-(p-1)X^{r-2}Y^2+\frac{(p-1)}{p}\binom{r}{2}\theta^2X^{r-2p-2}\biggl]\\
     &=& \biggl[g^0_{1,0}, (p-1)(X^{r-2p}Y^{2p}-2X^{r-2}Y^2+X^{r-2}Y^2)+\frac{(p-1)}{p}\binom{r}{2}\theta^2X^{r-2p-2}\biggl]\\
     &\equiv& \biggl[g^0_{1,0}, (p-1)(X^{r-2p}Y^{2p}-2X^{r-p-1}Y^{p+1}+X^{r-2}Y^2)+\frac{(p-1)}{p}\binom{r}{2}\theta^2X^{r-2p-2}\biggl]\\
     %&=& \biggl[g^0_{1,0}, (p-1)\theta^2 X^{r-2p-2}+\frac{(p-1)}{p}\binom{r}{2}\theta^2X^{r-2p-2}\biggl]\\
     %&=& \biggl[g^0_{1,0}, (p-1)\biggl(1+\frac{1}{p}\binom{r}{2}\biggl)\theta^2 X^{r-2p-2}\biggl]\\
      &=& \biggl[g^0_{1,0},\frac{r-2p}{2p}\theta^2X^{r-2p-2}\biggl] \pmod{X_{r-1}+V_r^{(3)}},
     \end{eqnarray*}
     where in the third line we have used Lemma~\ref{lem:cong 2 f_1} part $(ii)$.
     Finally we have 
     \begin{eqnarray*}
         (T-a_p)(f_0+f_1+f_2)= \frac{r-2p}{2p} \biggl[g^0_{1,0},\theta^2X^{r-2p-2}\biggl]-\frac{a_p}{p^2}\biggl[\mathrm{Id}, \theta^2X^{r-2p-2}-\theta^2 X^{r-3p-1}Y^{p-1}\biggl].
    \end{eqnarray*}
     So $(T-a_p)(f_0+f_1+f_2)$ is integral and maps to $\frac{a_p}{p^2}\biggl[\mathrm{Id}, Y^2\biggl]$ in $\mathrm{ind}_{KZ}^G J_5$ which generates this as $G$-module. Hence $F_5=0$.

(ii) Now let $p \mid r-1$ and consider $f_0,f_1,f_2\in \mathrm{ind}_{KZ}^G \mathrm{Sym}^r \bar{\mathbb{Q}}_p^2$ defined by \begin{eqnarray*}
        f_2 &=& \sum_{\mu \in {\mathbb{F}_p^\times}} \biggl[g^0_{2,p\left[\mu \right]}, \frac{Y^r-X^{r-p-1}Y^{p+1}}{p}\biggl]+(1-p)\biggl[g^0_{2,0}, \frac{Y^r-X^{r-p-1}Y^{p+1}}{p}\biggl]\\
         f_1 &=&  \biggl[g^0_{1,0},\frac{p-1}{p \cdot a_p} \sum_{\substack{
         2<j<r \\
         j \equiv 2 \pmod{p-1}
         }}
         \alpha_j \cdot X^{r-j} Y^{j}\biggl]\\
        f_0 &=& \biggl[\mathrm{Id},\frac{\theta^2 X^{r-2p-2}-\theta^2 X^{r-3p-1}Y^{p-1}}{2p^2}\biggl],
    \end{eqnarray*}
where $\alpha_j$ are integers as in Lemma~\ref{lem:2.7}.  Clearly $T^+f_2$, $-a_pf_2$, $T^+f_1$, $T^-f_1$ and $T^-f_0$ die mod $p$. We have 
\begin{eqnarray*}
T^-f_2&=&  \biggl[g^0_{1,0},\frac{p-1}{p} \sum_{\substack{
0<j<r \\
j \equiv 2 \pmod{p-1}
}}
 \binom{r}{j} X^{r-j} Y^{j}\biggl], \text{ so}\\
 T^-f_2-a_pf_1 &=&  \biggl[g^0_{1,0},\frac{p-1}{p} \binom{r}{2}X^{r-2}Y^2 +\frac{p-1}{p} \sum_{\substack{
2<j<r \\
j \equiv 2 \pmod{p-1}
}}
 \biggl(\binom{r}{j}-\alpha_j\biggl) X^{r-j} Y^{j}\biggl]\\
   &=& \biggl[g^0_{1,0},\frac{p-1}{p} \binom{r}{2}X^{r-2}Y^2 +\biggl( \frac{p-1}{p} \sum_{\substack{
2<j<r \\
j \equiv 2 \pmod{p-1}
}}
 \biggl(\binom{r}{j}-\alpha_j\biggl) X^{r-j} Y^{j}\\
 && \qquad -(p-1)\biggl(\frac{1}{p}\binom{r}{2}-\frac{1}{2}\biggl)X^{r-2p}Y^{2p}-(p-1)\biggl(1-\frac{2}{p}\binom{r}{2}\biggl)X^{r-p-1}Y^{p+1}\biggl)\\
 && \qquad +(p-1)\biggl(\frac{1}{p}\binom{r}{2}-\frac{1}{2}\biggl)X^{r-2p}Y^{2p}+(p-1)\biggl(1-\frac{2}{p}\binom{r}{2}\biggl)X^{r-p-1}Y^{p+1}\biggl]\\
 &=& \biggl[g^0_{1,0},\frac{p-1}{p}\binom{r}{2} \theta^2 X^{r-2p-2}-\frac{p-1}{2}X^{r-2p}Y^{2p}+(p-1)X^{r-p-1}Y^{p+1} \biggl],
\end{eqnarray*}
since the term in parantheses dies by Lemma~\ref{lem: im 3}. We have \[
T^+f_0 = \frac{1}{2}\biggl[g^0_{1,0}, X^{r-2}Y^2\biggl].
\]
Hence, we get \begin{eqnarray*}
    T^-f_2-a_pf_1+T^+f_0 &=& \biggl[g^0_{1,0},\frac{p-1}{p}\binom{r}{2} \theta^2 X^{r-2p-2}-\frac{p-1}{2}\theta^2X^{r-2p-2} \biggl]\\
    &=&  \biggl[g^0_{1,0},\frac{p+1-r}{2p}\theta^2X^{r-2p-2} \biggl].
\end{eqnarray*}
Combining these, we have \begin{eqnarray*}
    (T-a_p)(f_0+f_1+f_2)=\frac{p+1-r}{2p}\biggl[g^0_{1,0},\theta^2X^{r-2p-2} \biggl]-\frac{a_p}{2p^2}\biggl[\mathrm{Id},{\theta^2 X^{r-2p-2}-\theta^2X^{r-3p-1}Y^{p-1}}\biggl].
\end{eqnarray*}
%Hence we have \begin{eqnarray*}
         %(T-a_p)f_0=-\biggl[g^0_{1,0},X^{r-2}Y^2\biggl]-\frac{a_p}{p^2}\biggl[\mathrm{Id}, \theta^2 X^{r-3p-1}Y^{p-1}-\theta^2X^{r-2p-2}\biggl].
        % \end{eqnarray*}
        
So $(T-a_p)(f_0+f_1+f_2)$ is integral and maps to $\frac{a_p}{2p^2}\biggl[\mathrm{Id}, Y^2\biggl]$ in $\mathrm{ind}_{KZ}^G J_5$ which generates this as a $G$-module. Hence $F_5=0$.
\end{proof}

\begin{prop}\label{lem:G_4}
  Let $3p+2 \le r \equiv 2 \pmod{p-1}$, assume that $v_p(a_p)=2$ and $p \mid r$. %and $p \mid r$ but $p^2 \nmid r-2p$, so that $u\coloneqq \frac{2p}{r-2p}\cdot \frac{a_p}{p^2}$ is a $p$-adic unit. 
  Then \begin{itemize}
      \item [(i)] if $p^2 \mid r-2p$, then $F_4=0$ and if $p^2\nmid r-2p$, then $F_4$ is a quotient of $\pi(p-3,\bar{u},\omega^2)$ 
      \item[(ii)] if $p^2\nmid r-2p$, then $F_3$ is a quotient of $\pi(0,\bar{u}^{-1},\omega)$, 
  \end{itemize}
  where $u \coloneqq \frac{2p}{r-2p} \cdot \frac{a_p}{p^2}$.
\end{prop}
\begin{proof}
(i)    Let us consider $f_0,f_1,f_2 \in \mathrm{ind}_{KZ}^G \mathrm{Sym}^r \bar{\mathbb{Q}}_p^2$ defined by
    \begin{eqnarray*}
        f_2 &=& \sum_{{\lambda \in \mathbb{F}_p},\,\, \mu \in {\mathbb{F}_p^\times}} \biggl[g^0_{2,\left [\lambda \right]+p\left[\mu \right]}, \frac{Y^r-X^{r-2p}Y^{2p}}{p}\biggl]+(1-p)\sum_{\lambda \in \mathbb{F}_p}\biggl[g^0_{2,\left [\lambda \right]}, \frac{Y^r-X^{r-2p}Y^{2p}}{p}\biggl]\\
         f_1 &=& \sum_{\lambda \in {\mathbb{F}_p}} \biggl[g^0_{1,\left[\lambda \right]},\frac{p-1}{p \cdot a_p} \sum_{\substack{
2<j<r \\
j \equiv 2 \pmod{p-1}
}}
 \alpha_j \cdot X^{r-j} Y^{j}\biggl]\\
 f_0 &=& \biggl[\mathrm{Id},\frac{\theta^2 X^{r-2p-2}}{p^2}\biggl],
    \end{eqnarray*}
     where $\alpha_j$ are integers as in Lemma~\ref{lem:2.7}.
    %Clearly $T^+f_2$, $-a_pf_2$, $T^+f_1$, $T^-f_1$ and $T^-f_0$ die mod $p$. 
    As in the proof of Proposition~\ref{prop: alpha 1} part $(i)$, we obtain 
    \begin{eqnarray*}
    %We compute that \begin{eqnarray*}
        %T^-f_2&=& \sum_{\lambda \in {\mathbb{F}_p}} \biggl[g^0_{1,\left[\lambda \right]},\frac{p-1}{p} \sum_{\substack{
%0<j<r \\
%j \equiv 2 \pmod{p-1}
%}}
% \binom{r}{j} X^{r-j} Y^{j}\biggl]\\
 T^-f_2-a_pf_1  
 %\sum_{\lambda \in {\mathbb{F}_p}} \biggl[g^0_{1,\left[\lambda \right]},\frac{p-1}{p} \binom{r}{2}X^{r-2}Y^2 +\frac{p-1}{p} \sum_{\substack{
%2<j<r \\
%j \equiv 2 \pmod{p-1}
%}}
 %\biggl(\binom{r}{j}-\alpha_j\biggl) X^{r-j} Y^{j}\biggl]\\
 %&=& \sum_{\lambda \in {\mathbb{F}_p}} \biggl[g^0_{1,\left[\lambda \right]},\frac{p-1}{p} \binom{r}{2}X^{r-2}Y^2 + \biggl(\frac{p-1}{p} \sum_{\substack{
%2<j<r \\ 
%j \equiv 2 \pmod{p-1}
%}}
% \biggl(\binom{r}{j}-\alpha_j\biggl) X^{r-j} Y^{j}\\
% &-&{(p-1)}\biggl(1+\frac{1}{p}\binom{r}{2}\biggl)X^{r-2p}Y^{2p}+{(p-1)}\biggl(\frac{2}{p} \binom{r}{2}\biggl) X^{r-p-1}Y^{p+1}\biggl)\\
% &+& {(p-1)}\biggl(1+\frac{1}{p}\binom{r}{2}\biggl)X^{r-2p}Y^{2p}-{(p-1)}\biggl(\frac{2 }{p}\binom{r}{2}\biggl) X^{r-p-1}Y^{p+1}\biggl]\\
 &=& \sum_{\lambda \in {\mathbb{F}_p}} \biggl[g^0_{1,\left[\lambda \right]},(p-1)X^{r-2p}Y^{2p}+ \frac{p-1}{p}\binom{r}{2}\theta^2 X^{r-2p-2}\biggl].
 \end{eqnarray*}
%since the term in parantheses dies by Lemma~\ref{lem:im 2}.
 We have \[
 T^+f_0= \sum_{\lambda \in {\mathbb{F}_p}} \biggl[g^0_{1,\left[\lambda \right]}, X^{r-2}Y^2\biggl].
 \]
 Again, we obtain \begin{eqnarray*}
     T^-f_2-a_pf_1+T^+f_0   %\sum_{\lambda \in {\mathbb{F}_p}} \biggl[g^0_{1,\left[\lambda \right]},(p-1)X^{r-2p}Y^{2p}-(p-1)X^{r-2}Y^2+\frac{(p-1)}{p}\binom{r}{2}\theta^2X^{r-2p-2}\biggl]\\
    % &=& \sum_{\lambda \in {\mathbb{F}_p}} \biggl[g^0_{1,\left[\lambda \right]}, (p-1)(X^{r-2p}Y^{2p}-2X^{r-2}Y^2+X^{r-2}Y^2)+\frac{(p-1)}{p}\binom{r}{2}\theta^2X^{r-2p-2}\biggl]\\
    % &=& \sum_{\lambda \in {\mathbb{F}_p}} \biggl[g^0_{1,\left[\lambda \right]}, (p-1)(X^{r-2p}Y^{2p}-2X^{r-p-1}Y^{p+1}+X^{r-2}Y^2)+\frac{(p-1)}{p}\binom{r}{2}\theta^2X^{r-2p-2}\biggl]\\
     %&=&  \sum_{\lambda \in {\mathbb{F}_p}} \biggl[g^0_{1,\left[\lambda \right]}, (p-1)\theta^2 X^{r-2p-2}+\frac{(p-1)}{p}\binom{r}{2}\theta^2X^{r-2p-2}\biggl]\\
    % &=& \sum_{\lambda \in {\mathbb{F}_p}} \biggl[g^0_{1,\left[\lambda \right]}, (p-1)\biggl(1+\frac{1}{p}\binom{r}{2}\biggl)\theta^2 X^{r-2p-2}\biggl]\\
      &=& \sum_{\lambda \in {\mathbb{F}_p}} \biggl[g^0_{1,\left[\lambda \right]},\frac{r-2p}{2p}\theta^2X^{r-2p-2}\biggl].
     \end{eqnarray*}
      %where in the third line we have used Lemma~\ref{lem:cong 2 f_1} part $(ii)$.
     Combining these we have \[
     (T-a_p)(f_0+f_1+f_2)= \frac{r-2p}{2p}\sum_{\lambda \in {\mathbb{F}_p}} \biggl[g^0_{1,\left[\lambda \right]},\theta^2X^{r-2p-2}\biggl]-\frac{a_p}{p^2}\biggl[\mathrm{Id}, \theta^2X^{r-2p-2}\biggl],
     \]
     which is integral.
     Hence, if $p^2\mid r-2p$, then ${(T-a_p)(f_0+f_1+f_2)}$ maps to $ -\overline{\frac{a_p}{p^2}}\biggl[\mathrm{Id}, X^{p-3}\biggl]$ in $\mathrm{ind}_{KZ}^G J_4$. In particular $F_4=0$. If $p^2 \nmid r-2p$, then ${(T-a_p)(f_0+f_1+f_2)}$ maps to $\overline{\frac{r-2p}{2p}}(T-\bar{u})\biggl[\mathrm{Id}, X^{p-3}\biggl]$ in $\mathrm{ind}_{KZ}^G J_4$. 

    (ii) Let us consider $f_1,f_2\in \mathrm{ind}_{KZ}^G \mathrm{Sym}^r \bar{\mathbb{Q}}_p^2$ defined by
    \begin{eqnarray*}
        f_2 &=& \sum_{\lambda\in \mathbb{F}_p^{\times}} \biggl[g^0_{2,p\left[\lambda\right]},\frac{\left[\lambda\right]^{p-2}(Y^r-X^{r-2p}Y^{2p})}{p}\biggl]%+\,{r(1-p)}\biggl[g^0_{2,0},\frac{XY^{r-1}-X^{r-2p+1}Y^{2p-1}}{p^2}\biggl]\\
        \\
        f_1 &=& \biggl[g^0_{1,{0}},\frac{p-1}{p \cdot a_p} 
\sum_{\substack{
1<j<r-1 \\
j \equiv 1 \pmod{p-1}
}}
 \alpha_j \cdot  X^{r-j} Y^{j} \biggl],
    \end{eqnarray*}
    where $\alpha_j$ are integers in Lemma~\ref{lem:2.8}.
    It is clear that $T^+f_2$, $-a_pf_2$, $T^+f_1$ and $T^-f_1$ die mod $p$. We compute that \begin{eqnarray*}
         T^-f_2&=& \biggl[g^0_{1,{0}}, \frac{p-1}{p} 
         \sum_{\substack{
         1<j<r-1 \\
         j \equiv 1 \pmod{p-1}
          }}
         \binom{r}{j} X^{r-j} Y^{j} \biggl], \text{ so}\\
         T^-f_2-a_pf_1&=& \biggl[g^0_{1,{0}},\frac{2(p-1)}{p}X^{r-p}Y^p+\frac{p-1}{p} 
         \sum_{\substack{
         1<j<r-1 \\
         j \equiv 1 \pmod{p-1}
         }}
        \biggl(\binom{r}{j}''- \alpha_j\biggl) X^{r-j} Y^{j} \biggl],
    \end{eqnarray*}
    where $\binom{r}{p}''=\binom{r}{p}-2$ and $\binom{r}{j}''=\binom{r}{j}$ for $j>p$.

Now let us consider $f_1',f_2',f_0\in \mathrm{ind}_{KZ}^G \mathrm{Sym}^r \bar{\mathbb{Q}}_p^2$ defined by
 \begin{eqnarray*}
     f_2' &=& \sum_{\lambda\in \mathbb{F}_p^\times} \biggl[g^0_{2,p\left[\lambda\right]},\frac{XY^{r-1}-X^{r-2p+1}Y^{2p-1}}{p^2}\biggl]+(1-p)\biggl[g^0_{2,0},\frac{XY^{r-1}-X^{r-2p+1}Y^{2p-1}}{p^2}\biggl]\\
     f_1' &=& \biggl[g^0_{1,{0}},\frac{p-1}{p \cdot a_p} 
\sum_{\substack{
1<j<r-1 \\
j \equiv 1 \pmod{p-1}
}}
 \alpha_j'\cdot X^{r-j} Y^{j} \biggl]\\
 f_0'&=& \biggl[\mathrm{Id}, \frac{X^{r-1}Y-3X^{r-p}Y^p+3X^{r-2p+1}Y^{2p-1}-X^{r-3p+2}Y^{3p-2}}{p^2}\biggl],
 \end{eqnarray*}
 where $\alpha_j'$ are integers in Lemma~\ref{lem:2.9}.
 Clearly $T^+f_2'$, $T^+f_1'$, $T^-f_1'$ and $T^-f_0'$ die mod $p$. We have 
 \begin{eqnarray*}
     T^-f_2' &=& \biggl[g^0_{1,{0}}, \frac{p-1}{p} 
\sum_{\substack{
1\leq j<r-1 \\
j \equiv 1 \pmod{p-1}
}}
 \binom{r-1}{j} X^{r-j} Y^{j} \biggl]\\
  &=& \biggl[g^0_{1,{0}},\frac{(p-1)(r-1)}{p} X^{r-1}Y+ \frac{p-1}{p} 
\sum_{\substack{
1< j<r-1 \\
j \equiv 1 \pmod{p-1}
}}
 \binom{r-1}{j} X^{r-j} Y^{j} \biggl]\\
   &=& \biggl[g^0_{1,{0}},\frac{(p-1)(r-1)}{p} X^{r-1}Y-\frac{(p-1)(r-1)}{p}X^{r-p}Y^p\\
  && \qquad + \frac{p-1}{p} 
\sum_{\substack{
1< j<r-1 \\
j \equiv 1 \pmod{p-1}
}}
 \binom{r-1}{j}'' X^{r-j} Y^{j} \biggl].
 \end{eqnarray*}
 So \begin{eqnarray*}
 T^-f_2'-a_pf_1' &=& \biggl[g^0_{1,{0}},\frac{(p-1)(r-1)}{p} X^{r-1}Y-\frac{(p-1)(r-1)}{p}X^{r-p}Y^p\\
  && \qquad + \frac{p-1}{p} 
\sum_{\substack{
1< j<r-1 \\
j \equiv 1 \pmod{p-1}
}}
 \biggl(\binom{r-1}{j}''-\alpha_j'\biggl) X^{r-j} Y^{j} \biggl],
   \end{eqnarray*}
   where $\binom{r-1}{p}''=\binom{r-1}{p}+r-1$ and $\binom{r-1}{j}'' =\binom{r-1}{j}$ for $j>p$.
 On the other hand, we have
 \begin{eqnarray*}
     T^+f_0'-a_pf_0'&=& \biggl[g^0_{1,0}, \frac{X^{r-1}Y}{p}\biggl]-\frac{a_p}{p^2}\biggl[\mathrm{Id},{X^{r-1}Y-3X^{r-p}Y^p+3X^{r-2p+1}Y^{2p-1}-X^{r-3p+2}Y^{3p-2}}\biggl].
 \end{eqnarray*}
 Let us define $f=(r-1)(f_1+f_2)+2(f_1'+f_2')-2(r-1)(p-1)f_0'$. Combining these results, we have  \begin{eqnarray*}
    (T-a_p)f &=& \biggl[g^0_{1,{0}},\frac{(r-1)(p-1)}{p} 
\sum_{\substack{
1<j<r-1 \\
j \equiv 1 \pmod{p-1}
}}
\biggl(\binom{r}{j}''- \alpha_j\biggl) X^{r-j} Y^{j}\\
&& \qquad + \frac{2(p-1)}{p} 
\sum_{\substack{
1< j<r-1 \\
j \equiv 1 \pmod{p-1}
}}
 \biggl(\binom{r-1}{j}''-\alpha_j'\biggl) X^{r-j} Y^{j}\biggl]\\
 && \qquad + \frac{2(p-1)(r-1)a_p}{p^2}\biggl[\mathrm{Id},{X^{r-1}Y-3X^{r-p}Y^p+3X^{r-2p+1}Y^{2p-1}-X^{r-3p+2}Y^{3p-2}}\biggl]\\
 && \qquad -\frac{2a_p}{p^2}\sum_{\lambda\in \mathbb{F}_p} \biggl[g^0_{2,p\left[\lambda\right]},XY^{r-1}-X^{r-2p+1}Y^{2p-1}\biggl].
 \end{eqnarray*}
 Adding and subtracting, we get \begin{eqnarray*}
 (T-a_p)f &=&  \biggl[g^0_{1,{0}},\frac{(r-1)(p-1)}{p} 
\sum_{\substack{
1<j<r-1 \\
j \equiv 1 \pmod{p-1}
}}
\biggl(\binom{r}{j}''- \alpha_j\biggl) (X^{r-j} Y^{j}-X^{r-p}Y^p)\\
&& \qquad +\frac{(r-1)(p-1)}{p} 
\sum_{\substack{
1<j<r-1 \\
j \equiv 1 \pmod{p-1}
}}
\biggl(\binom{r}{j}''- \alpha_j\biggl) X^{r-p} Y^{p}\\
&& \qquad + \frac{2(p-1)}{p} 
\sum_{\substack{
1< j<r-1 \\
j \equiv 1 \pmod{p-1}
}}
 \biggl(\binom{r-1}{j}''-\alpha_j'\biggl) (X^{r-j} Y^{j}-X^{r-p}Y^p)\\
 && \qquad +\frac{2(p-1)}{p} 
\sum_{\substack{
1< j<r-1 \\
j \equiv 1 \pmod{p-1}
}}
 \biggl(\binom{r-1}{j}''-\alpha_j'\biggl) X^{r-p} Y^{p}\biggl]\\
 && \qquad +  \frac{2(p-1)(r-1)a_p}{p^2}\biggl[\mathrm{Id},{X^{r-1}Y-3X^{r-p}Y^p+3X^{r-2p+1}Y^{2p-1}-X^{r-3p+2}Y^{3p-2}}\biggl]\\
 && \qquad -\frac{2a_p}{p^2}\sum_{\lambda\in \mathbb{F}_p} \biggl[g^0_{2,p\left[\lambda\right]},XY^{r-1}-X^{r-2p+1}Y^{2p-1}\biggl].
 \end{eqnarray*}
 We have $X^{r-j}Y^j-X^{r-p}Y^p \equiv j \theta X^{r-2p}Y^{p-1} \pmod{V_r^{(2)}}$ by \cite[Lemma 2.12]{GV22}.
 %, and \begin{eqnarray*}
     %\frac{p-1}{p}\sum_{\substack{
%1<j<r-1 \\
%j \equiv 1 \pmod{p-1}
%}}
%j\biggl(\binom{r}{j}''- \alpha_j\biggl) &\equiv& \frac{(r-2p)(p-1)}{p} \pmod{p} \\
%\frac{p-1}{p}\sum_{\substack{
%1<j<r-1 \\
%j \equiv 1 \pmod{p-1}
%}}
%j\biggl(\binom{r-1}{j}''- \alpha_j'\biggl) &\equiv& 0 \pmod{p}.
%\end{eqnarray*}  
  %We note that $X^{r-j}Y^j-X^{r-p}Y^p \equiv j \theta X^{r-2p}Y^{p-1} \pmod{V_r^{(2)}}$. 
  We claim that 
    \begin{eqnarray*}
     {(p-1)}\sum_{\substack{
1<j<r-1 \\
j \equiv 1 \pmod{p-1}
}}
j\binom{r}{j} &\equiv& {r(p-1)} \pmod{p^2} .
\end{eqnarray*}
Indeed, we have \begin{eqnarray*}
    (p-1) \sum_{\substack{
    j \geq 0 \\
    j \equiv 1 \pmod{p-1}
}}
j\binom{r}{j}    &\equiv& r\sum_{\xi \in \mu'}  (1+\xi)^{2p-1} \pmod{p^2} \\
&\equiv& (p-1)r\biggl(1+\binom{2p-1}{p-1}+2p-1\biggl)\pmod{p^2} \\
&\equiv& (p-1)r\pmod{p^2}, \text{ so}\\
(p-1) \sum_{\substack{
    p\le j <r-1 \\
    j \equiv 1 \pmod{p-1}
}}
j\binom{r}{j} &\equiv&  (p-1)r \pmod{p^2},
\end{eqnarray*} 
proving the claim. It follows that
\begin{eqnarray*}
     \frac{p-1}{p}\sum_{\substack{
1<j<r-1 \\
j \equiv 1 \pmod{p-1}
}}
j\biggl(\binom{r}{j}''- \alpha_j\biggl) &\equiv& \frac{(r-2p)(p-1)}{p} \pmod{p}.
\end{eqnarray*}
Similarly, by Lemma~\ref{lem:2.1}, we have
\begin{eqnarray*}
\frac{p-1}{p}\sum_{\substack{
1<j<r-1 
j \equiv 1 \pmod{p-1}
}}
j\biggl(\binom{r-1}{j}''- \alpha_j\biggl) &\equiv& 0 \pmod{p}.
\end{eqnarray*} So, by 
\cite[Lemma 2.12]{GV22} applied twice, we obtain
\begin{eqnarray*}
 (T-a_p)f&=& \biggl[g^0_{1,0},\frac{r-2p}{p} \theta X^{r-2p}Y^{p-1}+\frac{(r-1)(p-1)}{p} 
\sum_{\substack{
1<j<r-1 \\
j \equiv 1 \pmod{p-1}
}}
\biggl(\binom{r}{j}''- \alpha_j\biggl) X^{r-p} Y^{p}\\
&&\qquad+\frac{2(p-1)}{p} 
\sum_{\substack{
1< j<r-1 \\
j \equiv 1 \pmod{p-1}
}}
 \biggl(\binom{r-1}{j}''-\alpha_j'\biggl) X^{r-p} Y^{p}\biggl]\\
 &&\qquad-\frac{2a_p}{p^2}\biggl[\mathrm{Id}, \theta X^{r-2p}Y^{p-1}-\theta X^{r-p-1}\biggl]- \frac{2a_p}{p^2}\sum_{\lambda\in \mathbb{F}_p} \biggl[g^0_{2,p\left[\lambda\right]},\theta X^{r-2p}Y^{p-1}-\theta Y^{r-p-1}\biggl].
 \end{eqnarray*}
Subracting $X^{r-1}Y$ from $X^{r-p}Y^p$ twice, by \cite[Lemma 2.12]{GV22} we obtain $(T-a_p)f$ is integral and maps to $-\frac{2a_p}{p^2}(T-\bar{u}^{-1})\left[g^0_{1,0},1\right]$ in $\mathrm{ind}_{KZ}^G J_3$.
\end{proof}
\begin{prop}\label{lem:G_5}
Let $3p+2 \le r \equiv 2 \pmod{p-1}$, $v_p(a_p)=2$ and $p \mid r-1$ but $p^2 \nmid r-p-1$, so that $u' \coloneqq \frac{p}{p+1-r} \cdot \frac{a_p}{p^2}$ is a $p$-adic unit. Then \begin{enumerate}
    \item[(i)] $F_2$ is a quotient of $\pi(p-1, \bar{u'}^{-1},\omega)$. 
    \item [(ii)] $F_4$ is a quotient of $\pi(p-3, \bar{u'},\omega^2)$.
\end{enumerate}
Moreover, if $p^2 \mid r-p-1$ then $F_4=0$.
\end{prop}
\begin{proof}
     (i) Let us consider $f_0,f_1\in \mathrm{ind}_{KZ}^G \mathrm{Sym}^r \bar{\mathbb{Q}}_p^2$ defined by\begin{eqnarray*}
  f_1 &=& \sum_{\lambda\in \mathbb{F}_p} \biggl[g^0_{1,\left[\lambda\right]}, \frac{XY^{r-1}-X^{r-p}Y^p}{p^2}\biggl]\\
         f_0 &=& \biggl[\mathrm{Id}, \frac{p-1}{p \cdot a_p} \sum_{\substack{
1<j<r-1 \\
j \equiv 1 \pmod{p-1}
}}
 \alpha_j \cdot X^{r-j} Y^{j}\biggl],
         \end{eqnarray*}
         where $\alpha_j$ are integers as in Lemma~\ref{lem:2.9}. 
         Since $p \mid r-1$, we have
         $\alpha_j \equiv {r-1 \choose j} \mod{p}$ for al1 $1 < j  < r-1$. Now it is clear that $T^+f_1$, $T^+f_0$ and $T^-f_0$ die mod $p$. We compute
         \begin{eqnarray*}
             T^-f_1 &=& \biggl[\mathrm{Id}, \frac{p-1}{p}\sum_{\substack{
1<j<r-1 \\
j \equiv 1 \pmod{p-1}
}}
 \binom{r-1}{j} X^{r-j} Y^{j}\biggl] \pmod{X_{r-1}}, \text{ so}\\
T^-f_1-a_pf_0 &=& \biggl[\mathrm{Id}, \frac{p-1}{p}\sum_{\substack{
1<j<r-1 \\
j \equiv 1 \pmod{p-1}
}}
 \biggl(\binom{r-1}{j}-\alpha_j\biggl) X^{r-j} Y^{j}\biggl].
 \end{eqnarray*}
 Now as before $X^{r-j}Y^j-X^{r-p}Y^p \equiv j \theta X^{r-2p}Y^{p-1} \pmod{V_r^{(2)}}$. So by Lemma~\ref{lem:2.1} applied twice, we obtain \begin{eqnarray*}
 T^-f_1-a_pf_0&=& \biggl[\mathrm{Id}, \frac{p-1}{p}\sum_{\substack{
1<j<r-1 \\
j \equiv 1 \pmod{p-1}
}}
 \biggl(\binom{r-1}{j}-\alpha_j\biggl) (X^{r-j} Y^{j}-X^{r-p}Y^p)\\
 && \qquad +\frac{p-1}{p}\sum_{\substack{
1<j<r-1 \\
j \equiv 1 \pmod{p-1}
}}
 \biggl(\binom{r-1}{j}-\alpha_j\biggl)X^{r-p}Y^p\biggl]\\
 &=& \biggl[\mathrm{Id}, \frac{(p-1)(p+1-r)}{p} X^{r-p} Y^{p}\biggl].
 \end{eqnarray*}
 Hence, we have \begin{eqnarray*}
     (T-a_p)(f_0+f_1) &=& -\frac{a_p}{p^2}\sum_{\lambda\in \mathbb{F}_p} \biggl[g^0_{1,\left[\lambda\right]}, \theta X^{r-p-1}\biggl]+\frac{(p+1-r)}{p}\biggl[\mathrm{Id}, \theta X^{r-p-1}\biggl] \pmod{X_{r-1}}.
  \end{eqnarray*}
  Hence, $(T-a_p)(f_0+f_1)$ is integral and maps to $-\overline{\frac{a_p}{p^2}} (T-\bar{u'}^{-1})\biggl[\mathrm{Id},X^{p-1}\biggl]$ in \(\mathrm{ind}_{KZ}^G J_2\).

    (ii) Let us consider $f_0,f_1,f_2\in \mathrm{ind}_{KZ}^G \mathrm{Sym}^r \bar{\mathbb{Q}}_p^2$ defined by \begin{eqnarray*}
        f_2 &=& \sum_{{\lambda \in \mathbb{F}_p},\,\mu \in {\mathbb{F}_p^\times}} \biggl[g^0_{2,\left [\lambda \right]+p\left[\mu \right]}, \frac{Y^r-X^{r-p-1}Y^{p+1}}{p}\biggl]+(1-p)\sum_{\lambda \in \mathbb{F}_p}\biggl[g^0_{2,\left [\lambda \right]}, \frac{Y^r-X^{r-p-1}Y^{p+1}}{p}\biggl]\\
         f_1 &=& \sum_{\lambda \in {\mathbb{F}_p}} \biggl[g^0_{1,\left[\lambda \right]},\frac{p-1}{p \cdot a_p} \sum_{\substack{
         2<j<r \\
         j \equiv 2 \pmod{p-1}
         }}
         \alpha_j \cdot X^{r-j} Y^{j}\biggl]\\
        f_0 &=& \biggl[\mathrm{Id},\frac{\theta^2 X^{r-2p-2}}{2p^2}\biggl],
    \end{eqnarray*}
    where $\alpha_j$ are integers as in Lemma~\ref{lem:2.7}.
   Clearly $T^+f_2$, $-a_pf_2$, $T^+f_1$, $T^-f_1$ and $T^-f_0$ die mod $p$. We have 
\begin{eqnarray*}
T^-f_2&=& \sum_{\lambda \in {\mathbb{F}_p}} \biggl[g^0_{1,\left[\lambda \right]},\frac{p-1}{p} \sum_{\substack{
0<j<r \\
j \equiv 2 \pmod{p-1}
}}
 \binom{r}{j} X^{r-j} Y^{j}\biggl], \text{ so}\\
 T^-f_2-a_pf_1 &=& \sum_{\lambda \in {\mathbb{F}_p}} \biggl[g^0_{1,\left[\lambda \right]},\frac{p-1}{p} \binom{r}{2}X^{r-2}Y^2 +\frac{p-1}{p} \sum_{\substack{
2<j<r \\
j \equiv 2 \pmod{p-1}
}}
 \biggl(\binom{r}{j}-\alpha_j\biggl) X^{r-j} Y^{j}\biggl].
\end{eqnarray*}
As in the proof of 
Propostion~\ref{prop: alpha 1} part ($ii$), we obtain
\begin{eqnarray*}
  T^-f_2-a_pf_1  &=& \sum_{\lambda \in {\mathbb{F}_p}} \biggl[g^0_{1,\left[\lambda \right]},\frac{p-1}{p} \binom{r}{2}X^{r-2}Y^2 + \frac{p-1}{p} \sum_{\substack{
2<j<r \\
j \equiv 2 \pmod{p-1}
}}
 \biggl(\binom{r}{j}-\alpha_j\biggl) X^{r-j} Y^{j}\\
 && \qquad -(p-1)\biggl(\frac{1}{p}\binom{r}{2}-\frac{1}{2}\biggl)X^{r-2p}Y^{2p}-(p-1)\biggl(1-\frac{2}{p}\binom{r}{2}\biggl)X^{r-p-1}Y^{p+1}\\
 && \qquad +(p-1)\biggl(\frac{1}{p}\binom{r}{2}-\frac{1}{2}\biggl)X^{r-2p}Y^{2p}+(p-1)\biggl(1-\frac{2}{p}\binom{r}{2}\biggl)X^{r-p-1}Y^{p+1}\biggl]\\
 &=& \sum_{\lambda \in {\mathbb{F}_p}} \biggl[g^0_{1,\left[\lambda \right]},\frac{p-1}{p}\binom{r}{2}\theta^2 X^{r-2p-2}-\frac{p-1}{2}X^{r-2p}Y^{2p}+(p-1)X^{r-p-1}Y^{p+1} \biggl],
\end{eqnarray*}
by Lemma~\ref{lem: im 3}.
We have \[
T^+f_0 = \frac{1}{2}\sum_{\lambda \in {\mathbb{F}_p}} \biggl[g^0_{1,\left[\lambda \right]}, X^{r-2}Y^2\biggl].
\]
Hence, we get \begin{eqnarray*}
    T^-f_2-a_pf_1+T^+f_0 &=& \sum_{\lambda \in {\mathbb{F}_p}} \biggl[g^0_{1,\left[\lambda \right]},\frac{p-1}{p}\binom{r}{2} \theta^2 X^{r-2p-2}-\frac{p-1}{2}\theta^2X^{r-2p-2} \biggl]\\
    &=& \sum_{\lambda \in {\mathbb{F}_p}} \biggl[g^0_{1,\left[\lambda \right]},\frac{p+1-r}{2p}\theta^2X^{r-2p-2} \biggl].
\end{eqnarray*}
Combining these we have \begin{eqnarray*}
    (T-a_p)(f_0+f_1+f_2)=\frac{p+1-r}{2p}\sum_{\lambda \in {\mathbb{F}_p}} \biggl[g^0_{1,\left[\lambda \right]},\theta^2X^{r-2p-2} \biggl]-\frac{a_p}{2p^2}\biggl[\mathrm{Id},{\theta^2 X^{r-2p-2}}\biggl].
\end{eqnarray*}
If $p^2 \nmid r-p-1$, then $(T-a_p)(f_0+f_1+f_2)$ is integral and maps to $\overline{\frac{p+1-r}{2p}}(T-\bar{u'})\left[\mathrm{Id}, X^{p-3}\right]$ in \(\mathrm{ind}_{KZ}^G J_4\).
 Moreover, if $p^2 \mid r-p-1$, then $(T-a_p)(f_0+f_1+f_2)$ maps to $-\overline{\frac{a_p}{2p^2}}\biggl[\mathrm{Id},X^{p-3}\biggl]$ in \(\mathrm{ind}_{KZ}^G J_4\).
\end{proof}
\begin{prop}\label{lem:G_6}
    Let $p \ge 3$ and $2p+1 \le r \equiv 2 \pmod{p-1}$. Then \begin{itemize}
        \item [(i)] if $p \mid r$, then $F_2=0$
        \item[(ii)] if $p \mid r-1$, then $F_3=0$. 
    \end{itemize}
    \end{prop}
    \begin{proof}
       (i) By \cite[Proposition $4.22$]{GV22}, we have $\frac{X_{r-1}^{*}}{X_{r-1}^{**}} \cong V_{p-1} \otimes D$. Hence, $F_2=0$.

       (ii) Again by  \cite[Proposition $4.22$]{GV22}, we have $\frac{X_{r-1}^{*}}{X_{r-1}^{**}} \cong V_{0} \otimes D$. Hence, $F_3=0$.
    \end{proof}
\begin{proof}[Proof of Theorem~\ref{thm:a=2}] 
The above propositions may be summarized by the following table:

\begin{table}[h]
\centering
\begin{tabular}{|>{\centering\arraybackslash}m{5cm}|
                >{\centering\arraybackslash}m{1.3cm}|
                >{\centering\arraybackslash}m{1.3cm}|
                >{\centering\arraybackslash}m{1.3cm}|
                >{\centering\arraybackslash}m{1.3cm}|
                >{\centering\arraybackslash}m{1.3cm}|}
                  
\hline
& $F_5$ & $F_4$ & $F_3$ & $F_2$ & $F_1$ \\ \hline
$p\nmid r(r-1)$
& $\checkmark$& &  & & \\ \hline

$p\mid r$,\; $p^2 \nmid r-2p$
& & $\checkmark$&$\checkmark$ & & \\ \hline

$p^2\mid r-2p$
& & &$\checkmark$ & &\\ \hline

$p\mid r-1$,\; $p^2 \nmid r-p-1$
& &$\checkmark$  & &$\checkmark$  & \\ \hline

$p^2 \mid r-p-1$
& & & &  $\checkmark$& \\ \hline
\end{tabular}
\end{table}

In particular, we have:
\begin{itemize}
    \item [(1)] If $p \nmid r(r-1)$, then by Proposition~\ref{lem:G_1} and Proposition~\ref{lem:G_2} we get $\overline{\Theta}^{\,{ss}} \cong F_5 \cong \pi(2,0,1)$. 
    \item [(2)] If $p \mid r$ but $p^2 \nmid r-2p$, then by Proposition~\ref{lem:G_1}, Proposition~\ref{prop: alpha 1} and Proposition~\ref{lem:G_6}, we get $\overline{\Theta}^{\,{ss}} \cong (F_3 \oplus F_4)^{\,{ss}}$. Moreover, Proposition~\ref{lem:G_4} tells us that we have surjections $\pi(p-3,\bar{u},\omega^2) \twoheadrightarrow F_4$ and $\pi(0,\bar{u}^{-1},\omega) \twoheadrightarrow F_3$. Hence, we have $\overline{\Theta}^{\,{ss}} \cong \pi(p-3,\bar{u},\omega^2)^{\,{ss}} \oplus \pi(0,\bar{u}^{-1},\omega)^{\,{ss}}$.
    \item[(3)]  If $p^2 \mid r-2p$, then by Proposition~\ref{lem:G_1}, Proposition~\ref{prop: alpha 1}, Proposition~\ref{lem:G_4} and Proposition~\ref{lem:G_6}, we have $\overline{\Theta}^{\,{ss}} \cong F_3 \cong \pi(0,0,\omega)$.
    \item [(4)] If $p \mid r-1$ but $p^2 \nmid r-p-1$, then by    Proposition~\ref{lem:G_1}, Proposition~\ref{prop: alpha 1} and Proposition~\ref{lem:G_6}, we get  $\overline{\Theta}^{\,{ss}} \cong (F_2 \oplus F_4)^{\,{ss}}$. Moreover, by Proposition~\ref{lem:G_5}, we have surjections $\pi(p-3,\bar{u'},\omega^2) \twoheadrightarrow F_4$ and $\pi(p-1,\bar{u'}^{-1},\omega) \twoheadrightarrow F_2$. Hence, we have $\overline{\Theta}^{\,{ss}} \cong \pi(p-3,\bar{u'},\omega^2)^{\,{ss}} \oplus \pi(p-1,\bar{u'}^{-1},\omega)^{\,{ss}}$.
    \item [(5)] If $p^2 \mid r-p-1$, then by Proposition~\ref{lem:G_1}, Proposition~\ref{prop: alpha 1}, Proposition~\ref{lem:G_5} and Proposition~\ref{lem:G_6}, we have $\overline{\Theta}^{\,{ss}} \cong F_2 \cong \pi(p-1,0,\omega)$.
\end{itemize}
Theorem~\ref{thm:a=2} now follows in all cases from the mod $p$ LLC.
\end{proof}
\section{The case $a=3$}\label{sec: a=3}
 In this section we will discuss the case $3p+2 \le r \equiv 3 \pmod{p-1}$.
 \begin{theorem}\label{thm:a=3}
     Let $3p+2 \le r \equiv 3 \pmod{p-1}$ and $v_p(a_p)=2$. Then \begin{enumerate}
     \item[(1)] If $p \nmid \binom{r}{4}$, then 
         $\overline{V}^{\,{ss}}_{k,a_p} \cong \mu_{\lambda}\cdot \omega\oplus \mu_{\lambda^{-1}} \cdot \omega^3$, where $$\lambda =\overline{\frac{r-3}{3\binom{r}{3}}\cdot\frac{a_p}{p^2}} \in \overline{\mathbb{F}}_p^\times.$$
         \item [(2)] If $p \mid r$, then $\overline{V}^{\,{ss}}_{k,a_p}\cong \mathrm{ind}\biggl(\omega_2^{p+3}\biggl)$.
          \item[(3)] If $p \mid r-1$, then $\overline{V}^{\,{ss}}_{k,a_p} \cong \mu_{\lambda}\cdot \omega^2 \oplus \mu_{\lambda^{-1}} \cdot \omega^2$, where $$ \lambda+\lambda^{-1}=\overline{\frac{a_p^2-p^3(\frac{r-2p-1}{2})}{p^2a_p}} \in \overline{\mathbb{F}}_p.$$
           \item[(4)] If $p \mid r-2$, then 
         $\overline{V}^{\,{ss}}_{k,a_p}\cong \mathrm{ind} \biggl(\omega_2^{3p+1}\biggl)$.
         \item[(5)] If $p \mid r-3$, then $\overline{V}^{\,{ss}}_{k,a_p}\cong \mathrm{ind} \biggl(\omega_2^{4p}\biggl)$.

     \end{enumerate}
 \end{theorem}
\begin{lemma}\label{lem: 5.3 for r-1}
Let $3p+2 \le r\equiv 3 \pmod{p-1}$ and $p \mid r-1$, then 
    \[\frac{1}{p}\sum_{\substack{
2 <j< r-2 \\
j \equiv 2 \pmod{p-1}
}}
 \biggl(\binom{r-1}{j}-\alpha_j\biggl) X^{r-j} Y^{j}-{A}X^{r-2p}Y^{2p}-{B}X^{r-p-1}Y^{p+1} \in V_r^{(3)},
 \]
 where $A= 1+ \frac{1}{p} \binom{r-1}{2}$, $B= \frac{-2}{p} \binom{r-1}{2}\ $ and $\alpha_j$ are integers as in Lemma~\ref{lem:2.7} (applied with $r-1$ instead of $r$).
  \end{lemma}
\begin{proof}
By Lemma~\ref{lem:2.1} with $r-1$ instead of $r$ and $(i,a)=(0,2)$, $(1,2)$ and $(2,p+1)$, respectively, we have

%\begin{comment} 
%   We know that \[
%  (p-1)\sum_{\substack{
%j \geq 0 \\
%j \equiv 2 \pmod{p-1}
%}}
% \binom{j}{m}\binom{r}{j} = \binom{r}{m} \sum_{\xi \in \mu'} \xi^{m-2}(1+\xi)^{r-m}.
%    \]
%    So, we have \begin{eqnarray*}
%       (p-1) \sum_{\substack{
%j \geq 0 \\
%j \equiv 2 \pmod{p-1}
%}}
% \binom{j}{m}\binom{r}{j} \equiv %\binom{r}{m} \sum_{\xi \in \mu'} %\xi^{m-2}.(1+\xi)^{2p-m} \pmod{p^2}.
%    \end{eqnarray*}
%\end{comment}

\begin{eqnarray*}
%       \item [(i)] For $m=0$ we have\begin{eqnarray*}
%       \begin{comment}
%           (p-1) \sum_{\substack{
%j \geq 0 \\
%j \equiv 2 \pmod{p-1}
%}}
% \binom{r}{j} &\equiv&  \sum_{\xi \in \mu'} \xi^{-2}(1+\xi)^{2p} \pmod{p^2}\\
% &\equiv& (p-1)\biggl(\binom{2p}{2}+\binom{2p}{p+1}+1\biggl) \pmod{p^2}\\
%  (p-1) \sum_{\substack{
%2<j<r-2  \\
%j \equiv 2 \pmod{p-1}
%}}
% \binom{r}{j} &\equiv& (p-%1)\biggl(\binom{2p}{2}+\binom{2p}{p+1}-\binom{r}{2}\biggl) %\pmod{p^2}\\
%  \end{comment}
 \frac{1}{p} \sum_{\substack{
2<j<r-2  \\
j \equiv 2 \pmod{p-1}
}}
 \binom{r-1}{j} &\equiv&  1-\frac{1}{p}\binom{r-1}{2}\pmod{p} \\
 %
 %\item[(ii)] For $m=1$ we have 
 %\begin{eqnarray*}
  %      
 %       \begin{comment}
  %          (p-1) \sum_{\substack{
%j \geq 0 \\
%j \equiv 2 \pmod{p-1}
%}}
%j \binom{r}{j} &\equiv&  r\sum_{\xi \in \mu'} \xi^{-1}(1+\xi)^{2p-1} \pmod{p^2}\\
%&\equiv& r(p-1)\biggl( 2p-1 + %\binom{2p-1}{p}+1\biggl) %\pmod{p^2}\\
% (p-1) \sum_{\substack{
%2<j<r-2 \\
%j \equiv 2 \pmod{p-1}
%}}
%j \binom{r}{j} &\equiv&  r(p-%1)\biggl( 2p-1 + \binom{2p-1}%{p}\biggl)-2(p-1)\binom{r}{2 } %\pmod{p^2}\\
%\end{comment}
\frac{1}{p} \sum_{\substack{
2<j<r-2 \\
j \equiv 2 \pmod{p-1}
}}
j \binom{r-1}{j} &\equiv& -\frac{2}{p} \binom{r-1}{2} \pmod{p} \\
% \item[(iii)] For $m=2$ we have 
 \frac{1}{p} \sum_{\substack{
2<j<r-2 \\
j \equiv 2 \pmod{p-1}
}}
\binom{j}{2} \binom{r-1}{j} &\equiv& 0 \pmod{p}.
 \end{eqnarray*}
 Consider \[ 
  H(X,Y)=\frac{1}{p}\sum_{\substack{
2 <j< r-2 \\
j \equiv 2 \pmod{p-1}
}}
\biggl( \binom{r-1}{j}-\alpha_j \biggl) X^{r-j} Y^{j}-AX^{r-2p}Y^{2p}-BX^{r-p-1}Y^{p+1}.\]
We determine constants \(A,B\) such that $H(X,Y)\in V_r^{(3)}$. Note that $X^3$ and $Y^3$ divide $H(X,Y)$. By the congruences above, we must solve:
 \begin{eqnarray*}
     1-\frac{1}{p}\binom{r-1}{2}-A-B &=& 0\\
     -\frac{2}{p}\binom{r-1}{2}-B &=& 0.
\end{eqnarray*}
Hence, we have \begin{eqnarray*}
    A &=& 1+\frac{1}{p}\binom{r-1}{2}\\
    B &=& -\frac{2}{p}\binom{r-1}{2},
\end{eqnarray*}
proving the lemma.
\end{proof}
\begin{lemma}
    Let $3p+2 \le r \equiv 3 \pmod{p-1}$. Then \[
    \binom{r}{2}X^{r-2}Y^2
+
AX^{r-2p}Y^{2p}
+
BX^{r-p-1}Y^{p+1}
+
CX^{r-3p+1}Y^{3p-1}
\in X_{r-1}+V_r^{(3)},
    \]
    where $C=-r\binom{r-1}{2}$, $B-C=2r-2r(r-1)$, $A+B+C=3-r-\binom{r}{2}$.
\end{lemma}
\begin{proof}
    We know that 
\begin{eqnarray}
\label{useful for last one}
    \sum_{\substack{
0\le j\le r  \nonumber \\
j \equiv 2 \pmod{p-1}
}} \binom{r}{j} X^{r-j}Y^j \in X_{r} \subset X_{r-1}, \text{ so}\\ 
\binom{r}{2} X^{r-2} Y^{2} + \sum_{\substack{
2<j<r-1 \\
j \equiv 2 \pmod{p-1}
}} 
\binom{r}{j}X^{r-j} Y^{j} \in X_{r-1}.
\end{eqnarray}
 By \cite[Lemma 2.14]{GV22}, we have 
\[
\sum_{\substack{
0\le j\le r\\
j\equiv 2 \!\!\!\pmod{p-1}
}}
\binom{r}{j}\binom{j}{m}
\equiv
\binom{r}{m}
\binom{[3-m]}{[2-m]}
+
\binom{r}{m}\delta_{p-1,[2-m]}
\pmod p.
\]
 Hence, we have  
\begin{align}
\label{important GM 1}\sum_{\substack{
2<j<r-1\\
j\equiv 2 \!\!\!\pmod{p-1}
}}
\binom{r}{j}
&\equiv
3-r-\binom{r}{2}
\pmod p,\\
\label{important GM 2}
\sum_{\substack{
2<j<r-1\\
j\equiv 2 \!\!\!\pmod{p-1}
}}
j\binom{r}{j}
&\equiv
2r-2r(r-1)
\pmod p,\\
\label{important GM 3}
\sum_{\substack{
2<j<r-1\\
j\equiv 2 \!\!\!\pmod{p-1}
}}
\binom{j}{2}\binom{r}{j}
&\equiv
-r\binom{r-1}{2}
\pmod p.
\end{align}
 Let us consider
\[
I(X,Y) \coloneqq \sum_{\substack{
2<j<r-1\\
j\equiv 2 \!\!\!\pmod{p-1}
}}
\binom{r}{j}X^{r-j}Y^j
-
AX^{r-2p}Y^{2p}
-
BX^{r-p-1}Y^{p+1}
-
CX^{r-3p+1}Y^{3p-1}.
\]
We determine constants \(A,B,C\) such that $I(X,Y)\in V_r^{(3)}$. Note that $X^3$ and $Y^3$ divide $I(X,Y)$.
By the congruences above, we must solve:
\begin{align*}
A+B+C &= 3-r-\binom{r}{2},\\
B-C &= 2r-2r(r-1),\\
C &= -r\binom{r-1}{2},
\end{align*}
which is clearly possible. 
%This proves the lemma.
Hence, we have \[\binom{r}{2}X^{r-2}Y^2
+
AX^{r-2p}Y^{2p}
+
BX^{r-p-1}Y^{p+1}
+
CX^{r-3p+1}Y^{3p-1}
\in X_{r-1}+V_r^{(3)},\]
proving the lemma.
\end{proof}
\begin{cor}\label{lem: required for p | r-2}
    Let $3p+2 \le r \equiv 3 \pmod{p-1}$ and $p \mid r-2$. Then $X^{r-2}Y^2\in X_{r-1}+V_r^{(3)}$.
    
\end{cor}
\begin{proof}
    As $p \mid r-2$, we have \[\binom{r}{2}X^{r-2}Y^2\in X_{r-1}+V_r^{(3)}.\] 
Hence, we have $X^{r-2}Y^2\in X_{r-1}+V_r^{(3)}$.
\end{proof}
\begin{cor}\label{lem: X^{r-p-1Y^{p+1}}}
    Let $3p+2 \le r \equiv 3 \pmod{p-1}$ and $p \mid r-1$. Then $X^{r-p-1}Y^{p+1} \in X_{r-1}+V_r^{(3)}$.
\end{cor}
\begin{proof}
       As $p \mid r-1 $, we have \begin{eqnarray*}
           \binom{r}{2}X^{r-2}Y^2+2X^{r-p-1}Y^{p+1} \in X_{r-1}+V_r^{(3)}.
\end{eqnarray*}
Hence, we have $X^{r-p-1}Y^{p+1} \in X_{r-1}+V_r^{(3)}$.
\end{proof}
\begin{prop}\label{Gh Ma 1}
    Let $3p+2 \le r \equiv 3 \pmod{p-1}$. Then \begin{itemize}
        \item [(i)] If $p \nmid (r-1)(r-2)(r-3)$, then $F_4=0$.
        \item[(ii)] If $p \mid r-3$, then $F_4=0$, $F_5=0$. 
    \end{itemize}
\end{prop}
\begin{proof}
  (i)  By \cite[Proposition $4.18$]{GV22}, we have $\frac{X_{r-1}^{**}}{X_{r-1}^{***}} \cong V_{p-2} \otimes D^2$, noting that $\mathcal{I}(3,2)=\{1,2,3\}$. Hence, $F_4=0$.

  (ii) By \cite[Proposition $4.18$]{GV22}, we have $\frac{X_{r-1}^{**}}{X_{r-1}^{***}} \cong \frac{V_r^{**}}{V_r^{***}}$. Hence, $F_4=0$ and $F_5=0$.
\end{proof} 
 \begin{prop}\label{Gh Ma 2}
    Let $3p+2 \le r \equiv 3 \pmod{p-1}$. Assume that $v_p(a_p)=2$ and $p \nmid r-2 $. Then $F_2=0=F_3$.
 \end{prop}
\begin{proof}
    Let us consider $f_0,f_2\in \mathrm{ind}_{KZ}^G \mathrm{Sym}^r \bar{\mathbb{Q}}_p^2$ defined by \begin{eqnarray*}
    f_2 &=& \biggl[g^0_{2,0}, \frac{\theta X^{r-2p-1}Y^p-\theta Y^{r-p-1}}{p}\biggl]\\
        f_0 &=& \biggl[\mathrm{Id}, \frac{\theta^2 X^{r-3p-1}Y^{p-1}-\theta^2 X^{r-2p-2}}{p^2}\biggl].\\
        \end{eqnarray*}
    Clearly $T^+f_2 $, $-a_pf_2$, $T^-f_0$ die mod $p$ and $-a_pf_0$ dies mod $V_r^{(2)}$. Moreover, we have \[
    T^-f_2 = \biggl[g^0_{1,0},XY^{r-1}\biggl].
    \]
    Finally, we have \[
    T^+f_0=-\biggl[g^0_{1,0}, X^{r-2}Y^2\biggl].
    \]
    Combining all these together we have \begin{eqnarray*}
        (T-a_p)(f_0+f_2) &\equiv& -\biggl[g^0_{1,0}, X^{r-2}Y^2-XY^{r-1}\biggl] \\
        &\equiv & \biggl[g^0_{1,0},(r-2)\theta X^{r-p-2}Y-\theta Y^{r-p-1}\biggl]\pmod{V_r^{(2)}},
    \end{eqnarray*}
    by \cite[Lemma 2.12]{GV22}. Note that $(r-2)\theta X^{r-p-2}Y-\theta Y^{r-p-1}$ in $V_r^{*}/V_r^{**}$.
    The short exact sequence  
\[
0 \longrightarrow V_{1}\otimes D
\longrightarrow \frac{V_r^{*}}{V_r^{**}} 
\longrightarrow J_3 = V_{p-2}\otimes D^2
\longrightarrow 0
\]
is non-split. The reduction $\overline {(T-a_p)(f_0+f_2)}$  maps to $(r-2)\left[g^0_{1,0}, X^{p-2} \right] \in \mathrm{ind}_{KZ}^G 
    J_3$ which since $p \nmid r-2$ is nonzero and generates it as a $G$-module.
    %(V_r^{*}/V_r^{**})$. Now $v$ maps to $X^{p-2}\in J_3$. 
    %$v$ generates the whole module $V_r^{*}/V_r^{**}$ as a $G$-module, so the element $\left[g^0_{1,0},v\right]$ generates $\mathrm{ind}_{KZ}^G (V_r^{*}/V_r^{**})$. 
    Thus $F_2=0$ and $F_3=0$.
\end{proof}
\begin{prop}\label{Gh Ma 3}
  Let $3p+2 \le r \equiv 3 \pmod{p-1}$. Assume that $v_p(a_p)=2$ and $p \mid r-1$. Then $F_5=0$.
\end{prop}
\begin{proof}
    Let us consider $f_0,f_1,f'_1,f_2\in \mathrm{ind}_{KZ}^G \mathrm{Sym}^r \bar{\mathbb{Q}}_p^2$ defined by 
\begin{eqnarray*}
f_2 &=& \sum_{\lambda \in \mathbb{F}_p}\biggl[g^0_{2,p\left[\lambda\right]},\frac{XY^{r-1}-X^{r-2p}Y^{2p}}{p^2}\biggl]\\
        f_1 &=& \biggl[g^0_{1,0}, \frac{p-1}{p\cdot a_p} \sum_{\substack{
2<j<r-1 \\
j \equiv 2 \pmod{p-1}
}}
 \alpha_j\cdot X^{r-j} Y^{j} \biggl]\\
 f_1' &=& \frac{a_p}{p^2}\biggl[g^0_{1,0}, \frac{\theta^2 X^{r-2p-2}}{p^2}\biggl]\\
 f_0 &=& \biggl[\mathrm{Id}, \frac{\theta^2X^{r-3p-1}Y^{p-1}-\theta^2X^{r-2p-2}}{p^2}\biggl],
    \end{eqnarray*}
    where $\alpha_j$ are integers as in Lemma~\ref{lem:2.7} (applied with $r-1$ instead of $r$).
    Clearly $T^+f_2$, $T^+f_1$, $T^-f_1$, $T^-f_1'$ and $T^-f_0$ die mod $p$ and mod $X_{r-1}$. We compute    \begin{eqnarray*}
        T^-f_2-a_pf_1 &=& \biggl[g^0_{1,0}, \frac{p-1}{p} \binom{r-1}{2} X^{r-2}Y^2+\frac{p-1}{p}\sum_{\substack{
2<j<r-1 \\
j \equiv 2 \pmod{p-1}
}}
 \biggl(\binom{r-1}{j}-\alpha_j\biggl) X^{r-j} Y^{j}\biggl]\\
 &=& \biggl[g^0_{1,0},\frac{p-1}{p} \binom{r-1}{2} X^{r-2}Y^2+ (p-1) \biggl(1+\frac{1}{p}\binom{r-1}{2}\biggl) X^{r-2p}Y^{2p}\\
 && \qquad -\frac{2(p-1)}{p}\binom{r-1}{2}X^{r-p-1}Y^{p+1}\biggl]\\
 &=& \biggl[g^0_{1,0}, -\frac{1}{p}\binom{r-1}{2} \theta^2 X^{r-2p-2}-X^{r-2p}Y^{2p}\biggl],
 \end{eqnarray*}
 where the second equality follows from Lemma~\ref{lem: 5.3 for r-1}.  Moreover, we have \begin{eqnarray*}
         T^+f_1' &=& \frac{a_p}{p^2}\sum_{\lambda\in \mathbb{F}_p} \biggl[g^0_{2,p\left[\lambda\right]}, X^{r-2}Y^2\biggl].
          \end{eqnarray*}
 On the other hand, we have \begin{eqnarray*}
     (T-a_p)f_0 &=& -\biggl[g^0_{1,0}, X^{r-2}Y^2\biggl]-\frac{a_p}{p^2}\biggl[\mathrm{Id}, \theta^2X^{r-3p-1}Y^{p-1}-\theta^2X^{r-2p-2}\biggl].
     \end{eqnarray*}
    Let $f=f_0+f_1+f_1'+f_2$, then we have \begin{eqnarray*}
         (T-a_p)f& \equiv & \biggl[g^0_{1,0}, -\frac{1}{p}\binom{r-1}{2}\theta^2X^{r-2p-2}-X^{r-2}Y^2-X^{r-2p}Y^{2p}\biggl]-\frac{a_p^2}{p^4}\biggl[g^0_{1,0},\theta^2 X^{r-2p-2}\biggl]\\
         && \qquad -\frac{a_p}{p^2}\biggl[\mathrm{Id},\theta^2X^{r-3p-1}Y^{p-1}-\theta^2X^{r-2p-2}\biggl]+ \frac{a_p}{p^2} \sum_{\lambda\in \mathbb{F}_p} \biggl[g^0_{2,p\left[\lambda\right]}, X^{r-2}Y^2+X^{r-2p}Y^{2p}\biggl]\\
         &\equiv &\biggl[g^0_{1,0}, -\frac{1}{p}\binom{r-1}{2}\theta^2X^{r-2p-2}-\theta^2X^{r-2p-2} -\frac{a_p^2}{p^4}\theta^2 X^{r-2p-2}\biggl]\\ 
         && \qquad-\frac{a_p}{p^2}\biggl[\mathrm{Id},\theta^2X^{r-3p-1}Y^{p-1}-\theta^2X^{r-2p-2}\biggl] + \frac{a_p}{p^2} \sum_{\lambda\in \mathbb{F}_p} \biggl[g^0_{2,p\left[\lambda\right]}, \theta^2 X^{r-2p-2}\biggl] \pmod{X_{r-1}+V_r^{(3)}},
      \end{eqnarray*}
      by Corollary~\ref{lem: X^{r-p-1Y^{p+1}}}.
      By Lemma~\ref{lem:2.10},  we see that $\theta^2 X^{r-3p-1}Y^{p-1}$ maps to $-Y$ in $J_5$. Hence, 
      %$(T-a_p)f$ is integral and 
      $(\overline{T-a_p)f}$ maps to $\overline{\frac{a_p}{p^2}}\left[\mathrm{Id},Y\right]$ in \(\mathrm{ind}_{KZ}^G J_5\).
\end{proof}
\begin{prop}\label{Gh Ma 4}
Let $3p+2 \le r\equiv 3 \pmod{p-1}$. Assume that $v_p(a_p)=2$ and $p \mid r-1$. Then $F_4$ is a quotient of $\frac{\mathrm{ind}_{KZ}^G (V_{p-2} \otimes D^2)}{T^2-cT+1} $, where $c= \overline{\frac{a_p^2-p^3(\frac{r-2p-1}{2})}{p^2a_p}} \in \overline{\mathbb{F}}_p$.
\end{prop}
\begin{proof}
    Let us consider $f_0,f_1,f'_1,f_2\in \mathrm{ind}_{KZ}^G \mathrm{Sym}^r \bar{\mathbb{Q}}_p^2$ defined by \begin{eqnarray*}
        f_2 &=& \sum_{\lambda,\mu \in \mathbb{F}_p}\biggl[g^0_{2,\left[\lambda\right]+p\left[\mu\right]},\frac{XY^{r-1}-X^{r-2p}Y^{2p}}{p^2}\biggl]\\
        f_1 &=& \sum_{\lambda \in {\mathbb{F}_p}} \biggl[g^0_{1,\left[\lambda\right]}, \frac{p-1}{p \cdot a_p} \sum_{\substack{
2<j<r-1 \\
j \equiv 2 \pmod{p-1}
}}
 \alpha_j \cdot X^{r-j} Y^{j} \biggl]\\
 f_1' &=& \frac{a_p}{p^2}\sum_{\lambda \in {\mathbb{F}_p}} \biggl[g^0_{1,\left[\lambda\right]}, \frac{\theta^2 X^{r-2p-2}}{p^2}\biggl]\\
 f_0 &=& -\biggl[\mathrm{Id}, \frac{\theta^2X^{r-2p-2}}{p^2}\biggl],
 \end{eqnarray*}
  where $\alpha_j$ are integers as in Lemma~\ref{lem:2.7} (applied with $r-1$ instead of $r$). Clearly $T^+f_2$, $T^+f_1$, $T^-f_1$, $T^-f_1'$ and $T^-f_0$ die mod $p$. As before, 
 \begin{eqnarray*}
     %T^-f_2 &=& \sum_{\lambda \in {\mathbb{F}_p}} \biggl[g^0_{1,\left[\lambda\right]}, \frac{p-1}{p} \sum_{\substack{
%2\le j<r-1 \\
%j \equiv 2 \pmod{p-1}
%}}
 %\binom{r-1}{2} X^{r-j} Y^{j} \biggl]\\
 T^-f_2-a_pf_1 &=& \sum_{\lambda \in {\mathbb{F}_p}} \biggl[g^0_{1,\left[\lambda\right]}, \frac{p-1}{p}\binom{r-1}{2}X^{r-2}Y^2+\frac{p-1}{p} \sum_{\substack{
2<j<r-1 \\
j \equiv 2 \pmod{p-1}
}}
\biggl(\binom{r-1}{2}- \alpha_j\biggl) X^{r-j} Y^{j} \biggl]\\
&=& \sum_{\lambda \in {\mathbb{F}_p}} \biggl[g^0_{1,\left[\lambda\right]}, -\frac{1}{p}\binom{r-1}{2} \theta^2X^{r-2p-2}-X^{r-2p}Y^{2p}\biggl].
 \end{eqnarray*}
 Moreover, we have \begin{eqnarray*}
     T^+f_1'= \frac{a_p}{p^2}\sum_{\lambda,\mu \in \mathbb{F}_p}\biggl[g^0_{2,\left[\lambda\right]+p\left[\mu\right]}, X^{r-2}Y^2 \biggl].
 \end{eqnarray*}
 On the other hand, we have \begin{eqnarray*}
 T^+f_0-a_pf_0&=&-\sum_{\lambda \in {\mathbb{F}_p}} \biggl[g^0_{1,\left[\lambda\right]}, X^{r-2}Y^2 \biggl]+\frac{a_p}{p^2}\biggl[\mathrm{Id},\theta^2 X^{r-2p-2}\biggl].
 \end{eqnarray*}
 Hence, by Corollary~\ref{lem: X^{r-p-1Y^{p+1}}} we have, \begin{eqnarray*}
     T^-f_2-a_pf_1+T^+f_0-a_pf_0 &\equiv& \sum_{\lambda \in {\mathbb{F}_p}} \biggl[g^0_{1,\left[\lambda\right]},-\frac{1}{p}\binom{r-1}{2}\theta^2 X^{r-2p-2}-\theta^2 X^{r-2p-2}\biggl]\\
     && \quad \qquad +\frac{a_p}{p^2}\biggl[\mathrm{Id},\theta^2 X^{r-2p-2}\biggl] \pmod{X_{r-1}+V_r^{(3)}},
     \end{eqnarray*}
     Finally, let $f=f_0+f_1+f_1'+f_2$, then again by Corollary~\ref{lem: X^{r-p-1Y^{p+1}}}, we have \begin{eqnarray*}
     (T-a_p)f &=&  \frac{a_p}{p^2}\sum_{\lambda,\mu \in \mathbb{F}_p}\biggl[g^0_{2,\left[\lambda\right]+p\left[\lambda\right]},\theta^2X^{r-2p-2}\biggl]-\biggl(\frac{a_p^2}{p^4}-\frac{r-2p-1}{2p}\biggl)\sum_{\lambda \in \mathbb{F}_p}\biggl[g^0_{1,\left[\lambda \right]},\theta^2 X^{r-2p-2}\biggl]\\
     && \qquad +\frac{a_p}{p^2}\biggl[\mathrm{Id}, \theta^2X^{r-2p-2}\biggl] \pmod{X_{r-1}+V_r^{(3)}}.
          \end{eqnarray*}
          Hence, we have $(T-a_p)f$ is integral and $\overline{(T-a_p)f}$
          has image  $\overline{\frac{a_p}{p^2}}(T^2-cT+1)\left[\mathrm{Id}, X^{p-2}\right] \in \mathrm{ind}_{KZ}^G J_4$, where $c=\overline{\frac{p^2}{a_p}\biggl(\frac{a_p^2}{p^4}-\frac{r-2p-1}{2p}\biggl)} \in \overline{\mathbb{F}}_p$.
\end{proof}
\begin{prop}\label{Gh Ma 5}
Let $3p+2 \le r \equiv 3 \pmod{p-1}$. Assume that $v_p(a_p)=2$ and $p \mid r-2$. Then $F_3$ is a quotient of $\pi(p-2,0,\omega^2)$.
\end{prop}
\begin{proof}
     Let us consider $f_0,f_1\in \mathrm{ind}_{KZ}^G \mathrm{Sym}^r \bar{\mathbb{Q}}_p^2$ defined by \begin{eqnarray*}
         f_1 &=& \sum_{\lambda \in \mathbb{F}_p} \biggl[g^0_{1,\left[\lambda\right]},\frac{XY^{r-1}-X^{r-p-1}Y^{p+1}}{p^2}\biggl]\\
         f_0 &=&  \biggl[\mathrm{Id}, \frac{p-1}{p \cdot a_p} \sum_{\substack{
2<j<r-1 \\
j \equiv 2 \pmod{p-1}
}}
 \alpha_j \cdot X^{r-j} Y^{j} \biggl],
     \end{eqnarray*}
      where $\alpha_j$ are integers as in Lemma~\ref{lem:2.7} (applied with $r-1$ instead of $r$). Clearly $T^+f_1$, $T^+f_0$ and $T^-f_0$ die mod $p$ and mod $X_{r-1}$. We have \begin{eqnarray*}
         T^-f_1 &=& \biggl[ \mathrm{Id}, \frac{p-1}{p}\sum_{\substack{
2\le j<r-1 \\
j \equiv 2 \pmod{p-1}
}}
\binom{r-1}{j} X^{r-j} Y^{j}  \biggl].
\end{eqnarray*}
Since $p \mid r-2$, the $j=2$ term is integral and may be dropped by Corollary~\ref{lem: required for p | r-2} modulo $X_{r-1}+V_r^{(3)}$, hence, we have
\begin{eqnarray*}
T^-f_1-a_pf_0 &\equiv& \biggl[ \mathrm{Id},\frac{p-1}{p} \sum_{\substack{
2 < j<r-1 \\
j \equiv 2 \pmod{p-1}
}}
\biggl(\binom{r-1}{j}-\alpha_j\biggl) X^{r-j} Y^{j}\biggl]\\
&\equiv&  \biggl[ \mathrm{Id},\frac{p-1}{p} \sum_{\substack{
2 < j<r-1 \\
j \equiv 2 \pmod{p-1}
}}
\biggl(\binom{r-1}{j}-\alpha_j\biggl) (X^{r-j} Y^{j}-X^pY^{r-p})\\
&& \qquad +\frac{p-1}{p} \sum_{\substack{
2 < j<r-1 \\
j \equiv 2 \pmod{p-1}
}}
\biggl(\binom{r-1}{j}-\alpha_j\biggl) X^{p} Y^{r-p}\biggl],
     \end{eqnarray*}
    As usual by \cite[Lemma 2.12]{GV22}, we have $X^{r-j}Y^j-X^pY^{r-p} \equiv (j-r)\theta X^{r-p-2}Y \pmod{V_r^{(2)}}$ and by the congruences \eqref{lem: needed for a=3 cl 1}, \eqref{lem: needed for a=3 cl 2} for $r-1$ instead of $r$ the first sum vanishes. So we obtain\begin{eqnarray*}
         T^-f_1-a_pf_0 &\equiv & \biggl[ \mathrm{Id},\frac{p-1}{p} \sum_{\substack{
2 < j<r-1 \\
j \equiv 2 \pmod{p-1}
}}
\biggl(\binom{r-1}{j}-\alpha_j\biggl) X^{p} Y^{r-p} \biggl] \pmod{X_{r-1}+V_r^{(2)}}.
     \end{eqnarray*}
     %which is integral and dies when mapped to $\mathrm{ind}_{KZ}^G J_3$ since $X^pY^{r-p}= XY^{r-1}+\theta Y^{r-p-1}$.
    Since $X^pY^{r-p} = XY^{r-1}+\theta Y^{r-p-1}
    $, by \cite[Lemma 2.12]{GV22} and \eqref{lem: needed for a=3 cl 2}, we have \begin{eqnarray*}
     (T-a_p)(f_0+f_1)&\equiv& -a_p f_1+T^-f_1-a_pf_0 \\
     &\equiv&  -\frac{a_p}{p^2}\sum_{\lambda \in \mathbb{F}_p}\biggl[g^0_{1,\left[\lambda\right]},\theta X^{r-p-2}Y-\theta Y^{r-p-1}\biggl] +\frac{r-p-2}{2p}\biggl[\mathrm{Id}, \theta Y^{r-p-1}\biggl]
\pmod {X_{r-1}+V_r^{(2)}}.
     \end{eqnarray*}
     So $(T-a_p)(f_0+f_1)$ is integral and $\overline{(T-a_p)(f_0+f_1)}$ maps to $-\overline{\frac{a_p}{p^2}}T\biggl[\mathrm{Id}, X^{p-2}\biggl]$ in $\mathrm{ind}_{KZ}^G J_3$.
\end{proof}
\begin{prop}\label{Gh Ma 6}
Let $3p+2 \le r \equiv 3 \pmod{p-1}$. Assume that $v_p(a_p)=2$ and $p \nmid r-3$. 
    \begin{enumerate}
        \item [(i)] If $p\mid \binom{r}{3}$, then $F_1=0$.
        \item [(ii)] If $p\nmid \binom{r}{3}$, then $F_1$ is a quotient of  $\pi(p-4,{\bar{\lambda}},\omega^3)$, where $\lambda \coloneqq \frac{r-3}{3 \binom{r}{3}} \cdot \frac{a_p}{p^2}$. 
    \end{enumerate}
\end{prop}
\begin{proof}
 Let us consider $f_0,f_1\in \mathrm{ind}_{KZ}^G \mathrm{Sym}^r \bar{\mathbb{Q}}_p^2$ defined by \begin{eqnarray*}
     f_1 &=& \sum_{\lambda \in \mathbb{F}_p}\biggl[g^0_{1,\left[\lambda\right]}, \frac{Y^r-X^{r-3}Y^3}{p}\biggl]\\
     f_0 &=& \biggl[\mathrm{Id}, \frac{p-1}{p \cdot a_p}\sum_{\substack{
3\le j<r \\
j \equiv 3 \pmod{p-1}
}}
\alpha_j \cdot X^{r-j} Y^{j} \biggl],
 \end{eqnarray*}    
 where $\alpha_j$ are integers as in Lemma~\ref{lem:2.2}. Clearly $T^+f_1$, $-a_pf_1$ and $T^-f_0$ die mod $p$ and mod $X_{r-1}$.
 We have \begin{eqnarray*}
     T^-f_1 &=& \biggl[\mathrm{Id}, \frac{p-1}{p}\sum_{\substack{
3\le j<r \\
j \equiv 3 \pmod{p-1}
}}
\binom{r}{j} X^{r-j} Y^{j} \biggl], \text{ so}\\
T^-f_1-a_pf_0 &=& \biggl[\mathrm{Id}, \frac{p-1}{p}\sum_{\substack{
3\le j<r \\
j \equiv 3 \pmod{p-1}
}}
\biggl(\binom{r}{j}-\alpha_j\biggl) X^{r-j} Y^{j} \biggl]\\
&=& \biggl[\mathrm{Id}, \frac{r-3}{3}X^{r-3}Y^3 \biggl]
 \end{eqnarray*}
 since $X^{r-j}Y^j \equiv X^{r-3}Y^3 \pmod{V_r^{(1)}}$ and where we have used Lemma~\ref{lem:2.1}.
 Moreover, by Lemma~\ref{lem:2.2} part $(3)$,  we have \[
 T^+f_0 = \sum_{\lambda \in \mathbb{F}_p}\biggl[g^0_{1,\left[\lambda\right]}, \frac{-p^2}{a_p} \binom{r}{3} X^{r-3}Y^3 \biggl].
 \]
 Combining all these together, we have \[
 (T-a_p)(f_0+f_1)=  \biggl[\mathrm{Id}, \frac{r-3}{3}X^{r-3}Y^3 \biggl]-\frac{p^2}{a_p} \binom{r}{3}\sum_{\lambda \in \mathbb{F}_p}\biggl[g^0_{1,\left[\lambda\right]}, X^{r-3}Y^3 \biggl]. \]
 If $p\mid \binom{r}{3}$, then $(T-a_p)(f_0+f_1)$ is integral and $\overline{(T-a_p)(f_0+f_1)}$ maps to $\overline{\frac{r-3}{3}}\biggl[\mathrm{Id}, X^{p-4}\biggl]$ in $\mathrm{ind}_{KZ}^G J_1$. If $p \nmid \binom{r}{3}$, then $(T-a_p)(f_0+f_1)$ is integral and $\overline{(T-a_p)(f_0+f_1)}$ maps to $-\overline{{\frac{p^2}{a_p}}\binom{r}{3}}(T-\bar{\lambda})\biggl[\mathrm{Id}, X^{p-4}\biggl]$ in $\mathrm{ind}_{KZ}^G J_1$.
 \end{proof}
 \begin{prop}\label{Gh Ma 7}
      Let $2p+1 \le r \equiv 3 \pmod{p-1}$ and $p \nmid r-3$. Then $F_2=0$.
 \end{prop}
 \begin{proof}
     By \cite[Lemma 4.18]{GV22}, we have $\frac{X_{r-1}^{*}}{X_{r-1}^{**}} \cong V_1 \otimes D$, noting that $\mathcal{I}(3,1)=\{3\}$.
 \end{proof}
\begin{prop}\label{Gh Ma 8}
    Let $3p+2 \le r \equiv 3 \pmod{p-1}$. Assume that $v_p(a_p)=2$ and $p \nmid \binom{r}{4}$. Then $F_5$ is a quotient of $\pi(1,\bar{\lambda}^{-1},\omega)$, where $\lambda \coloneqq \frac{r-3}{3 \binom{r}{3}} \cdot \frac{a_p}{p^2}$.
    \end{prop}
\begin{proof}
    Let us consider $f_1\in \mathrm{ind}_{KZ}^G \mathrm{Sym}^r \bar{\mathbb{Q}}_p^2$ defined by \[
    f_1 = \sum_{\lambda \in {\mathbb{F}_p}} \biggl[g^0_{1,\left[{\lambda}\right]}, \frac{(\theta^2 X^{p-1}Y^{r-3p-1}-\theta^2Y^{r-2p-2})}{p^2}\biggl].
    \]
    Then clearly $T^+f_1$ dies mod $p$. We compute \begin{eqnarray*}
        T^-f_1 &=& \sum_{\lambda \in \mathbb{F}_p} \biggl[g^0_{1,\left[{\lambda}\right]} \alpha, -X^2Y^{r-2} \biggl]\\
        &=& \sum_{\lambda \in \mathbb{F}_p}\biggl[\mathrm{Id},-X^2(\left[\lambda \right]X+Y)^{r-2}\biggl]\\
&=& \biggl[\mathrm{Id}, \sum_{\substack{
0<j<r-2 \\
j \equiv 1 \pmod{p-1}
}}
 \binom{r-2}{j} X^{r-j} Y^{j}\biggl]\\
 &\equiv& \biggl[\mathrm{Id}, -(r-2)X^{r-p}Y^{p}\biggl] \pmod{X_{r-1}+V_r^{(3)}},
    \end{eqnarray*}
    as  \begin{eqnarray*}
        \sum_{\substack{
1<j<r-2 \\
j \equiv 1 \pmod{p-1}
}}
 \binom{r-2}{j} X^{r-j} Y^{j}+(r-2)X^{r-p}Y^p \in V_r^{(3)},
    \end{eqnarray*}
    by \cite[Lemma 2.7, Lemma 2.14]{GV22}.
Now, let $a$, $b$ and $c$ be mod $p$ integers satisfying
\begin{align*}
3-r-\binom{r}{2}+a+b+c &= 0\\
2r(2-r)+ra+(r+1)b+(r+2)c &= 0\\
-r\binom{r-1}{2}
+\binom{r}{2}a
+\binom{r+1}{2}b
+\binom{r+2}{2}c &= 0.
\end{align*}
Then by \eqref{important GM 1}, \eqref{important GM 2} and \eqref{important GM 3} we have 
\begin{eqnarray*}
\sum_{\substack{
2<j<r-1\\
j\equiv 2 \!\!\!\pmod{p-1}
}}
\binom{r}{j}X^{r-j}Y^j+aX^pY^{r-p}
+bX^{2p-1}Y^{r-2p+1}
+cX^{3p-2}Y^{r-3p+2} &\in& V_r^{(3)}
\end{eqnarray*}
(note that the polynomial is divisible by $X^3$ and $Y^3$). Thus by \eqref{useful for last one}, we have 
\begin{eqnarray*}
J(X,Y) \coloneqq -\binom{r}{2}X^{r-2}Y^2
+aX^pY^{r-p}
+bX^{2p-1}Y^{r-2p+1}
+cX^{3p-2}Y^{r-3p+2}
&\in& V_r^{(3)}+X_{r-1},
\end{eqnarray*}
 Similarly, the polynomial
\[
F(X,Y)
\coloneqq
-\binom{r}{2}X^{r-2}Y^2
+(a+3-r)X^pY^{r-p}
+bX^{2p-1}Y^{r-2p+1}
+cX^{3p-2}Y^{r-3p+2} \>  \in \>  V_r^{(2)}. \]
By \cite[Lemma 2.12]{GV22}, we obtain
\begin{align*}
F(X,Y)
&\equiv
\biggl(
(a+3-r)\binom{r}{2}
+b\binom{r+1}{2}
+c\binom{r+2}{2}
\biggr)
\theta^2X^{r-3p-1}Y^{p-1}-\binom{r}{2}\theta^2 X^{r-2p-2}
\pmod{V_r^{(3)}}\\
&=
\binom{r}{2}(\theta^2X^{r-3p-1}Y^{p-1}-\theta^2 X^{r-2p-2})
\pmod{V_r^{(3)}},
\end{align*}
by the third relation satisfied by $a$, $b$, and $c$.
Subracting $J(X,Y)$ from $F(X,Y)$, we obtain
\[
(3-r)X^pY^{r-p}
\equiv
\binom{r}{2}(\theta^2X^{r-3p-1}Y^{p-1}-\theta^2 X^{r-2p-2})
\pmod{V_r^{(3)}+X_{r-1}}.
\]
Applying the Weyl involution 
$\left(\begin{smallmatrix}
0 & 1\\
1 & 0
\end{smallmatrix} \right)$,
we obtain
\[
(3-r)X^{r-p}Y^p
\equiv
\binom{r}{2}(\theta^2X^{p-1}Y^{r-3p-1}-\theta^2 Y^{r-2p-2})
\pmod{V_r^{(3)}+X_{r-1}}.
\]
Since $r\not\equiv 3 \pmod p$, %we  have %may divide by $3-r$ to deduce that
%\[
%(r-2)X^{r-p}Y^p
%\equiv
%\frac{\binom{r}{2}(r-2)}{3-r}
%(\theta^2X^{p-1}Y^{r-3p-1}-\theta^2 Y^{r-2p-2})
%\pmod{V_r^{(3)}+X_{r-1}}.
%\]
%Using the identity
%\[
%(r-2)\binom{r}{2}=3\binom{r}{3},
%\]
we have
\[
-(r-2)X^{r-p}Y^p
\equiv
\frac{3\binom{r}{3}}{r-3}
(\theta^2X^{p-1}Y^{r-3p-1}-\theta^2Y^{r-2p-2})
\pmod{V_r^{(3)}+X_{r-1}}.
\]
   Combining these, we have \begin{eqnarray*}
   T^-f_1-a_pf_1 &\equiv& -\frac{a_p}{p^2}\sum_{\lambda \in {\mathbb{F}_p}} \biggl[g^0_{1,\left[{\lambda}\right]}, {(\theta^2 X^{p-1}Y^{r-3p-1}-\theta^2Y^{r-2p-2})}\biggl]\\
   && \qquad +\biggl[\mathrm{Id},\frac{3\binom{r}{3}}{r-3}(\theta^2 X^{p-1}Y^{r-3p-1}-\theta^2Y^{r-2p-2}) \biggl]\pmod{V_r^{(3)}+X_{r-1}}.
   \end{eqnarray*}  
   Hence, $(T-a_p)f_1$ is integral and $\overline{(T-a_p)f_1}$ maps to $-\overline{\frac{a_p}{p^2}}(T-{\bar{\lambda}}^{-1})\left[\mathrm{Id}, X\right]$ in $\mathrm{ind}_{KZ}^G J_5$,
by (applying the Weyl involution to) Lemma~\ref{lem:2.10}.
\end{proof}
\begin{prop}\label{Gh Ma 9}
 Let $3p+2 \le r \equiv 3 \pmod{p-1}$. Assume that $v_p(a_p)=2$ and $p \mid r-2$. Then \begin{itemize}
     \item [(i)] $F_5=0$
     \item[(ii)] $F_4=0$. 
 \end{itemize}
\end{prop}
\begin{proof}
   (i)  Let us consider $f_0\in \mathrm{ind}_{KZ}^G \mathrm{Sym}^r \bar{\mathbb{Q}}_p^2$ defined by \begin{eqnarray*}
         f_0 = \biggl[\mathrm{Id}, \frac{\theta^2X^{r-3p-1}Y^{p-1}-\theta^2X^{r-2p-2}}{p^2}\biggl].
     \end{eqnarray*}
 Clearly $T^-f_0$ dies mod $p$. We compute that \begin{eqnarray*}
     T^+f_0 &=& -\biggl[g^0_{1,0}, X^{r-2}Y^2\biggl] \equiv 0 \pmod{X_{r-1}+V_r^{(3)}},
 \end{eqnarray*}    
by Corollary~\ref{lem: required for p | r-2}.
Hence, we have $(T-a_p)f_0=-\frac{a_p}{p^2}\biggl[\mathrm{Id}, \theta^2X^{r-3p-1}Y^{p-1}-\theta^2X^{r-2p-2}\biggl]$ is integral and maps to a generator
$\overline{\frac{a_p}{p^2}} \left[\mathrm{Id}, Y\right] \in \mathrm{ind}_{KZ}^G J_5$.

  (ii)   Let us consider $f_0\in \mathrm{ind}_{KZ}^G \mathrm{Sym}^r \bar{\mathbb{Q}}_p^2$ defined by \[
     f_0 = \biggl[\mathrm{Id}, \frac{\theta^2X^{r-2p-2}}{p^2}\biggl].
     \]
     Clearly $T^-f_0$ dies mod $p$. We compute \[
     T^+f_0= \sum_{\lambda \in \mathbb{F}_p}\biggl[g^0_{1,\left[\lambda \right]}, X^{r-2}Y^2\biggl]\equiv 0 \pmod{X_{r-1}+V_r^{(3)}},
     \]
     by Corollary~\ref{lem: required for p | r-2}.
     Hence, we have $(T-a_p)f_0= -\frac{a_p}{p^2}\biggl[\mathrm{Id}, \theta^2X^{r-2p-2}\biggl]$ is integral and maps to a generator
$-\overline{\frac{a_p}{p^2}} \left[\mathrm{Id}, X^{p-2}\right] \in \mathrm{ind}_{KZ}^G J_4$.
\end{proof}

\vspace{2cm}
{%\color{red}we have to write something

%we have to write something

%we have to write something}

\begin{proof}[Proof of Theorem~\ref{thm:a=3}] 
The above propositions may be summarized by the following table:
\begin{table}[h]
\centering
\begin{tabular}{|>{\centering\arraybackslash}m{5cm}|
                >{\centering\arraybackslash}m{1.3cm}|
                >{\centering\arraybackslash}m{1.3cm}|
                >{\centering\arraybackslash}m{1.3cm}|
                >{\centering\arraybackslash}m{1.3cm}|
                >{\centering\arraybackslash}m{1.3cm}|}
                  
\hline
& $F_5$ & $F_4$ & $F_3$ & $F_2$ & $F_1$ \\ \hline
$p\nmid \binom{r}{4}$
& $\checkmark$& &  &  &$\checkmark$ \\ \hline

$p\mid r$
&  $\checkmark$& & & & \\ \hline

$p\mid r-1$
& &$\checkmark$ & & &\\ \hline

$p\mid r-2$
& &  &$\checkmark$ & & \\ \hline

$p \mid r-3$
& & & &  &$\checkmark$ \\ \hline
\end{tabular}
\end{table}

In particular, we have:
\begin{itemize}
    \item [(1)] If $p \nmid \binom{r}{4}$, then by Proposition~\ref{Gh Ma 1} part $(i)$ and Proposition~\ref{Gh Ma 2} we get $\overline{\Theta}^{\,{ss}} \cong (F_5 \oplus F_1)^{\,{ss}}$. Moreover, Proposition~\ref{Gh Ma 6} part $(ii)$ and Proposition~\ref{Gh Ma 8} tells us that we have surjections $\pi(p-4,\bar{\lambda},\omega^3) \twoheadrightarrow F_1$ and $\pi(1,\bar{\lambda}^{-1},\omega) \twoheadrightarrow F_5$, where $\lambda= \frac{r-3}{3 \binom{r}{3}} \cdot \frac{a_p}{p^2}$. Hence, we have $\overline{\Theta}^{\,{ss}} \cong \pi(p-4,\bar{\lambda},\omega^3)^{\,{ss}} \oplus \pi(1,\bar{\lambda}^{-1},\omega)^{\,{ss}}$.
    \item [(2)] If $p \mid r$, then by Proposition~\ref{Gh Ma 1} part $(i)$, Proposition~\ref{Gh Ma 2} and Proposition~\ref{Gh Ma 6} part $(i)$, we get $\overline{\Theta}^{\,{ss}} \cong F_5 \cong \pi(1,0,\omega)$. 
    \item[(3)]  If $p \mid r-1$, then by Proposition~\ref{Gh Ma 2}, Proposition~\ref{Gh Ma 3}, Proposition~\ref{Gh Ma 6} part $(i)$ and Proposition~\ref{Gh Ma 4} we have $\overline{\Theta}^{\,{ss}}  \cong \pi(p-2,{\lambda},\omega^2)^{\,{ss}}\oplus\pi(p-2,{\lambda}^{-1},\omega^2)^{\,{ss}} $, where $\lambda+{\lambda}^{-1}=  \overline{\frac{a_p^2-p^3(\frac{r-2p-1}{2})}{p^2a_p}}$.
    \item [(4)] If $p \mid r-2$, then by Proposition~\ref{Gh Ma 5}, Proposition~\ref{Gh Ma 6} part $(i)$, Proposition~\ref{Gh Ma 7} and Proposition~\ref{Gh Ma 9}, we get  $\overline{\Theta}^{\,{ss}} \cong F_3 \cong \pi(p-2,0,\omega^2)$. 
    \item [(5)] If $p \mid r-3$, then by Proposition~\ref{Gh Ma 1} part $(ii)$, Proposition~\ref{Gh Ma 2}, we get $\overline{\Theta}^{\,{ss}} \cong F_1 \cong \pi(p-4,0,\omega^3)$.
\end{itemize}
Theorem~\ref{thm:a=3} now follows in all cases from the mod $p$ LLC.
\end{proof}
\section{Examples}\label{examples}
In this section, we give some examples illustrating Theorem~\ref{main theorem} and also discuss in more detail the particular cases of the local mod $7$ and mod $3$ reductions of the Galois representations attached to $\Delta_{26}$ and $\Delta$ mentioned in the introduction.

\subsection{Table}

The following table was populated using the SageMath program of Rozensztajn \cite{Roz18}. 
 For simplicity, we assume $a_p=p^2$ in each 
 case. The table has 18 rows corresponding
 to the 18 possibilities in Theorem~\ref{main theorem}.
\begin{table}[h]
\centering
\caption{Reductions of crystalline representations of slope $2$}
\renewcommand{\arraystretch}{1.5}
\begin{tabular}{|>{\centering\arraybackslash}m{2cm}|
                >{\centering\arraybackslash}m{2cm}|
                >{\centering\arraybackslash}m{2cm}|
                >{\centering\arraybackslash}m{4cm}|}
                \hline
$r$& $p$& $a$ & $\overline{V}^{\,{ss}}_{k,a_p}$ \\ \hline % Row 1

$21$& $5$ & $1$ &$ \omega^2 \oplus  1$\\ \hline % Row 2
$29$& $5$ & $1$ &$\mathrm{ind}\left(\omega_2^{2}\right)$ \\ \hline % Row 3
$25$& $5$ & $1$ &$\mathrm{ind}\left(\omega_2^{2}\right)$\\ \hline % Row 4
$105$& $5$ &$1$ &  \\ \hline % Row 5
$22$& $5$ &$2$ & $\mathrm{ind}\left(\omega_2^{3}\right)$ \\ \hline % Row 6
$30$& $5$ &$2$ & $\mu_{3}\cdot \omega \oplus \mu_{2}\cdot \omega^2 $ \\ \hline % Row 7
$110$& $5$&$2$ &  \\ \hline % Row 8
$26$& $5$ & $2$& $ \omega \oplus \omega^2$\\ \hline % Row 9
$106$ & $5$ & $2$& \\ \hline % Row 10
$19$& $5$ & $3$&$\mu_{3}\cdot \omega \oplus \mu_{2}\cdot \omega^3$ \\ \hline % Row 11
$35$& $5$ &$3$ & $\mathrm{ind}\left(\omega_2^{2}\right) \otimes\, \omega$\\ \hline % Row 12
$31$& $5$ & $3$&$\mu_{\lambda}\cdot \omega^2 \oplus \mu_{\lambda^{-1}}\cdot \omega^2$, $\lambda+{\lambda^{-1}}=4$ \\ \hline % Row 13
 $27$ & $5$ &$3$ & $\mathrm{ind}\left(\omega_2^{4}\right) \otimes\, \omega^2$ \\ \hline % Row 14
$23$& $5$ & $3$& $\mathrm{ind}\left(\omega_2^{2}\right) \otimes\, \omega^3$\\ \hline % Row 15
$23$& $7$ & $5$& $\mu_{2} \cdot \omega^4 \oplus \mu_{4} \cdot \omega^2$ \\ \hline % Row 16
$53$& $7$ & $5$&$\mathrm{ind}\left(\omega_2^{4}\right)\otimes\, \omega^4$\\ \hline % Row 17
$47$& $7$ &$5$ & $\mu_{5}\cdot \omega^5 \oplus \mu_{3}\cdot \omega$\\ \hline % Row 18
$299$& $7$ & $5$&\\ \hline % Row 19
\end{tabular}
\end{table}

\vspace{7cm}

\noindent The final column matches with the results in Theorem~\ref{main theorem} (some entries
are left blank if the program does not terminate in a reasonable amount of time).

\subsection{$\Delta_{26}$ and $p = 7$} \label{example Delta 26}

Recall that $\Delta_{26} \coloneqq \Delta \cdot E_{14}$ is the unique normalized cuspidal newform of level $1$ and weight $k = 26$. Its $7$-th Fourier coefficient  is $a_7=39080597192$, which  has $7$-adic valuation $2$ as $39080597192=7^2 \times 797563208$. Since $\Delta_{26}$ has weight $k=26$ and $r = k-2=24 \ge 3 \cdot 7+2$ and $r \equiv 6 \neq 4 \pmod{6}$, 
Theorem~\ref{main theorem} applies.  By the first part of the last case of the theorem, %Theorem~\ref{main theorem}
the semisimplification of the global mod $7$ Galois representation  $\bar{\rho}_{\Delta_{26}}$ attached to $\Delta_{26}$ restricted to $G_{\mathbb{Q}_p}$ satisfies $${(\overline \rho_{{\Delta_{26}} |_{G_{\mathbb{Q}_p}}  }})^{ss} \cong \overline{V}_{26, a_7}^{\,ss} \cong \omega^2 \oplus \omega^5,$$ because $7 \nmid 18 \cdot 19 $ and $\lambda = \frac{18 \cdot 19 }{5 \cdot 4} \cdot \frac{a_7}{49} \equiv 1 \pmod{7}$. 

On the other hand, 
 %Also we have $E_{14} \equiv E_2 \pmod{7}$ as $n^{13} \equiv n \pmod{7}$. Hence, its $n$-th Fourier coefficient $b_n$ satisfies 
%$$b_n \equiv \sigma_1(n) \pmod{7}.$$ A short computation shows that the convolution $$a_n= \sum_{k=0}^{n} \tau(k)b_{n-k} 
%\equiv  n\sigma_3(n)+4\sum_{k=1}^{n-1}k\sigma_3(k)\sigma_1(n-k) \equiv n^2\sigma_3(n) \pmod{7}.$$
if
\begin{eqnarray*}
    E_2  =  1 -24 \sum_{n=1}^\infty \sigma_1(n) q^n & \text{and} & 
    E_{14} =  1 - 24 \sum_{n=1}^\infty \sigma_{13}(n) q^n
\end{eqnarray*}
and $q \frac{d}{dq}$ is the Ramanujan theta operator, then $q \frac{d}{dq} (\Delta)=\Delta \cdot E_2$ (take  the logarithmic derivative of $\Delta= q\prod_{n=1}^{\infty}(1-q^n)^{24}$ and multiply both sides by $q \cdot \Delta$), but since $E_{14} \equiv E_{2} \pmod{7}$ as $n^{13} \equiv n \pmod{7}$, we have $\Delta_{26} = \Delta \cdot E_{14} \equiv q \frac{d}{dq}(\Delta)\pmod{7}$,  so the $n$-th Fourier coefficient $a_n$ of $\Delta_{26}$ satisfies $$a_n \equiv n \cdot \tau(n) \equiv n^2\sigma_3(n)\pmod{7},$$
by a famous congruence of  Ramanujan which says that $\tau(n) \equiv n \sigma_3(n) \pmod{7}$.
Thus, $a_{\ell} \equiv \ell^2 +\ell^5 \pmod{7}$ for all prime $\ell$. Hence,  $\bar{\rho}_{\Delta_{26}}$  has semisimplification 
$$\bar\rho^{\,ss}_{\Delta_{26}}  \cong \omega^5 \oplus \omega^2,$$ so certainly its
restriction to $G_{\mathbb{Q}_p}$ has the same shape, corroborating 
the conclusion of Theorem~\ref{main theorem} in this case.

\subsection{ $\Delta$ and $p=3$}

In this section, we will give a separate treatment of the case of $\Delta$ and $p = 3$ using the techniques of this paper, because $\Delta$ is the first cusp form and so is of special interest.
Note that the slope is $2$ as $\tau(3)=252 = 9 \times 28$. 
However, note that  the prime $p =3$ (and $p = 2$) is generally avoided in this paper.
Moreover $r=10 \not\geq 3p+2$ so Theorem~\ref{main theorem} again does not apply. 
%Note that $r \equiv 2 \pmod{2}$ and  $r=10 \equiv 1 \pmod{9}$.

The following result has been known for some time (see \cite{Ser73}, it follows from the congruence $\tau(n) \equiv n \sigma_1(n) \pmod{3}$), but we give an alternative proof using the methods of this paper.

\begin{prop} We have 
    $ (\bar \rho_{{\Delta}}|_{G_{\mathbb{Q}_3}})^{ss} \cong 1 \oplus \omega$.
\end{prop}
\begin{proof}
Throughout the proof $r=10$ and $p=3$.
    Let us consider the functions $f_1\in \mathrm{ind}_{KZ}^G \mathrm{Sym}^r \bar{\mathbb{Q}}_p^2$ defined by \begin{eqnarray*}
        f_1 &=& \sum_{\lambda \in \mathbb{F}_p} \biggl[g^0_{1, \left[ \lambda \right]}, \frac{Y^r-X^{r-2}Y^2}{p}\biggl]\\
        f_{-1} &=& (1-p)\biggl[\alpha, \frac{X^r-X^{p+1}Y^{r-p-1}}{p}\biggl].
     \end{eqnarray*}
      Clearly $T^+f_1$, $-a_pf_1$, $T^-f_{-1}$ and $-a_pf_{-1}$ die mod $p $ and mod $X_{r-1}$.  We compute that \begin{eqnarray*}
         T^+f_{-1}=(1-p
         )\biggl[\mathrm{Id},\frac{X^r}{p}\biggl].
     \end{eqnarray*}
     Hence, we have 
     \begin{eqnarray*}
         T^-f_1 +T^+f_{-1}&\equiv & \biggl[\mathrm{Id}, \frac{p-1}{p}\sum_{\substack{
2\leq j<r \\
j \equiv 2 \pmod{p-1}
}}
 \binom{r}{j} X^{r-j} Y^{j}  \biggl]\pmod{X_{r}}\\
 %&=& \biggl[\mathrm{Id}, \frac{p-1}{p} \binom{r}{2} X^{r-2}Y^2+ \frac{p-1}{p}\sum_{\substack{
%2< j<r \\
%j \equiv 2 \pmod{p-1}
%}}
 %\binom{r}{j} X^{r-j} Y^{j}  \biggl]\\
 &\equiv& \biggl[\mathrm{Id}, \frac{2}{3}\biggl(\binom{10}{4}X^{6}Y^4+\binom{10}{6}X^{4}Y^6\biggl)\biggl]\\
&=& \biggl[\mathrm{Id},X^{8}Y^2 \biggl],
 \end{eqnarray*} 
 as $X^{r-j}Y^j \equiv X^{r-2}Y^2 \pmod{V_r^{(1)}}$ and $2\cdot \frac{2}{3} \cdot \binom{10}{4} \equiv 1 \pmod{3}$.
Hence, $F_1$ dies.

Moreover, $F_3$ dies by Proposition~\ref{lem:G_6} part (ii).

Let us consider $f_1\in \mathrm{ind}_{KZ}^G \mathrm{Sym}^r \bar{\mathbb{Q}}_p^2$ defined by \begin{eqnarray*}
  f_1 &=& \sum_{\lambda\in \mathbb{F}_p} \biggl[g^0_{1,\left[\lambda\right]}, \frac{XY^{r-1}-X^{r-p}Y^p}{p^2}\biggl].
         %f_0 &=& \biggl[\mathrm{Id}, \frac{p-1}{p \cdot a_p} \sum_{\substack{
%1<j<r-1 \\
%j \equiv 1 \pmod{p-1}
%}}
 %\alpha_j \cdot X^{r-j} Y^{j}\biggl],
         \end{eqnarray*}
         %where $\alpha_j$ are integers as in Lemma~\ref{ramanujan delta r-1}. Now it is clear that $T^+f_1$, $T^+f_0$ and $T^-f_0$ die mod $p$. 
         Clearly $T^+f_1$ dies mod $p$.
         We compute 
         \begin{eqnarray*}
             T^-f_1 &=& \biggl[\mathrm{Id}, \frac{p-1}{p}\sum_{\substack{
1 \le j<r-1 \\
j \equiv 1 \pmod{p-1}
}}
 \binom{r-1}{j} X^{r-j} Y^{j}\biggl] \pmod{X_{r-1}}\\
 &\equiv& \biggl[\mathrm{Id}, \frac{2}{3} \biggl(\binom{9}{3} X^{7}Y^3 %+\binom{9}{5}X^{5}Y^5+ \binom{9}{7}X^{3}Y^7\biggl)
 \biggl]\\
 %T^-f_1-a_pf_0 &=& \biggl[\mathrm{Id}, \frac{p-1}{p}\sum_{\substack{
%1<j<r-1 \\
%j \equiv 1 \pmod{p-1}
%}}
 %\biggl(\binom{r-1}{j}-\alpha_j\biggl) X^{r-j} Y^{j}\biggl].
 %\end{eqnarray*}
 %Now as before $X^{r-j}Y^j-X^{r-p}Y^p \equiv j \theta X^{r-2p}Y^{p-1} \pmod{V_r^{(2)}}$. So by Lemma~\ref{lem:2.1} applied twice, we obtain 
 %\begin{eqnarray*}
 %&=& \biggl[\mathrm{Id}, \frac{p-1}{p}\sum_{\substack{
%1<j<r-1 \\
%j \equiv 1 \pmod{p-1}
%}}
 %\biggl(\binom{r-1}{j}-\alpha_j\biggl) (X^{r-j} Y^{j}-X^{r-p}Y^p)\\
 %&& \qquad +\frac{p-1}{p}\sum_{\substack{
%1<j<r-1 \\
%j \equiv 1 \pmod{p-1}
%}}
 %\biggl(\binom{r-1}{j}-\alpha_j\biggl)X^{r-p}Y^p\biggl]\\
 &\equiv& \biggl[\mathrm{Id}, -X^{7} Y^{3}\biggl].
 \end{eqnarray*}
 Hence, we have \begin{eqnarray*}
     (T-a_p)f_1 &=& -\frac{a_p}{p^2}\sum_{\lambda\in \mathbb{F}_p} \biggl[g^0_{1,\left[\lambda\right]}, \theta X^{r-p-1}\biggl]+\biggl[\mathrm{Id}, \theta X^{r-p-1}\biggl] \pmod{X_{r-1}}.
  \end{eqnarray*}
  We note that $\frac{\tau(3)}{9} \equiv 1 \pmod{3}$.
  Hence, $(T-a_p)f_1$ is integral and maps to $- (T-1)\biggl[\mathrm{Id},X^{p-1}\biggl]$ in \(\mathrm{ind}_{KZ}^G J_2\).

Let us consider $f_0, f_0', f_2\in \mathrm{ind}_{KZ}^G \mathrm{Sym}^r \bar{\mathbb{Q}}_p^2$ defined by \begin{eqnarray*}
        f_2 &=& \sum_{{\lambda \in \mathbb{F}_p},\,\mu \in {\mathbb{F}_p^\times}} \biggl[g^0_{2,\left [\lambda \right]+p\left[\mu \right]}, \frac{Y^r-X^{r-p-1}Y^{p+1}}{p}\biggl]+(1-p)\sum_{\lambda \in \mathbb{F}_p}\biggl[g^0_{2,\left [\lambda \right]}, \frac{Y^r-X^{r-p-1}Y^{p+1}}{p}\biggl]\\
         %f_1 &=& \sum_{\lambda \in {\mathbb{F}_p}} \biggl[g^0_{1,\left[\lambda \right]},\frac{p-1}{p \cdot a_p} \sum_{\substack{
         %2<j<r \\
         %j \equiv 2 \pmod{p-1}
         %}}
         %\alpha_j \cdot X^{r-j} %Y^{j}\biggl]\\
         f_0' &=& \sum_{\lambda \in {\mathbb{F}_p}} \biggl[g^0_{1,\left[\lambda \right]}\alpha,\frac{X^r-X^{p+1}Y^{r-p-1}}{p}\biggl]\\
        f_0 &=& \biggl[\mathrm{Id},\frac{\theta^2 X^{r-2p-2}}{2p^2}\biggl].
    \end{eqnarray*}
      Clearly $T^+f_2$, $-a_pf_2$, $T^-f_0$, $-a_pf_0'$ and $T^-f_0'$ die mod $p$. We compute \begin{eqnarray*}
          T^+f_0' &=& \sum_{\lambda \in {\mathbb{F}_p}} \biggl[g^0_{1,\left[\lambda \right]}, \frac{X^r}{p} \biggl].
      \end{eqnarray*} We have 
\begin{eqnarray*}
T^-f_2+T^+f_0'&=& \sum_{\lambda \in {\mathbb{F}_p}} \biggl[g^0_{1,\left[\lambda \right]},\frac{p-1}{p} \sum_{\substack{
0<j<r \\
j \equiv 2 \pmod{p-1}
}}
 \binom{r}{j} X^{r-j} Y^{j}\biggl]\\
 &=& \sum_{\lambda \in {\mathbb{F}_3}} \biggl[g^0_{1,\left[\lambda \right]},-X^{4}Y^{6}+2X^{6}Y^{4} \biggl].
\end{eqnarray*}
We have \[
T^+f_0 = \frac{1}{2}\sum_{\lambda \in {\mathbb{F}_p}} \biggl[g^0_{1,\left[\lambda \right]}, X^{r-2}Y^2\biggl]=-\sum_{\lambda \in {\mathbb{F}_3}} \biggl[g^0_{1,\left[\lambda \right]}, X^{8}Y^2\biggl].
\]
Hence, we get \begin{eqnarray*}
    T^-f_2+T^+f_0'+T^+f_0 &=& \sum_{\lambda \in {\mathbb{F}_3}} \biggl[g^0_{1,\left[\lambda \right]},-\theta^2X^{2} \biggl].
\end{eqnarray*}
Combining these we have \begin{eqnarray*}
    (T-a_p)(f_0+f_0'+f_2)=-\sum_{\lambda \in {\mathbb{F}_p}} \biggl[g^0_{1,\left[\lambda \right]},\theta^2X^{r-2p-2} \biggl]-\frac{a_p}{2p^2}\biggl[\mathrm{Id},{\theta^2 X^{r-2p-2}}\biggl].
\end{eqnarray*}
Hence, $(T-a_p)(f_0+f_0'+f_2)$ is integral and maps to $- (T-1)\biggl[\mathrm{Id},X^{p-1}\biggl]$ in \(\mathrm{ind}_{KZ}^G J_4\).

We note that $V_r^{(3)}=0$ and $V_r^{(2)} \cong V_{p-1}$. 

Combining all these together we have surjections $\pi(2,1,\omega) \twoheadrightarrow F_2$ and $\pi(2,1,1) \twoheadrightarrow F_4$, hence, we have $\overline{\Theta}^{\,{ss}} \cong \pi(2,1,\omega)^{\,ss} \oplus \pi(2,1,1)^{\,ss}$. It is well known that $\pi(p-1,\lambda,\eta)^{\,ss} \cong  \pi(0,\lambda,\eta)^{\,ss}$. Hence, by applying the  mod $p$ LLC, we have $ (\bar \rho_{{\Delta}}|_{G_{\mathbb{Q}_3}})^{ss} \cong 1 \oplus \omega$. 
\end{proof}

\begin{rmk}
    The conclusion of the proposition 
    matches with the fourth part of the $a = 2$ case of Theorem~\ref{main theorem} from the introduction, showing that the theorem is true in this particular case even though $p = 3$.
\end{rmk}

{\noindent \bf Acknowledgements: }
    The first author acknowledges the support received as the Visiting Libra Professor at UMaine during Fall 2026.
    The second author would like to thank his advisor (the first author) for suggesting this problem and for his guidance, as well as the School of Mathematics, TIFR for its support. Both authors acknowledge the support of project 1303/9/2025-R\&D-II-DAE/TIFR-17312. No AI tools were used at any stage in this paper.

\vspace{1cm}

\noindent
\textsc{School of Mathematics, Tata Institute of Fundamental Research, Homi Bhabha Road, Mumbai 400005, India.}\\ $\>$ \\
\textit{Email:} \texttt{eghate@math.tifr.res.in}\\
\textit{Email:} \texttt{sudipta@math.tifr.res.in}


\medskip






\begin{thebibliography}{[BGR18]}

\bibitem[BL94]{BL94}
L. Barthel and R. Livn\'e, 
\textit{Irreducible modular representations of ${\rm GL}_2$ of a local field}, 
Duke Math. J. \textbf{75} (1994), no.~2, 261--292.



\bibitem[BL95]{BL95}
L. Barthel and R. Livn\'e, 
\textit{Modular representations of ${\rm GL}_2$ of a local field: the ordinary, unramified case}, 
J. Number Theory \textbf{55} (1995), no.~1, 1--27.

\bibitem[Ber10]{Ber10}
L. Berger, 
\textit{Représentations modulaires de ${\rm GL}_2({\bf Q}_p)$ et représentations galoisiennes de dimension 2}, 
Astérisque (2010), no.~330, 263--279.


\bibitem[Ber11]{Ber11}
L. Berger, 
\textit{La correspondance de Langlands locale $p$-adique pour ${\rm GL}_2({\bf Q}_p)$}, 
Ast\'erisque (2011), no.~339, S\'eminaire Bourbaki, vol. 2009/2010, Exp. No. 1017, 157--180.

\bibitem[BL22]{BL22}
J. Bergdall and B. Levin, 
\textit{Reductions of some two-dimensional crystalline representations via Kisin modules}, 
Int. Math. Res. Not. IMRN (2022), no.~4, 3170--3197.

\bibitem[BLZ04]{BLZ04}
L. Berger, H. Li and H. J. Zhu, 
\textit{Construction of some families of 2-dimensional crystalline representations}, 
Math. Ann. \textbf{329} (2004), no.~2, 365--377.



\bibitem[BG15]{BG15}
S. Bhattacharya and E. Ghate, 
\textit{Reductions of Galois representations for slopes in $(1,2)$}, 
Doc. Math. \textbf{20} (2015), 943--987.

\bibitem[BGR18]{BGR18}
S. Bhattacharya, E. Ghate and S. Rozensztajn, 
\textit{Reductions of Galois representations of slope 1}, 
J. Algebra \textbf{508} (2018), 98--156.

\bibitem[Bre03a]{Bre03a}
C. Breuil, 
\textit{Sur quelques repr\'esentations modulaires et $p$-adiques de ${\rm GL}_2(\mathbf{Q}_p)$. I}, 
Compositio Math. \textbf{138} (2003), no.~2, 165--188.

\bibitem[Bre03b]{Bre03b}
C. Breuil, 
\textit{Sur quelques repr\'esentations modulaires et $p$-adiques de ${\rm GL}_2(\mathbf{Q}_p)$. II}, 
J. Inst. Math. Jussieu \textbf{2} (2003), no.~1, 23--58.

\bibitem[BG09]{BG09}
K. Buzzard and T. Gee, 
\textit{Explicit reduction modulo $p$ of certain two-dimensional crystalline representations}, 
Int. Math. Res. Not. IMRN (2009), no.~12, 2303--2317.

\bibitem[BG13]{BG13}
K. Buzzard and T. Gee, 
\textit{Explicit reduction modulo $p$ of certain 2-dimensional crystalline representations, II}, 
Bull. Lond. Math. Soc. \textbf{45} (2013), no.~4, 779--788.



\bibitem[Edi92]{Edi92}
B. Edixhoven, 
\textit{The weight in Serre's conjectures on modular forms}, 
Invent. Math. \textbf{109} (1992), no.~3, 563--594.



\bibitem[GG15]{GG15}
A. Ganguli and E. Ghate, 
\textit{Reductions of Galois representations via the ${\rm mod}\,p$ local Langlands correspondence}, 
J. Number Theory \textbf{147} (2015), 250--286.

\bibitem[GV22]{GV22}
E. Ghate and R. Vangala, 
\textit{The monomial lattice in modular symmetric power representations}, 
Algebr. Represent. Theory \textbf{25} (2022), no.~1, 121--185.




\bibitem[GR25]{GR25}
E. Ghate and V. Rai, 
\textit{Reductions of Galois representations of slope $\frac{3}{2}$}, 
Kyoto J. Math. \textbf{65} (2025), no.~3, 595--636.




\bibitem[Glo78]{Glo78}
D. J. Glover, 
\textit{A study of certain modular representations}, 
J. Algebra \textbf{51} (1978), no.~2, 425--475.


\bibitem[NP25]{NP25}
E. Nagel and A. Pande, 
\textit{Reductions of modular Galois representations of slope (2,3)}, 
Ramanujan J. \textbf{67} (2025), no.~3, Paper No. 70, 62 pp.

\bibitem[Roz18]{Roz18}
S. Rozensztajn, 
\textit{An algorithm for computing the reduction of 2-dimensional crystalline representations of ${\rm Gal}(\overline{\mathbf{Q}}_p/\mathbf{Q}_p)$}, 
Int. J. Number Theory \textbf{14} (2018), no.~7, 1857--1894.

\bibitem[Ser73]{Ser73}
J-P. Serre, 
\textit{Congruences et formes modulaires}, 
S\'eminaire Bourbaki. Vol. 1971/1972, Expos\'es 400--417, Exp. No. 416, Lecture Notes in Math., vol. 317, Springer, Berlin, 1973, pp. 319--338.

\end{thebibliography}
\end{document}